\documentclass[a4paper, 11pt]{amsart}
\usepackage[T1]{fontenc}
\usepackage[utf8]{inputenc}
\usepackage{amsmath}
\usepackage{amsfonts}
\usepackage{amssymb}
\usepackage{amsthm}
\usepackage{graphicx}
\usepackage[english]{babel}
\usepackage{version} 
\usepackage{nicefrac}
\usepackage{tipa} 
\usepackage{hyperref}
\usepackage{xcolor}

\theoremstyle{definition}
\newtheorem{defin}{Definition}[section]
\newtheorem{ex}[defin]{Example}
\theoremstyle{plain}
\newtheorem{theo}[defin]{Theorem}
\newtheorem{lemma}[defin]{Lemma}
\newtheorem{obs}[defin]{Remark}
\newtheorem{prop}[defin]{Proposition}
\newtheorem{cor}[defin]{Corollary}

\newtheorem{theorem}{Theorem}
\newtheorem{corollary}[theorem]{Corollary}

\newcommand{\id}{\textup{id}}

\title[A GH-compactification of CAT$(0)$-groups]{A GH-compactification of CAT$(0)$-groups via totally disconnected, unimodular actions}

\author{Nicola Cavallucci}
\thanks{N.Cavallucci is partially supported by the SFB/TRR 191, funded by the DFG}
\address{Nicola Cavallucci, Karlsruhe Institute of Technology, Engelstrasse 2, D-76128 Karlsruhe}
\email{n.cavallucci23@gmail.com}
\date{}

\begin{document}
\maketitle

\footnotesize
\begin{abstract}
	We give a detailed description of the possible limits in the equivariant-Gromov-Hausdorff sense of sequences $(X_j,G_j)$, where the $X_j$'s are proper, geodesically complete, uniformly packed,  CAT$(0)$-spaces and the $G_j$'s are closed, totally disconnected, unimodular, uniformly cocompact groups of isometries. We show that the class of metric quotients $G/X$, where $X$ and $G$ are as above, is compact under Gromov-Hausdorff convergence. In particular it is a geometric compactification of the class of locally geodesically complete, locally compact, locally CAT$(0)$-spaces with uniformly packed universal cover and uniformly bounded diameter.
\end{abstract}
\normalsize

\tableofcontents

\section{Introduction}
The author, together with A.Sambusetti, studied in \cite{CS23} the possible Gromov-Hausdorff limits of locally geodesically complete, locally CAT$(0)$-spaces with bounded diameter and with uniformly packed universal cover, with particular attention to the collapsing case, i.e. when the dimension of the limit space is smaller than the dimension of the approximating ones.\\
In order to synthetize part of the results of \cite{CS23} we fix some notations. Let $\mathcal{O}\text{-CAT}_0^{\text{disc}}(P_0,r_0,D_0)$ be the class of quotient metric spaces $M:=\Gamma \backslash X$, where $X$ is a proper, geodesically complete, $(P_0,r_0)$-packed, CAT$(0)$-space and $\Gamma$ is a \emph{discrete} group of isometries of $X$ with $\text{diam}(\Gamma\backslash X) \leq D_0$. Here $\mathcal{O}$ stays for \emph{orbispace}. The numbers $P_0,r_0,D_0$ are structural constants that have to be thought arbitrary, but fixed once for all. The packing condition on $X$ is a very weak synthetic lower bound on the curvature of the metric space (cp. Section \ref{sec-packing} and \cite{CavS20} for more details) and is a necessary and sufficient condition for precompactness under Gromov-Hausdorff convergence (cp. Lemma \ref{lemma-GH-compactness-packing}, \cite{CS23} and \cite{CavS20}). In case $\Gamma$ is torsion-free then $M$ is actually a compact, locally geodesically complete, locally CAT$(0)$-space of diameter at most $D_0$.\\
One of the main results of \cite{CS23} is
\begin{theorem}[\textup{\cite[Theorem G]{CS23}}]
	\label{theo-intro-convergence}
	Let $\lbrace M_j \rbrace \subseteq \mathcal{O}\textup{-CAT}_0^{\textup{disc}}(P_0,r_0,D_0)$ be such that $M_j$ converges to some space $M_\infty$ in the Gromov-Hausdorff sense. Then $M_\infty$ is the quotient of a proper, geodesically complete, $(P_0,r_0)$-packed, \textup{CAT}$(0)$-space by a closed, totally disconnected group of isometries, and \textup{diam}$(M_\infty) \leq D_0$.
\end{theorem}
The proof of Theorem \ref{theo-intro-convergence} is based on a careful description of the convergence of the group actions $(X_j,\Gamma_j)$ to a limit group action $(X_\infty, G_\infty)$ in the equivariant pointed Gromov-Hausdorff sense (cp. Section \ref{sec-convergence} and \cite{CS23}), where $M_j = \Gamma_j \backslash X_j$ and $M_\infty = G_\infty \backslash X_\infty$. The main issue is that in general $G_\infty$ is not discrete, so it is not clear a priori the relation between $X_\infty$ and $M_\infty$. The analysis was divided in two cases: the collapsed and the non-collapsed one. The two cases can be summarized in the following equivalences (\cite[Theorem 7.10, Theorem 7.11, Theorem 7.15]{CS23}):
\begin{equation}
	\label{eq-collapsed-noncollapsed}
	\begin{aligned}
		 \text{Non-collapsed} &\Leftrightarrow G_\infty^o = \lbrace \id \rbrace \Leftrightarrow \dim(M_\infty) = \dim(M_j) \text{ for } j\gg 0;\\
		\text{Collapsed} &\Leftrightarrow G_\infty^o \neq \lbrace \id \rbrace \Leftrightarrow \dim(M_\infty) < \dim(M_j) \text{ for } j\gg 0,
	\end{aligned}
\end{equation}
where $G_\infty^o$ is the connected component of the identity of the topological group $G_\infty$. From \eqref{eq-collapsed-noncollapsed} we can see the validity of Theorem \ref{theo-intro-convergence} in the non-collapsed case. The same result in the collapsed case requires a deep understanding of $G_\infty^o$ which is a non trivial result of \cite{CS23}. We remark that in the generality of Theorem \ref{theo-intro-convergence}, it can happen that $M_\infty$ does not belong to $\mathcal{O}\text{-CAT}_0^{\text{disc}}(P_0,r_0,D_0)$, i.e. that the group $G_\infty$ is not discrete, even in the non-collapsed case (cp. \cite[Example 7.14]{CS23}). Therefore totally disconnected, non-discrete, actions appear in this framework as limit case of discrete actions and it is natural to study them in order to find a compactification of the space $\mathcal{O}\text{-CAT}_0^{\text{disc}}(P_0,r_0,D_0)$.\\
\vspace{1mm}

With this in mind we denote by $\mathcal{O}\text{-CAT}_0^{\text{td,u}}(P_0,r_0,D_0)$ the class of quotients $M:=G \backslash X$, where $X$ is again a proper, geodesically complete, $(P_0,r_0)$-packed, CAT$(0)$-space and $G$ is a \emph{closed, totally disconnected and unimodular} group of isometries of $X$ with $\text{diam}(G\backslash X) \leq D_0$. We refer to Section \ref{sec-preliminaries} for the definition of unimodularity. It naturally holds $$\mathcal{O}\text{-CAT}_0^{\text{disc}}(P_0,r_0,D_0) \subseteq \mathcal{O}\text{-CAT}_0^{\text{td,u}}(P_0,r_0,D_0)$$ 
since every discrete group is totally disconnected and unimodular. The unimodular conditions is justified by the following result, which is the main theorem of the paper. It extends Theorem \ref{theo-intro-convergence} to totally disconnected, unimodular actions, that is to sequences belonging to $\mathcal{O}\textup{-CAT}_0^{\textup{td-u}}(P_0,r_0,D_0)$. 
\begin{theorem}[Compactification Theorem]
	\label{theo-intro-closure}${}$\\
	The class $\mathcal{O}\textup{-CAT}_0^{\textup{td,u}}(P_0,r_0,D_0)$ is closed (and compact) under Gromov-Hausdorff convergence. In other words let $\lbrace M_j \rbrace \subseteq \mathcal{O}\textup{-CAT}_0^{\textup{td,u}}(P_0,r_0,D_0)$ be such that $M_j$ converges to some space $M_\infty$ in the Gromov-Hausdorff sense. Then $M_\infty$ is the quotient of a proper, geodesically complete, $(P_0,r_0)$-packed, \textup{CAT}$(0)$-space by a closed, totally disconnected, unimodular group of isometries, and \textup{diam}$(M_\infty) \leq D_0$.
\end{theorem}
Observe that in particular Theorem \ref{theo-intro-closure} shows that the space $M_\infty$ appearing in Theorem \ref{theo-intro-convergence} is not only the quotient by a totally disconnected group, but actually by a \emph{unimodular} one. A more important consequence is that the class $\mathcal{O}\textup{-CAT}_0^{\textup{td,u}}(P_0,r_0,D_0)$ is a compactification of $\mathcal{O}\textup{-CAT}_0^{\textup{disc}}(P_0,r_0,D_0)$ with respect to the Gromov-Hausdorff topology. This fact is useful in case we want to find uniform geometric estimates holding for all the spaces in $\mathcal{O}\textup{-CAT}_0^{\textup{disc}}(P_0,r_0,D_0)$. Examples of such uniform estimates will be proved especially in Section \ref{subsec-controlled} and will be used for the proof of Theorem \ref{theo-intro-closure}. To clarify this idea we present here a particular case of one of them.
\begin{theorem}
	\label{theo-intro-order-td}
	Let $P_0,r_0,D_0 > 0$. Then there exists $M_0 = M_0(P_0,r_0,D_0)$ such that the following holds.
	Let $X$ be a proper, geodesically complete, $(P_0,r_0)$-packed, \textup{CAT}$(0)$-space. Let $G<\textup{Isom}(X)$ be closed, totally disconnected, unimodular with \textup{diam}$(G\backslash X) \leq D_0$. Let $g\in G$ be a finite order isometry and write its order as $\textup{ord}(g) = p_1^{\alpha_1} \cdots p_k^{\alpha_k}$, where the $p_i$'s are prime numbers.
	Then $\max_i p_i \leq M_0$.
\end{theorem}
It is not difficult to check the validity of this statement on a regular tree. In that case the full isometry group is totally disconnected and unimodular. Theorem \ref{theo-intro-order-td} says that a similar phenomenon happens for every totally disconnected, unimodular group acting on a space as in the assumptions. The author does not know a proof of Theorem \ref{theo-intro-order-td} for discrete groups that does not use totally disconnected actions. As remarked after \eqref{eq-collapsed-noncollapsed} any argument by contradiction provides limit groups that can be totally disconnected, or worse, but in general not discrete.\\
Using the fact that if a group $G$ acts cocompactly on a metric space $X$, then $X$ is $(P_0,r_0)$-packed for some $P_0,r_0 > 0$ (cp. \cite[Lemma 5.4]{Cav21ter}), we also get the following, non-uniform, estimate.
\begin{corollary}
	Let $X$ be a proper, geodesically complete, \textup{CAT}$(0)$-space and let $G < \textup{Isom}(X)$ be closed, totally disconnected, unimodular and cocompact. Then there exists $M \in \mathbb{N}$ such that every prime number appearing in the prime decomposition of the order of every finite order isometry of $G$ is at most $M$.
\end{corollary}

The proof of Theorem \ref{theo-intro-closure} follows a scheme similar to the one we used in \cite{CS23}, but with two additional difficulties. First of all we have to deal with the fact that the groups under consideration are no more discrete, but just totally disconnected. This impacts every key argument in \cite{CS23}: indeed the discussion therein starts from the Margulis Lemma. That is why in Section \ref{sec-Margulis} we prove the following, new version of the Margulis Lemma holding for totally disconnected actions on proper metric spaces. It generalizes \cite[Corollary 11.17]{BGT11} and it is proved by similar techniques.

\begin{theorem}[Margulis Lemma for totally disconnected group actions.]
	\label{theo-intro-Margulis-TD}${}$\\
	Let $K \geq 1$. There exists $\varepsilon = \varepsilon(K)$ such that the following is true. Let $X$ be any proper metric space and let $x\in X$ be a point such that $\textup{Cov}(\overline{B}(x,4), 1) \leq K$. Let $G$ be a closed, totally disconnected group of isometries of $X$. Then the group $G_\varepsilon(x) = \langle S_\varepsilon(x)\rangle$ contains an open, compact, normal subgroup $N$ such that $G_\varepsilon(x)/N$ is discrete, finitely generated and nilpotent.
\end{theorem}
The expression $\textup{Cov}(\overline{B}(x,4), 1)$ refers to the minimal cardinality of a $1$-net inside $\overline{B}(x,4)$, see Section \ref{sec-packing} for a precise definition and its relation with the packing. The set $S_\varepsilon(x)$ is the set of elements of $G$ moving $x$ less than $\varepsilon$, sometimes it is called the almost stabilizer of $x$. For more properties of these sets see Section \ref{subsection-isometries}. Observe that if $G$ is discrete then $N$ is finite, and Theorem \ref{theo-intro-Margulis-TD} becomes the known Margulis Lemma \cite[Corollary 11.17]{BGT11}.\\
In general the information provided by Theorem \ref{theo-intro-Margulis-TD} is not sufficient to get a good geometric description of the action of the almost stabilizers. The situation is different if the space $X$ is CAT$(0)$ and the group $G$ is cocompact. Under these assumptions the group $G_\varepsilon(x)/N$ is actually abelian and a geometric description of the action of such groups can be provided (cp. Section \ref{subsec-almost-abelian-CAT}). It resembles the description we gave in \cite[§2.4]{CS23} for virtually abelian groups acting on CAT$(0)$-spaces.\\
With these tools in our hands we will adapt the proofs of the Splitting Theorem D and the Renormalization Theorem E of \cite{CS23} to our setting. This will be done in Sections \ref{sec-splitting} and \ref{sec-renormalization}.
In Section \ref{sec-convergence}, and in particular in Section \ref{sub-almost stabilizers}, we use the unimodularity assumption, that is not used until this point, to attack the proof of Theorem \ref{theo-intro-closure}. This is enough to prove \eqref{eq-collapsed-noncollapsed} in this more general setting. To be more precise we can prove that if $M_\infty$ is the Gromov-Hausdorff limit of spaces $M_j \in \mathcal{O}\text{-CAT}_0^{\text{td,u}}(P_0,r_0,D_0)$ then $M_\infty = G_\infty' \backslash X_\infty'$ for some proper, geodesically complete, $(P_0,r_0)$-packed, CAT$(0)$-space $X_\infty'$ and some closed, totally disconnected, $D_0$-cocompact group of isometries $G_\infty'$. Observe that this is exactly the statement of Theorem \ref{theo-intro-convergence} if the $M_j$'s belong to $\mathcal{O}\text{-CAT}_0^{\text{disc}}(P_0,r_0,D_0)$.
\vspace{1mm}

\noindent The second additional difficulty in the proof of Theorem \ref{theo-intro-closure} with respect to \cite{CS23} consists in the last step, namely in proving that the group $G_\infty'$ above is actually \emph{unimodular}. We sketch here the argument in the non-collapsed case, the collapsed one being more involved but conceptually similar. In the non-collapsed case, calling $M_j = G_j\backslash X_j$, we know that $(X_j,G_j)$ converges in the equivariant Gromov-Hausdorff sense to some $(X_\infty, G_\infty)$. Moreover $G_\infty$ is totally disconnected by the equivalent of \eqref{eq-collapsed-noncollapsed}, so $M_\infty = G_\infty \backslash X_\infty$, in particular we can take $X_\infty = X_\infty'$ and $G_\infty = G_\infty'$. It is here that the non-collapsed assumption simplifies the argument. We need to prove that $G_\infty$ is unimodular. First of all we prove that $G_j$ converges in the pointed Gromov-Hausdorff sense to $G_\infty$ with respect to suitable metrics (cp. Proposition \ref{prop-convergence-group-metric}). Then we give a sufficient condition for the convergence of the Haar measures of $G_j$ to a Haar measure of $G_\infty$ along this pointed Gromov-Hausdorff convergence (cp. Proposition \ref{prop-convergence-Haar}). Lastly we use uniform quantitative estimates in the spirit of Theorem \ref{theo-intro-order-td} to prove that this sufficient condition for the convergence is satisfied.
\vspace{1mm}

\noindent We end the introduction by highlighting a result of independent interest that we will use in the paper: it is a characterization of cocompact, totally disconnected subgroups of isometries of a proper, geodesically complete, CAT$(0)$-space in terms of the properties of the action (cp. Theorem \ref{theo-characterization-td}).
\begin{theorem}
	Let $X$ be a proper, geodesically complete, \textup{CAT}$(0)$-space and let $G$ be a closed, cocompact group of isometries of $X$. Then the following are equivalent:
	\begin{itemize}
		\item[(i)] $G$ is totally disconnected;
		\item[(ii)] $G$ is semisimple and $\inf \lbrace \ell(g) \text{ s.t. } g \text{ hyperbolic} \rbrace > 0$;
		\item[(iii)] the orbit $Gx$ is discrete  for every $x\in X$.
		\item[(iv)] the orbit $Gx$ is discrete for one point $x\in X$.
	\end{itemize}
\end{theorem}
We recall that $G$ is semisimple if it does not contain parabolic isometries, while the quantity $\ell(g)$ appearing in (ii) is the translation length of the isometry $g$, i.e. $\ell(g) = \inf_{x\in X} d(x,gx)$. The theorem is false even for proper, CAT$(0)$-spaces that are not geodesically complete (see Example \ref{ex-td-non-discrete-orbits}).
 
\vspace{2mm}

\small {\em
	\noindent {\sc Acknowledgments.} The author thanks A.Lytchak for the interesting discussions during the preparation of this paper.}
\normalsize

\vspace{2mm}

\section{Preliminaries}
\label{sec-preliminaries}

We start recalling several facts we will need along the paper. We start with notions on metric spaces.

\subsection{CAT$(0)$-spaces}
\label{sec-CAT}
Throughout the paper $X$ will be a {\em proper} metric space with distance $d$.
The open (resp. closed) ball in $X$ of radius $r$, centered at $x$, will be denoted by $B_X(x,r)$ (resp.  $\overline{B}_X(x,r)$); we will often drop the subscript $X$ when the space  is clear from the context.\\
A {\em geodesic} in a metric space $X$ is an isometry $c\colon [a,b] \to X$, where $[a,b]$ is an interval of $\mathbb{R}$. The {\em endpoints} of the geodesic $c$ are the points $c(a)$ and $c(b)$; a geodesic with endpoints $x,y\in X$ is also denoted by $[x,y]$. A {\em geodesic ray} is an isometry $c\colon [0,+\infty) \to X$ and a geodesic line is an isometry $c\colon \mathbb{R} \to X$. A metric space $X$ is called   {\em geodesic}  if for every two points $x,y \in X$ there is a geodesic with endpoints   $x$ and $y$. 

\noindent A metric space $X$ is called CAT$(0)$ if it is geodesic and every geodesic triangle $\Delta(x,y,z)$  is thinner than its Euclidean comparison triangle   $\overline{\Delta} (\bar{x},\bar{y},\bar{z})$: that is,  for any couple of points $p\in [x,y]$ and $q\in [x,z]$ we have $d(p,q)\leq d(\bar{p},\bar{q})$ where $\bar{p},\bar{q}$ are the corresponding points in $\overline{\Delta} (\bar{x},\bar{y},\bar{z})$ (see   for instance \cite{BH09} for the basics of CAT(0)-geometry).
A CAT$(0)$-space is {\em uniquely geodesic}: for every $x,y \in X$ there exists a unique geodesic with endpoints $x$ and $y$.

\noindent A CAT$(0)$-metric space $X$ is {\em geodesically complete} if any geodesic $c\colon [a,b] \to X$ can be extended to a geodesic line. 

\noindent The  {\em boundary at infinity} of a CAT$(0)$-space $X$ (that is, the set of equivalence classes of geodesic rays, modulo the relation of being asymptotic), endowed with the Tits distance,  will be denoted by $\partial X$, see \cite[Chapter II.9]{BH09}.

\noindent A subset $C$ of $X$ is said to be {\em convex} if for all $x,y\in C$ the geodesic $[x,y]$ is contained in $C$. Given a subset $Y\subseteq X$ we denote by $\text{Conv}(Y)$ the smallest convex, closed subset of $X$ containing $Y$. 
If $C$ is a convex subset of a CAT$(0)$-space $X$ then it is itself CAT$(0)$, and its boundary at infinity $ \partial C$ naturally and isometrically embeds in $\partial X$.
\vspace{2mm}

We will denote by HD$(X)$ and TD$(X)$ the Hausdorff and the topological dimension of a metric space $X$, respectively.  
By \cite{LN19} we know that if $X$ is a  proper and geodesically complete CAT$(0)$-space then every point $x\in X$ has a well defined integer dimension in the following sense: there exists $n_x\in \mathbb{N}$ such that every small enough ball around $x$ has Hausdorff dimension equal to $n_x$. This defines a {\em stratification of $X$} into pieces of different integer dimensions: namely, if $X^k$ denotes the subset of points of $X$ with dimension $k$, then 
\vspace{-5mm}

$$X= \bigcup_{k\in \mathbb{N}} X^k.$$
The {\em dimension} of $X$ is the supremum of the dimensions of its points:  it coincides with the {\em topological dimension} of $X$, cp. \cite[Theorem 1.1]{LN19}.

\subsection{Packing and covering}
\label{sec-packing}
\noindent Let $X$ be a metric space and $r>0$.
A subset $Y$ of $X$ is called {\em $r$-separated} if $d(y,y') > r$ for all $y,y'\in Y$, while it is a {\em $r$-net} if for all $x\in X$ there exists $y\in Y$ such that $d(x,y)\leq r$.
Given $x\in X$ and $0<r\leq R$ we denote by Pack$(\overline{B}(x,R), r)$ the maximal cardinality of a $2r$-separated subset of $\overline{B}(x,R)$, and by Cov$(\overline{B}(x,R), r)$ the minimal cardinality of a $r$-net of $\overline{B}(x,R)$. Moreover we denote by Pack$(R,r)$ (resp. Cov$(R,r)$) the supremum of Pack$(\overline{B}(x,R), r)$ (resp. Cov$(\overline{B}(x,R), r)$) among all points of $X$. The packing and covering quantities defined above are classically related as follows (cp. \cite[§ 4]{CavS20}):
\begin{gather}
	\label{eq-pack-cov}
		\text{Pack}(\overline{B}(x,R), r) \leq \text{Cov}(\overline{B}(x,R), r) \leq \text{Pack}\left(\overline{B}(x,R), \frac{r}{2}\right),\\
		\text{Pack}(R, r) \leq \text{Cov}(R, r) \leq \text{Pack}\left(R, \frac{r}{2}\right).
\end{gather}

\noindent Given $P_0,r_0 > 0$ we say that $X$ is {\em $(P_0,r_0)$-packed} if Pack$(3r_0,r_0) \leq P_0$.
The packing condition should be thought as a metric, weak replacement  of a Ricci curvature lower bound: for more details and examples see \cite{CavS20}. Also remark that every metric space admitting a cocompact action is packed (for some $P_0, r_0$), see the proof of \cite[Lemma 5.4]{Cav21ter}.
\vspace{1mm}

\noindent The packing condition  has many interesting geometric consequences for complete, geodesically complete CAT$(0)$-spaces, as showed in \cite{CavS20bis}, \cite{Cav21bis} and  \cite{Cav21}. For instance it provides an upper bound on the dimension.

\begin{prop}[\textup{\cite[Theorems 4.2, 4.9]{CavS20}}]
	\label{prop-packing}
	Let $X$ be a complete, geodesically complete, $(P_0,r_0)$-packed, $\textup{CAT}(0)$-space. Then $X$ is proper and 
	\begin{itemize}
		\item[(i)] $\textup{Pack}(R,r) \leq P_0(1+P_0)^{\frac{R}{\min\lbrace r,r_0\rbrace} - 1}$ for all $0 <r\leq R$;	 
		\item[(ii)] the dimension of $X$ is at most $n_0 := P_0/2$;
	\end{itemize}
\end{prop}

\noindent In particular,  for a geodesically complete, CAT$(0)$ space $X$ which is $(P_0, r_0)$-packed, the assumptions {\em complete} and {\em proper} are interchangeable.

\subsection{Topological groups}
We recall now some terminology about topological groups. A topological group is a group $G$ endowed with a topology for which the operation and the inverse are continuous. Every topological group we will consider will be locally compact and $\sigma$-compact, in particular second-countable. Any such group $G$ admits a left invariant Haar measure, as well as a right invariant Haar measure. By definition they are Radon measures on the Borel $\sigma$-algebra of $G$ that are preserved by the multiplication on the left (resp. on the right). The next lemma characterizes Haar measures.
\begin{lemma}
	\label{lemma-Haar-measure}
	Let $G$ be a locally compact, second countable group. Then a left (resp. right) invariant measure on the Borel $\sigma$-algebra of $G$ is a left (resp. right) invariant Haar measure if and only if it is finite on compact subsets and positive on open subsets.
\end{lemma}
\begin{proof}
	The Haar measure of every open set is positive by left (resp. right) invariance. Viceversa since $G$ is locally compact and second countable every open set is $\sigma$-compact. Therefore by \cite[Theorem 2.18]{Rud63} every left (resp. right) invariant measure which is finite on compact subsets and positive on open subsets is Radon, and so a Haar measure.
\end{proof}
A topological group is called \emph{unimodular} if it admits a \emph{left and right invariant} Haar measure. For instance every discrete group is unimodular.\\
The connected component of the identity of a topological group $G$ will be denoted by $G^o$. It is always a normal (actually characteristic), closed subgroup of $G$. The quotient group $G/G^o$ endowed with the quotient topology is still locally compact and totally disconnected as a topological space. Therefore a topological group $G$ is totally disconnected if and only if $G^o = \lbrace 1 \rbrace$.\\
A topological group $G$ is \emph{compactly generated} if there exists a compact subset $S\subseteq G$ such that $\langle S \rangle = G$. The notation $\langle S \rangle$ stays for the group generated by $S$. For instance a discrete group is compactly generated if and only if it is finitely generated. 

\subsection{Isometry groups}
\label{subsection-isometries}
For us the main source of topological groups is the full isometry group of a proper metric space.
Let $\text{Isom}(X)$ be the group of isometries of $X$, endowed with the compact-open topology: as $X$ is proper, it is a topological, locally compact, $\sigma$-compact  group. 
\noindent Let $G$ be a closed subgroup of Isom$(X)$. For $x\in X$ and $r,R\geq 0$ we set 
\begin{equation}	
	\label{defsigma} 	
	\begin{aligned}
			\overline{S}_r(x, R, X) &:= \lbrace g\in G \text{ s.t. } d(y,gy) \leq r \text{ for all } y \in \overline{B}(x,R)\rbrace,\\
			S_r(x, R, X) &:= \lbrace g\in G \text{ s.t. } d(y,gy) < r \text{ for all } y \in \overline{B}(x,R)\rbrace,\\
			\overline{G}_r(x, R, X) &:= \langle \overline{S}_r(x, R, X) \rangle, \quad	G_r(x, R, X) := \langle S_r(x, R, X) \rangle.
	\end{aligned}
\end{equation}
For $R=0$ we simply write $\overline{S}_r(x, 0, X) =: \overline{S}_r(x, X)$ and similarly for the others. When $r = 0$ we use the notation $\overline{S}_0(x, R, X) =: \textup{Stab}_G(x,R)$: it is the pointwise stabilizer of the ball $\overline{B}(x,R)$. In the same way when both $R = r = 0$ we use tha notation $\overline{S}_0(x, 0, X) =: \textup{Stab}_G(x)$: it is the stabilizer of the point $x$. When the context is clear we will omit to write the dependence on the metric space $X$. The sets $\overline{S}_r(x, R)$ are compact by Ascolì-Arzelà Theorem, since $X$ is proper. In particular the groups $\overline{G}_r(x, R)$ are compactly generated. Instead the sets ${S}_r(x, R)$ are open. It follows from the definition that the sets $\lbrace S_r(x, R) \rbrace_{r,R > 0}$ form a neighbourhood basis of the identity.
\vspace{1mm}

\noindent Let $x\in X$ be a fixed base point. The left and right invariant $L^\infty$-pseudometrics at scale $R$ of center $x$ are by definition
\begin{equation}
	\label{eq-defin-L-infty-pseuodidistance}
		d_{x,R}^\ell(g,h) := \max_{y\in \overline{B}(x,R)} d(gy,hy), \quad d_{x,R}^\textup{r}(g,h) := \max_{y\in \overline{B}(x,R)} d(g^{-1}y,h^{-1}y).
\end{equation}
As the name suggests they are left (resp. right) invariant pseudometrics on $\text{Isom}(X)$. The expressions
\begin{equation}
	\label{eq-defin-metric-group}
	d_{x}^\ell (g,h) := \inf_{R > 0} \left\lbrace \frac{1}{R} + d_{x,R}^\ell(g,h) \right\rbrace; \quad d_{x}^\textup{r} (g,h) := \inf_{R > 0} \left\lbrace \frac{1}{R} + d_{x,R}^\textup{r}(g,h) \right \rbrace
\end{equation}
define proper, left (resp. right) invariant metrics on $\text{Isom}(X)$ inducing the compact-open topology (cp. \cite[§2.5]{SZ22}). This follows also from the next result whose proof is a direct consequence of the definitions. The closed ball of center $g$ and radius $r$ with respect to the metric $d_{x}^\ell$ (resp. $d_{x}^\text{r}$) will be denoted by $\overline{B}_{x}^\ell(g,r)$ (resp. $\overline{B}_{x}^\text{r}(g,r)$).
\begin{lemma}
	\label{lemma-comparison-metric-group}
	Let $X$ be a proper metric space, $G< \textup{Isom}(X)$ be a closed subgroup, $x\in X$ and $r>0$. Then 
	$$\overline{S}_{\frac{r}{2}}\left(x, \frac{2}{r}\right) \subseteq \overline{B}_{x}^\ell(\id,r), \overline{B}_{x}^\textup{r}(\id,r) \subseteq \overline{S}_r\left(x, \frac{1}{r}\right).$$
\end{lemma}

\noindent A subgroup $G$ is \emph{totally disconnected} if it is totally disconnected as a subset of \textup{Isom}$(X)$ (with respect to the compact-open topology). 
A closed group $G < \textup{Isom}(X)$ is said to be {\em cocompact} if the quotient metric space $G \backslash X$ is compact; in this case, we call \emph{codiameter} of $G$ the diameter of the quotient, and we will say that $G$ is $D_0$-cocompact if it has codiameter at most $D_0$.
Notice that the codiameter of $G$ coincides with
\begin{equation}
	\label{eq-def-codiameter}
	\inf \lbrace r>0 \text{ s.t. } G \cdot \overline{B}(x,r) = X \,\,\,\,\forall x\in X\rbrace.
\end{equation}	
It is well-known that if $X$ is geodesic and $G < \textup{Isom}(X)$ is closed and $D_0$-cocompact, then  {\em the subset  $\overline{S}_{3D_0}(x)$ is a  generating set for $G$,} that is $\overline{G}_{3D_0}(x)=G$, for every $x\in X$. If the orbit $Gx$ is discrete then the same is true for $2D_0$ in place of $3D_0$.
\vspace{2mm}

\noindent The {\em translation length} of $g\in \text{Isom}(X)$ is by definition  $\ell(g) := \inf_{x\in X}d(x,gx).$ 
When the infimum is realized, the isometry $g$ is called {\em elliptic} if $\ell(g) = 0$ and {\em hyperbolic} otherwise. The {\em minimal set of $g$}, $\text{Min}(g)$, is defined as the subset of points of $X$ where $g$ realizes its translation length; notice that if $g$ is elliptic then $\text{Min}(g)$ is the subset of points fixed by $g$. An isometry is called \emph{semisimple} if it is either elliptic or hyperbolic; a subgroup $G$ of Isom$(X)$ is  called \emph{semisimple} if all of its elements are semisimple. The \emph{displacement function} $d_g \colon X \to X$ defined by $d_g(x) = d(x,gx)$ is convex (\cite[Proposition II.6.2]{BH09}). For $\lambda \geq 0$ we call the set $M_\lambda(g) := d_g^{-1}([0,\lambda])$ the \emph{$\lambda$-level set} of $g$. It is a closed, convex subset of $X$. For instance $M_{\ell(g)}(g) = \text{Min}(g)$.

	\noindent The {\em free-systole} of a group $G < \textup{Isom}(X)$ {\em at a point} $x\in X$ is
	$$\text{sys}^\diamond(G,x) := \inf_{g\in G \setminus G^\diamond} d(x,gx),$$
	where $G^\diamond $ is the subset of all elliptic isometries of $G$. The {\em free-systole of} $G$ is accordingly defined as 
	$$\text{sys}^\diamond(G,X) = \inf_{x\in X}\text{sys}^\diamond(G,x).$$
	Similarly,  the {\em free-diastole of} $G$ is defined as  	
	$$\text{dias}^\diamond(G,X) = \sup_{x\in X}\text{sys}^\diamond(G,x).$$	
	
	\noindent By definition, we have the trivial inequality
	$\text{sys}^\diamond(G,X) \leq  \text{dias}^\diamond(G,X)$.
	\noindent The following result, which extends \cite[Theorem 3.1]{CS23} to general groups acting on geodesically complete, CAT$(0)$-spaces, shows that the free systole and the free diastole are for small values quantitatively equivalent, provided one knows an a priori bound on the diameter of the quotient. Therefore it is equivalent to consider actions with small free-systole or small free-diastole. In the rest of the paper we will focus on the free-systole.
	
	\begin{prop}
		\label{prop-vol-sys-dias}
		Let $X$ be a proper, geodesically complete, $(P_0,r_0)$-packed, $\textup{CAT}(0)$-space and let $G<\textup{Isom}(X)$ be $D_0$-cocompact. Then
			$$\textup{sys}^\diamond(G, X) \geq 
			\left( 1+P_0 \right)e^{-\frac{(2D_0+1)}{\min\lbrace \textup{dias}^\diamond(G,X), r_0 \rbrace}  }.$$
	\end{prop}	
	
	\noindent The proof is the same of \cite[Theorem 3.1]{CS23}, we report it for completeness. It is based on the following important fact we will use also later.
	
	\begin{prop}[\textup{\cite[Proposition 3.3]{CS23}, \cite[Proposition 4.5]{CavS20bis}}]
		\label{lemma-Sylvain}
		Let $P_0,r_0,R> 0$ and $0<\varepsilon \leq r_0$. Then there exists $\delta(P_0,r_0, R, \varepsilon)  > 0$ with the following property. Let $X$ be a proper, geodesically complete, $(P_0,r_0)$-packed, \textup{CAT}$(0)$-space. If $g$ is a non-elliptic isometry of $X$ and $x$ is a point of $X$ such that $d(x,gx) \leq \delta(P_0,r_0, R, \varepsilon)$, then for every $y\in X$ with $d(x,y)\leq R$ there exists $m\in \mathbb{Z}^*$ such that $d(y,g^m y) \leq \varepsilon$.\\
		(We can choose, explicitely,  $\delta(P_0,r_0, R, \varepsilon)= (1+P_0)e^{-\frac{(2R+1)}{\varepsilon}}$).
	\end{prop}	
	
	\begin{proof}[Proof of Proposition 	\ref{prop-vol-sys-dias}]
		Assume that 
		$\min\lbrace \text{dias}^\diamond (G ,  X), r_0 \rbrace > \varepsilon$. 
		By definition there exists $x_0 \in X$ such that for every hyperbolic isometry $g \in G$ one has  $d(x_0,gx_0) >\varepsilon$. 
		Now, if sys$^\diamond (G, X) \leq (1+P_0)e^{-\frac{(2D_0+1)}{\varepsilon}} =: \delta$, we could find $x \in X$ and a hyperbolic $g \in G$ such that $d(x,gx) \leq \delta$. 
		By Proposition \ref{lemma-Sylvain}, for every $y \in \overline{B}(x,D_0)$ there would exists a non trivial power $g^m$ satisfying $d(y,g^my) \leq \varepsilon$. 
		But then, since the action is $D_0$-cocompact we could find a conjugate $\gamma$ of $g^m$ (thus,  a hyperbolic isometry) such that $ d(x_0,\gamma x_0)\leq \varepsilon$, a contradiction. The conclusion follows by the arbitrariness of $\varepsilon$.
	\end{proof}
	
\subsection{Totally disconnected, cocompact groups of CAT$(0)$-spaces}
We collect here some important properties of cocompact, totally disconnected groups acting on a proper, geodesically complete, CAT$(0)$-space that we will use extensively along the paper. We suggest the reader to consult also \cite{Cap09} and \cite{CM09b}.

\begin{theo}
	\label{theo-characterization-td}
	Let $X$ be a proper, geodesically complete, \textup{CAT}$(0)$-space and let $G$ be a closed, cocompact group of isometries of $X$. Then the following are equivalent:
	\begin{itemize}
		\item[(i)] $G$ is totally disconnected;
		\item[(ii)] $G$ is semisimple and $\textup{sys}^\diamond(G,X) > 0$;
		\item[(iii)] the orbit $Gx$ is discrete  for every $x\in X$.
		\item[(iv)] the orbit $Gx$ is discrete for one point $x\in X$.
	\end{itemize}
	Moreover if any of the conditions above hold then:
	\begin{itemize}
		\item[(a)] the sets $\textup{Stab}_G(x,R)$ are open for every $x\in X$ and every $R \geq 0$. They form a fundamental system of neighbourhoods of $\lbrace \textup{id} \rbrace$.
		\item[(b)] The sets $\overline{S}_r(x, R)$ are open for every $x\in X$ and every $r, R\geq 0$. 
		\item[(c)] The set $\lbrace gx \text{ s.t. } g\in K\rbrace$ is finite for every $x\in X$ and every precompact $K\subseteq G$.
	\end{itemize}
\end{theo}


We need the following fact.
\begin{lemma}
	\label{prop-connected-component}
	Let $X$ be a geodesically complete, \textup{CAT}$(0)$-space and let $G < \textup{Isom}(X)$ be a closed, cocompact group. Then $G^o$ is either the identity or it contains an infinite cyclic discrete subgroup.
\end{lemma}
\begin{proof}
	Every locally compact group either contains a compact open subgroup or an infinite cyclic discrete subgroup. We apply this statement to $G^o$: if it contains a compact open subgroup then it must be the whole $G^o$ since it is connected. Therefore $G^o$ is a compact, normal subgroup of $G$. In particular $G$ stabilizes the non-empty (cp. \cite[Corollary II.2.8]{BH09}) closed, convex set $\text{Fix}(G^o)$. By minimality of $G$ (cp. \cite[Proposition 1.5]{CM09b}) we have that $\text{Fix}(G^o) = X$, so $G^o = \lbrace \text{id} \rbrace$. Thus either $G^o = \lbrace \id \rbrace$ or $G^o$ contains an infinite cyclic discrete subgroup.
\end{proof}

We will use some basic facts about convergence recalled in Section \ref{sec-convergence}. 

\begin{proof}[Proof of Theorem \ref{theo-characterization-td}]
	First of all we observe that, given $x\in X$, the orbit $Gx$ is discrete if and only if the group $\text{Stab}_G(x)$ is open. Indeed suppose $Gx$ is discrete and take $g_j \in G$ converging to $g$, with $gx = x$. Since $g_jx$ converges to $gx = x$ and the orbit $Gx$ is discrete, then $g_jx = x$ for $j$ big enough. On the other hand, suppose $\text{Stab}_G(x)$ is open and that $Gx$ is not discrete. The set $Gx$ is closed and not discrete, so it has an accumulation point $gx$. This means that there are pairwise distinct $g_jx$ converging to $gx$ with $g_j \in G$. Up to multiply on the left by $g^{-1}$ we can suppose that $g_jx$ converges to $x$, and it is still composed by pairwise distinct points. The sequence $g_j$ converges, up to a subsequence, to some $g' \in \text{Stab}_G(x)$. Since the latter is open then $g_j \in \text{Stab}_G(x)$ for $j$ big enough, i.e. $g_jx = x$, which is a contradiction.\\
	Now, (i) implies (ii) by \cite[Corollary 6.3]{CM09b}, while (i) implies (a) by \cite[Theorem 6.1]{CM09b} and so (i) implies also (iii) by the first observation. It is clear that (iii) implies (iv).\\
	Now we show that (iv) implies (i). As in the proof of \cite[Theorem 6.1]{CM09b} we denote by $C$ the subset of $X$ of points whose stabilizer is open. Then $C$ is convex (if $x,y\in C$ then $\text{Stab}_G(x)\cap \text{Stab}_G(y)$ is an open subgroup that stabilizes $[x,y]$) and $G$-invariant. It is also non-empty because of (iv) and the first observation. Since the action of $G$ is minimal (cp. \cite[Proposition 1.5]{CM09b}) we conclude that $C$ is dense. We use this fact to provide a neighbourhood basis of the identity made of compact open subgroups. This is enough to show that $G$ is totally disconnected. We fix a basepoint $x_0\in X$, $R\geq 0$ and $\varepsilon > 0$ and we choose a finite subset $\lbrace c_1,\ldots,c_n\rbrace$ of $C\cap \overline{B}(x,R)$ which is $\varepsilon$-dense in $\overline{B}(x,R)$. We set $U(R,\varepsilon) := \bigcap_{i=1}^n \text{Stab}_G(c_i)$. It is clear that $U(R,\varepsilon)$ is open and compact, being a finite intersection of open and compact subgroups. We claim that the family $\lbrace U(R,\varepsilon) \rbrace$ form a neighbourhood basis of $\lbrace \id \rbrace$. For that it is enough to show that $\bigcap_{R\geq 0, \varepsilon > 0} U(R,\varepsilon) = \lbrace \id \rbrace$. Let $g$ be in this intersection. For any $R\geq 0$ and any $\varepsilon > 0$ we know that $g$ acts trivially on a $\varepsilon$-net of $\overline{B}(x_0,R)$, so it displaces every point of $\overline{B}(x_0,R)$ by at most $2\varepsilon$. Since this is true for every $\varepsilon$ we deduce that $g$ fixes pointwise the whole $\overline{B}(x_0,R)$. By arbitrariness of $R$ we conclude that $g$ fixes pointwise every point of $X$, i.e. $g=\id$.\\
	It remains to show that (ii) implies (i). We suppose that $G$ has only elliptic and hyperbolic elements and that $\text{sys}^\diamond(G,X) =: s > 0$. We claim that the set of elliptic isometries $G^\diamond$ is open and closed in $G$. Suppose $g_j \in G$ converges to $g \in G^\diamond$ and let $x\in \text{Fix}(g)$. Then, for $j$ big enough, we have $d(x,g_jx) < s$, so $g_j\in G^\diamond$. This shows that $G^\diamond$ is open.
	Suppose now to have $g_j \in G^\diamond$ converging to a hyperbolic isometry $g$. Let $x$ be a point such that $d(x, g x) = \ell(g)$. 	
	We denote by $z_j$ the projection of $x$ on the closed, convex set $M_{\frac{s}{4}}(g_j)$. It must hold $d(x, z_j) \to +\infty$, otherwise $\ell(g) \leq \frac{s}{4}$. Let $y_j$ be the point along the geodesic $[z_j, x]$ such that $d(g_jy_j, y_j) = \frac{s}{2}$. Observe that by convexity of the displacement function $d_{g_j}$ along the geodesic $[z_j,x]$ we have $d(z_j,y_j) \to +\infty$. We consider the sequence of isometric actions $(X, y_n, G)$ and we fix a non-principal ultrafilter $\omega$ (see Section \ref{sec-convergence} for the definitions). Since $G$ is cocompact, by Lemma \ref{lemma-ultralimit-cocompact} we have that the ultralimit $(X_\omega, y_\omega, G_\omega)$ of this sequence is naturally equivariantly isometric to $(X,x,G)$. Observe that the sequence $g_j$ defines an element $g_\omega$ of $G_\omega$. The condition $d(z_j,y_j) \to +\infty$ forces the inequality $\ell(g_\omega) \geq \frac{s}{4}$, so $g_\omega$ is hyperbolic. On the other hand $d(g_\omega y_\omega, y_\omega) = \frac{s}{2}$, so $\ell(g_\omega) \leq \frac{s}{2}$. This contradicts the assumption $\text{sys}^\diamond(G_\omega,X_\omega) = \text{sys}^\diamond(G,X) = s$. Therefore the set $G^\diamond$ is open and closed, implying $G^\diamond\cap G^o = G^o$. By Lemma \ref{prop-connected-component} either $G^o = \lbrace \id \rbrace$ or it contains an infinite cyclic discrete subgroup. Every element of $G^o$ is elliptic, so if it generates a discrete subgroup then it must be of finite order. This forces to have $G^o=\lbrace\text{id}\rbrace$, i.e. $G$ is totally disconnected.
	This concludes the proof of the equivalences (i)-(iv).\\
	These conditions imply (a) by \cite[Theorem 6.1]{CM09b}. Now (b) is a consequence of (i). Indeed suppose $\overline{S}_r(x, R)$ is not open for some $x\in X$ and $r, R\geq 0$. Then
	there exists a sequence  $g_j \in G \setminus \overline{S}_r(x, R)$ converging to $g\in \overline{S}_r(x, R)$.
	Then $g^{-1}g_j$ converges to $\textup{id}\in \textup{Stab}_G(x, R)$. By (a) it must hold $g^{-1}g_j \in \textup{Stab}_G(x, R)$ for $j$ big enough, i.e. $g_j = g$ on $\overline{B}(x,R)$. This contradicts the assumption $g_j \notin \overline{S}_r(x, R)$. Finally if $K$ is a compact subset of $G$ then $Gx$ is a compact subset of $X$ which is also discrete by (iii). Therefore it must be finite.
\end{proof}
The equivalences (i)-(iv) are no more true if the space is proper, CAT$(0)$, but not geodesically complete. In the following example we show that (i) does not imply (iii). 

\begin{ex}
	\label{ex-td-non-discrete-orbits}
	We construct a tree in the following iterative way. We start with $X_1 = [-1,1]$. The set $V_1 = \lbrace \pm 1 \rbrace$ is the set of free vertices of $X_1$, i.e. those vertices with only one edge issuing from them. To construct $X_2$ we glue a segment of length $1$ at the point $0 \in X_1$. $X_2$ is composed by three segments of length $1$ glued at a point. Its set of free vertices $V_2$ has $3$ elements and contains $V_1$. To construct $X_3$ we glue segments of length $1/2$ to the midpoints of all the segments of length $1$ composing $X_2$, so that $X_3$ is composed by $9$ segments of length $1/2$. Its set of free vertices is denoted by $V_3$: it contains $V_2$. Given $X_j$, made of $3^{j-1}$-segments of length $2^{-j + 2}$, we construct $X_{j+1}$ by gluing segments of length $2^{-j+1}$ to each midpoint of the segments of $X_j$. We denote by $V_{j+1}$ the set of free vertices of $X_j$: it contains $V_j$. We repeat the procedure inductively and we define the space $X_\infty$: it is a tree, so a CAT$(0)$-space, which is also compact, so proper, but not geodesically complete. Let us call $V_\infty$ its set of free vertices. For every $j$ and every two free vertices $v,w \in V_j$ we can find an isometry of $X_j$ sending $v$ to $w$. Such isometry extends naturally to an isometry of $X_\infty$, by construction of $X_\infty$. Therefore $\text{Isom}(X_\infty)$ acts transitively on the set of free vertices $V_\infty$. This implies that its action on $X_\infty$ is minimal. The same proof of Lemma \ref{prop-connected-component} shows that either $\text{Isom}(X_\infty)$ is totally disconnected or it contains an infinite cyclic discrete subgroup. The second possibility cannot occur because $\text{Isom}(X_\infty)$ is compact, so it contains only elliptic isometries and a discrete group generated by an elliptic isometry is finite. This shows that $\text{Isom}(X_\infty)$ is totally disconnected. However the orbit by $\text{Isom}(X_\infty)$ of every free vertex is not discrete.
\end{ex}

\subsection{Crystallographic groups in the Euclidean space}
We will denote points in $\mathbb{R}^k$ by a bold letter {\bf v},  and the origin by $\bf O$. Among CAT$(0)$-spaces, the Euclidean space ${\mathbb R}^k$ and its discrete groups play a special role.
A \emph{crystallographic group} is a discrete, cocompact group $\Gamma$ of isometries of some $\mathbb{R}^k$. The simplest and most important of them, in view of Bieberbach's Theorem, are {\em Euclidean lattices}: i.e. 
free abelian crystallographic groups.
It is well known that a lattice must act by translations on $\mathbb{R}^k$ (see for instance \cite{farkas}); so, alternatively, a lattice $\mathcal{L}$ can be seen as the set of linear combinations with integer coefficients of $k$ independent vectors $\bf b_1, \ldots, \bf b_k$ (we will make no difference between a lattice and this representation). \\
A {\em basis} $\mathcal{B} = \lbrace \bf b_1,\ldots,\bf b_k\rbrace$ of a lattice $\mathcal{L}$ is a set of $k$ independent vectors that generate $\mathcal{L}$ as a group. There are many geometric invariants classically associated to a lattice $\mathcal{L}$, we will need three of them:\\
-- the \emph{covering radius}, which is defined as
$$\rho(\mathcal{L}) = \inf\left\lbrace r > 0 \text{ s.t. } \bigcup_{\bf v \in \mathcal{L}} \overline{B}({\bf v},r) = \mathbb{R}^k\right\rbrace;$$
-- the \emph{shortest generating radius}, that is
$$\lambda(\mathcal{L}) = \inf\lbrace r > 0 \text{ s.t. } \mathcal{L} \text{ contains } k \text{ independent vectors of length} \leq r\rbrace;$$
-- the \emph{shortest vector}, that is 
$$\tau(\mathcal{L}) = \inf \left\lbrace \Vert {\bf v} \Vert \text{ s.t. } {\bf v} \in \mathcal{L} \setminus \lbrace {\bf O} \rbrace \right\rbrace.$$

\noindent Notice that, by the triangle inequality, any lattice  $\mathcal{L}$ is $2\rho(\mathcal{L})$-cocompact. Moreover $\tau(\mathcal{L})$ coincides with the free-systole of $\mathcal{L}$.
It is always possible to find a basis $\mathcal{B}= \lbrace \bf b_1,\ldots,\bf b_k\rbrace$ of $\mathcal{L}$ such that $\tau(\mathcal{L}) = \Vert \bf b_1 \Vert \leq \cdots \leq \Vert b_k \Vert = \lambda(\mathcal{L})$;
this is called a \emph{shortest basis} of $\mathcal{L}$.
The shortest generating radius and the covering radius are related as follows:
\begin{equation}
	\label{eq-lattice-relation}
	\rho(\mathcal{L}) \leq \frac{\sqrt{k}}{2} \cdot \lambda(\mathcal{L}).
\end{equation}
In particular every lattice $\mathcal{L}$ is $(\sqrt{k}\cdot \lambda(\mathcal{L}))$-cocompact.
The proof of the Svarc-Milnor lemma in \cite[Theorem IV.B.23]{dLH00} gives the following estimate. We recall that given a generating set $S$ of a lattice $\mathcal{L}$ then $\ell_S({\bf v})$ denotes the length of the shortest word in the alphabet $S$ needed to write ${\bf v}$.
\begin{lemma}
	\label{lemma-lattice-svarc-milnor}
	Let $\mathcal{L}$ be a lattice of $\mathbb{R}^k$. Recall it is $(\sqrt{k}\cdot \lambda(\mathcal{L}))$-cocompact. Let $S = \overline{S}_{2\sqrt{k}\cdot \lambda(\mathcal{L})}({\bf O})$ as defined in \eqref{defsigma} and recall it is a generating set. Then $\ell_S({\bf v}) \leq \frac{\Vert {\bf v} \Vert}{\tau(\mathcal{L})} + 1$.
\end{lemma}

A property we will need in Section \ref{sec-convergence} is the existence of nice bases.

\begin{lemma}
	\label{lemma-LLL-basis}
	Let $\mathcal{L}$ be a lattice of $\mathbb{R}^k$. Then there exists a basis $\mathcal{B} = \lbrace \bf b_1,\ldots,\bf b_k\rbrace$ of $\mathcal{L}$ such that 
	\begin{itemize}
		\item[(i)] 	$\max_i \Vert {\bf b_i} \Vert \leq 2^{k-1} \cdot \lambda(\mathcal{L})$;
		\item[(ii)] $\angle({\bf b_i}, \sum_{j\neq i}\mathbb{R}{\bf b_j}) \geq \vartheta_k > 0$, where $\vartheta_k$ depends only on $k$.
	\end{itemize}
\end{lemma}
\noindent The notation $\sum_{j\neq i}\mathbb{R}{\bf b_j}$ means the $(k-1)$-dimensional subspace of $\mathbb{R}^k$ generated by $\lbrace {\bf b_j} \rbrace_{j\neq i}$.
\begin{proof}
	It follows by the combination of \cite[§1, and Proposition 1.12]{LLL82} and \cite[Theorem 5.1]{Bab86}.
\end{proof}

\noindent A set of vectors $\lbrace {\bf v_1}, \ldots, {\bf v_k} \rbrace$ satisfying condition (ii) of Lemma \ref{lemma-LLL-basis} for some $\vartheta > 0$ are called \emph{$\vartheta$-linearly independent}.\\
From any lattice we can always extract a sublattice with controlled geometry.

\begin{lemma}
	\label{lemma-sublattice-controlled}
	Let $\mathcal{L}$ be a lattice of $\mathbb{R}^k$. Then there exists a sublattice $\mathcal{L}'$ of $\mathcal{L}$ such that $\lambda(\mathcal{L}') = \lambda(\mathcal{L})$ and $\lambda(\mathcal{L}') \leq 2\cdot \tau(\mathcal{L}')$.
\end{lemma}
\begin{proof}
	We proceed recursively. If $\lambda(\mathcal{L}) \leq 2\cdot \tau(\mathcal{L})$ there is nothing to prove. Otherwise we take a shortest basis $\lbrace{\bf b_1}, \ldots, {\bf b_k}\rbrace$. In particular $\Vert {\bf b_1} \Vert = \tau(\mathcal{L})$. Let $\lambda_1 \in \mathbb{N}$ be such that $\lambda_1 \Vert {\bf b_1} \Vert \leq \lambda(\mathcal{L}) \leq 2 \lambda_1 \Vert {\bf b_1} \Vert$. We replace $\mathcal{L}$ with the sublattice $\mathcal{L}'$ generated by $\lbrace{\lambda_1 \bf b_1}, \ldots, {\bf b_k}\rbrace$. We have $\lambda(\mathcal{L}') = \lambda(\mathcal{L})$ and $\tau(\mathcal{L}') \geq \tau(\mathcal{L})$. We repeat the procedure up to arrive to a lattice satisfying the thesis. The algorithm stops because we need to do at most $N$ steps, where $N= \# \lbrace {\bf v} \in \mathcal{L} \text{ s.t. } \Vert {\bf v} \Vert < \frac{\lambda(\mathcal{L})}{2} \rbrace < +\infty$. 
\end{proof}

The content of the famous Bieberbach's Theorems can be stated as follows.
\begin{prop}[Bieberbach's Theorem]
	\label{prop-Bieberbach}${}$
	
	\noindent There exists $J(k)$, only depending on $k$, such that the following holds true.
	For any crystallographic group $\Gamma$ of $\mathbb{R}^k$ 
	the  subgroup ${\mathcal L}(\Gamma) = \Gamma \cap \textup{Transl}({\mathbb R}^k)$ is a normal subgroup of index at most $J(k)$, in particular a lattice. 
\end{prop}
\noindent Here $\textup{Transl}({\mathbb R}^k)$ denotes the normal subgroup of translations of $\textup{Isom}(\mathbb{R}^k)$.
The subgroup $\mathcal{L}(\Gamma)$ is called the \emph{maximal lattice} of $\Gamma$.
As in \cite{CS23} we will use the next fact in the study of collapsing sequences: it follows directly from Proposition \ref{prop-Bieberbach}.

\begin{lemma}[\textup{\cite[Lemma 2.7]{CS23}}]
	\label{lemma-bieber} 
	Let $\Gamma$ be  a  crystallographic group  of $\mathbb{R}^k$, and let  $0<r < \sqrt{2 \sin \left( \frac{\pi}{J(k) } \right)}$. If $g \in \Gamma$ moves all points of  $B_{\mathbb{R}^k} ({\bf O}, \frac{1}{r})$ less than $r$,    then $g$ is a translation.
\end{lemma}

\section{Margulis Lemma and almost abelian groups}
\label{sec-Margulis}
Our aim is to extend the Margulis' Lemma to totally disconnected group actions. In order to do so we need to introduce a new terminology.
\subsection{Almost nilpotent and almost abelian groups}
A topological group $A$ is \emph{almost nilpotent (resp. almost abelian)} if there exist a compact, open, normal subgroup $N\triangleleft A$ such that $A/N$ is discrete, finitely generated and virtually nilpotent (resp. virtually abelian). 
Observe that the existence of a compact, open subgroup $N$ implies that $A$ is locally compact.
Recall that for locally compact groups it holds the open mapping theorem, and in particular the classical isomorphism theorems provide topological isomorphisms.\\
The notion of being almost nilpotent or abelian passes to closed subgroups.
\begin{lemma}
	\label{lemma-subgroup-almost-abelian}
	Let $A$ be an almost nilpotent (resp. abelian) group. Then any closed subgroup $B < A$ is almost nilpotent (resp. abelian). 
\end{lemma}
\begin{proof} 
	Let $N$ be a compact, open, normal subgroup of $A$ such that $A/N$ is discrete, finitely generated and virtually nilpotent (resp. virtually abelian). The group $N' = N \cap B$ is open, compact and normal in $B$ because $B$ is closed. The group $B / N'$ is topologically isomorphic to a subgroup of $A/N$, thus it is discrete, finitely generated and virtually nilpotent (resp. virtually abelian).
\end{proof}

The rank of an almost abelian group is well defined.
\begin{lemma}
	\label{lemma-rank-almost-abelian}
	Let $A$ be an almost abelian group.	Then there exists $k\geq 0$ such that for every compact, open, normal subgroup $N \triangleleft A$, the group $A/N$ is discrete, finitely generated and virtually abelian of rank $k \geq 0$.	Such $k$ is called the \textup{rank} of $A$ and it is denoted by $\textup{rk}(A)$.
\end{lemma}
\begin{proof}
	Let $N_0$ be a compact, open, normal subgroup of $A$ such that $A/N_0$ is discrete, finitely generated and virtually abelian. Every virtually abelian, finitely generated group contains a free abelian subgroup of finite index of some rank $k$. Let $A_0 < A$ be the preimage under the projection map $\pi \colon A \to A/N_0$ of such a free abelian group. Observe that the index of $A_0$ in $A$ is finite and $A_0 / N_0 \cap A_0$ is topologically isomorphic to $\mathbb{Z}^k$.\\
	Let $N\triangleleft A$ be an arbitrary compact, open, normal subgroup.
	The index of $N\cap N_0$ in both the compact groups $N$ and $N_0$ is finite because it is an open subgroup. 
	So $A_0/N\cap N_0 \cap A_0$ is a finite extension of $A_0/N_0 \cap A_0 \cong \mathbb{Z}^k$. Therefore $A_0/N\cap N_0 \cap A_0$ is discrete, finitely generated and virtually abelian of rank $k$. Since $A_0/N \cap A_0$ is the quotient of $A_0/N\cap N_0 \cap A_0$ by a finite group, then it is still virtually abelian of same rank $k$. Finally $A/N$ contains $A_0/N \cap A_0$ as a finite index group, so it is still discrete, finitely generated and virtually abelian of rank $k$. 
\end{proof}

We recall that a closed subgroup $H$ of a locally compact group $G$ is said \emph{cocompact} if the topological space $G/H$ is compact. We will use the two definitions of cocompactness, this one and the one for group actions on metric spaces, in different contexts. We hope the reader will not be confused.
 Observe that any finite index subgroup is automatically cocompact, and that for discrete groups the two notions coincide. Moreover the following holds.
\begin{lemma} 
	\label{lemma-cocompactness-group}
	Let $G$ be a locally compact group and $H,K<G$ be closed subgroups.
	\begin{itemize}
		\item[(i)] If $K<H<G$ with $K$ cocompact in $H$ and $H$ cocompact in $G$ then $K$ is cocompact in $G$;
		\item[(ii)] if $H$ is cocompact in $G$ then $H\cap K$ is cocompact in $K$;
		\item[(iii)] if $H$ and $K$ are cocompact in $G$ then $H\cap K$ is cocompact in $G$.
	\end{itemize}
\end{lemma} 
\begin{proof}
	It is classical that if $G$ is locally compact then a closed subgroup $H<G$ is cocompact if and only if $G= C\cdot H$ for some compact subset $C$. Therefore (i) follows. Instead (ii) is straightforward from the definition and (iii) follows by (i) and (ii). Indeed $H\cap K$ is cocompact in $K$ by (ii), but $K$ is cocompact in $G$, so $H\cap K$ is cocompact in $G$ by (i).
\end{proof}

We end this introduction on almost abelian groups with a useful result.

\begin{lemma}
	\label{lemma-cocompact-rank}
	Let $A$ be an almost abelian group and let $B<A$ closed. Then $\textup{rk}(B) \leq \textup{rk}(A)$. Moreover $\textup{rk}(B) = \textup{rk}(A)$ if and only if $B$ is cocompact in $A$.
\end{lemma}

\begin{proof} 
	The first inequality is obvious from Lemma \ref{lemma-rank-almost-abelian}, so we pass to the second part of the statement. As in the beginning of the proof of Lemma \ref{lemma-rank-almost-abelian} we take a compact, open, normal subgroup $N_0$ of $A$ and a finite index subgroup $A_0 < A$ such that $A_0/N_0$ is topologically isomorphic to $\mathbb{Z}^k$. Let $\pi \colon A \to A/N_0$ be the projection map.\\
	Suppose first that $B$ is cocompact in $A$. In particular $B_0 = B \cap A_0$ is cocompact in $A_0$ by Lemma \ref{lemma-cocompactness-group}.(ii). The image of $B_0$ through $\pi$ is therefore cocompact in the image of $A_0$. Since the image of $A_0$ is $\mathbb{Z}^k$, then all its cocompact subgroups have finite index. Moreover every finite index subgroup of $\mathbb{Z}^k$ is isomorphic to $\mathbb{Z}^k$. Therefore $B_0/B_0 \cap N_0$ is topologically isomorphic to $\mathbb{Z}^k$. By Lemma \ref{lemma-rank-almost-abelian} we conclude that $\text{rk}(B) = k = \text{rk}(A)$.\\
	Suppose now $\text{rk}(B) = \text{rk}(A) = k$. Observe that also $B_0 = B\cap A_0$ is almost abelian of rank $k$, being a cocompact subgroup of $B$ since it has finite index. Lemma \ref{lemma-rank-almost-abelian} says that $B_0/N \cap B_0$ is discrete, finitely generated and virtually abelian of rank $k$, and moreover it is topologically isomorphic to a subgroup of $A_0/A_0 \cap N_0 \cong \mathbb{Z}^k$. In particular $B_0/B_0\cap N_0$ is free abelian of rank $k$, so it has finite index in $A_0/A_0\cap N_0$. We deduce that $\pi^{-1}(\pi(B_0))$ has finite index in $A_0$. Moreover $\pi^{-1}(\pi(B_0)) = B_0\cdot N_0$, so $B_0$ is cocompact in $\pi^{-1}(\pi(B_0))$. Since $\pi^{-1}(\pi(B_0))$ has finite index in $A_0$, so in $A$, we deduce that $B_0$, and so $B$, is cocompact in $A$.
\end{proof}

\subsection{The Margulis Lemma for totally disconnected group actions}

For discrete groups $\Gamma < \textup{Isom}(X)$ of a proper metric space $X$ the following version of the Margulis' Lemma, due to Breuillard-Green-Tao, holds.

\begin{prop}[\cite{BGT11}, Corollary 11.17]
	\label{theo-Margulis-BGT}
	Let $K \geq 1$. There exists $\varepsilon(K)$ such that the following is true. Let $X$ be any proper metric space and let $x\in X$ be a point such that $\textup{Cov}(\overline{B}(x,4), 1) \leq K$. Let $\Gamma<\textup{Isom}(X)$ be discrete. Then the almost stabilizer $\overline{\Gamma}_{\varepsilon'}(x)$ is virtually nilpotent for all $0\leq \varepsilon' \leq \varepsilon(K)$.
\end{prop}

Our aim is to extend this result to totally disconnected groups of isometries, and in particular to prove Theorem \ref{theo-intro-Margulis-TD}.
\noindent The proof of Proposition \ref{theo-Margulis-BGT} is based on the precise description of the structure of $K$-approximate subgroups of discrete groups developed in \cite{BGT11}.
\begin{defin}
	Let $K > 1$. A $K$-approximate subgroup of a group $G$ is a symmetric subset of $G$ containing the identity for which there exists a symmetric subset $X \subseteq G$ of cardinality at most $K$ such that $A \cdot A \subseteq X \cdot A$.
\end{defin}
The more general structure theory for open precompact approximate subgroups of locally compact groups was studied in \cite{Car15} and partially improved in \cite{TT21}. We report here a slightly modified version of \cite[Corollary 6.11]{TT21} which adapts better to our situation.
\begin{prop}
	\label{prop-TT}
	Let $K \geq 1$. Then there exists $K' =K'(K)$, depending only on $K$, such that the
	following holds. Assume $G$ is a totally disconnected, locally compact group generated by a compact symmetric set $S$ containing the identity. Let $A$ be an open precompact $K$-approximate subgroup of $G$ such that $S^{K'} \subseteq A$. Then there are subgroups
	$N',L \triangleleft G$ with $N' \subseteq A^6 \cap L$ such that $N'$ is compact, $L$ is open and has
	index $O_K(1)$ in $G$, and such that $L/N'$ is a Lie group of dimension at most
	$O_K(1)$.
\end{prop}
Here the notation $O_K(1)$ stays for a quantity which is bounded above by a constant depending only on $K$. It is the same notation used in \cite{BGT11}, \cite{Car15} and \cite{TT21}.
The proof is a straightforward modification of \cite[Corollary 6.11]{TT21}, where one uses $A$ in place of $S^n$ therein.
This gives the analogue of \cite[Corollary 11.2]{BGT11} for totally disconnected, locally compact groups that are compactly generated. 

\begin{prop}
	\label{prop-pre-Margulis}
	Let $K > 1$. Then there exists $K' = K'(K)$, depending only on $K$, such
	that the following holds true. Assume $G$ is a totally disconnected, locally compact group generated by a compact symmetric set $S$ containing the identity. Let $A$ be an open precompact $K$-approximate subgroup of $G$ such that $S^{K'} \subseteq A$. Then there exist $G_1 < G$ of index $\leq O_K(1)$ and a compact, open, normal subgroup $N\triangleleft G$, with $N\subseteq G_1$, such that $G_1/N$ is discrete, finitely generated and nilpotent of rank and step $\leq O_K(1)$. In particular $G$ is almost nilpotent.
\end{prop}
\begin{proof}
	Let $K'$ be the maximum between the constant of Proposition \ref{prop-TT} and the one provided by \cite[Corollary 11.2]{BGT11}.
	By Proposition \ref{prop-TT} we can find subgroups $N',L \triangleleft G$ such that
	\begin{itemize}
		\item[(a)] $L$ is open and it has finite index in $G$;
		\item[(b)] $N'$ is compact and $L /N'$ is a Lie group.
	\end{itemize}
	Condition (a) implies $L$ is also closed, so in particular it is totally disconnected. The quotient of a totally disconnected, locally compact group by a closed, normal subgroup is again totally disconnected. Therefore by (b) we deduce that $L /N'$ is discrete. In particular $N'$ is open in $L$, so also in $G$, because it is the continuous preimage of the open set $\lbrace 1\rbrace \subseteq L/N'$.
	Let $\pi \colon G \to G/N'$ be the projection map. The projections of the precompact sets $S$ and $A$ are precompact, hence finite. Clearly $\pi(S)$ generates $G/N'$, $\pi(A)$ is a $K$-approximate group in the sense of \cite{BGT11} and $\pi(S)^{K'} = \pi(S^{K'})\subseteq \pi(A)$. Therefore $G/N'$ is finitely generated and it satisfies the assumptions of \cite[Corollary 11.2]{BGT11}. Hence we can find $\overline{G_1} < G/N'$ of index at most $O_K(1)$ and a finite subgroup $\overline{N} \triangleleft G/N'$, with $\overline{N} < \overline{G_1}$, such that $\overline{G_1}/\overline{N}$ is nilpotent of rank and step $\leq O_K(1)$. In order to get the thesis it is enough to take $G_1 = \pi^{-1}(\overline{G_1})$ and $N = \pi^{-1}(\overline{N})$. Indeed $N$ is clearly normal, it is open because it contains $N'$ and it is compact because $N/N' \cong \overline{N}$ is finite.
\end{proof}

Finally we can argue as in \cite[Corollary 11.17]{BGT11} to get the announced version of the Margulis lemma.

\begin{theo}
	\label{theo-Margulis}
	Let $K \geq 1$. There exists $\varepsilon = \varepsilon(K)$ such that the following is true. Let $X$ be any proper metric space and let $x\in X$ be a point such that $\textup{Cov}(\overline{B}(x,4), 1) \leq K$. Let $G < \textup{Isom}(X)$ be closed and totally disconnected. Then the open almost stabilizer $G_{\varepsilon'}(x)$ is almost nilpotent for all $0\leq \varepsilon' \leq \varepsilon(K)$.
\end{theo}

\begin{proof}[Proof of Theorem \ref{theo-Margulis}]
	The set $A:= S_{2}(x)$ is open and it is precompact because $G$ is closed and $X$ is proper.
	Take $K$ points $x_i \in \overline{B}(x,4)$ such that $\overline{B}(x,4) \subseteq \bigcup_{i=1}^K \overline{B}(x_i,1)$. Suppose that for $i=1,\ldots, k \leq K$ there is $g_i\in S_{4}(x)$ such that $d(x_i,g_ix) \leq 1$, while this does not happen for $k+1\leq i \leq K$. For every $g\in A^2 \subseteq S_{4}(x)$ we find $i$ such that $d(gx,x_i) \leq 1$, so $d(gx,g_ix) \leq 2$, i.e. $g_i^{-1}g\in S_{2}(x)$. In other words $A^2 \subseteq S_{4}(x) \subseteq \bigcup_{i=1}^k g_i\cdot S_{2}(x)$. Therefore $A$ is an open, precompact, $K$-approximate subgroup of $G$. Now let $K'=K'(K)$ be the constant of Proposition \ref{prop-pre-Margulis} and take $\varepsilon(K) := \frac{2}{K'}$. The thesis holds for this choice of $\varepsilon(K)$.
\end{proof}

\noindent The conclusion of Theorem \ref{theo-Margulis} can be improved for totally disconnected, cocompact groups acting on geodesically complete, CAT$(0)$-spaces.
\begin{cor}
	\label{prop-Margulis-CAT}
	Let $P_0,r_0 > 0$. Then there exists $\varepsilon_0 = \varepsilon_0(P_0,r_0) > 0$ such that the following holds. Let $X$ be a proper, geodesically complete, $(P_0,r_0)$-packed, \textup{CAT}$(0)$-space. Let $G$ be a closed, totally disconnected, cocompact group of isometries of $X$. Then both the open and closed almost stabilizers $G_\varepsilon(x)$ and $\overline{G}_\varepsilon(x)$ are almost abelian, for all $x\in X$ and all $0\leq \varepsilon \leq \varepsilon_0$.
\end{cor}
\begin{proof}
	First of all $\text{Cov}(4,1)$ is bounded in terms of the packing constants by Proposition \ref{prop-packing} and \eqref{eq-pack-cov}. Therefore by Theorem \ref{theo-Margulis} we can find $\varepsilon_0$ depending only on $P_0,r_0$ such that every open stabilizer $G_\varepsilon(x)$ with $0\leq \varepsilon \leq \varepsilon_0$ is almost nilpotent. Moreover the same is true for the closed stabilizers because by Theorem \ref{theo-characterization-td}.(b) the sets $\overline{S}_r(x)$ are open and the proof of Theorem \ref{theo-Margulis} holds also for these open, compact, approximate subgroups.\\
	Let us prove that the open almost stabilizers above are actually almost abelian. The same proof will apply for the closed ones. The action of $G$ is by semisimple isometries by Theorem \ref{theo-characterization-td}.(ii). Let $N$ be a compact, open, normal subgroup of $G_\varepsilon(x)$ such that $G_\varepsilon(x) /N$ is discrete, finitely generated and virtually nilpotent. We can find a finite index subgroup $G' < G_\varepsilon(x)$ such that $G'/G' \cap N$ is nilpotent. The subset $\text{Min}(N)$ is closed, convex and non-empty because $N$ is compact (\cite[Corollary II.2.8]{BH09}): in particular it is a CAT$(0)$-space. The group $G/N$ acts by semisimple isometries (\cite[Proposition II.6.2]{BH09}) on the CAT$(0)$-space $\text{Min}(N)$. The kernel of this action is another normal subgroup $N_0$ of $G$ which is open (because it contains $N$) and compact (because it fixes some point of $X$). Moreover $\text{Min}(N) = \text{Min}(N_0)$. Therefore the group $G/N_0$ acts freely on $\text{Min}(N_0)$, in particular so does the group $G' / G'\cap N_0$ which is discrete, finitely generated and nilpotent. It is not restrictive, up to take a finite index subgroup of $G'$, to suppose that $G' / G'\cap N_0$ is torsion-free.	By \cite[Theorem 5.3.(i)]{FSY04} we conclude that $G'/G'\cap N_0$ is actually abelian. Therefore $G_\varepsilon(x)/N_0$ is discrete, finitely generated and virtually abelian, i.e. $G_\varepsilon(x)$ is almost abelian.
\end{proof}

\noindent We will often refer to the constant $\varepsilon_0=\varepsilon_0 (P_0,r_0)$ as the {\em Margulis' constant}.

\noindent From the proof of Theorem \ref{theo-Margulis} we can extract the following fact that will be used in Section \ref{sec-convergence}. For simplicity it will be presented for geodesically complete, CAT$(0)$-spaces, but a similar version holds for proper metric spaces.
\begin{cor}
	\label{cor-measure-powers}
	Let $P_0,r_0 > 0$. Then there exists $M_0 = M_0(P_0,r_0)$ such that the following holds.
	Let $X$ be a proper, geodesically complete, $(P_0,r_0)$-packed, \textup{CAT}$(0)$-space, let $G < \textup{Isom}(X)$ be closed and let $\mu$ be a left-invariant Haar measure of $G$.	
	Then
	$$\mu(\overline{S}_r(x)^p) \leq M_0^{p-1} \cdot \mu(\overline{S}_r(x))$$
	for every $x\in X$, every $r \geq 0$ and every $p\in \mathbb{N}$.
\end{cor}
\begin{proof}
	As in the proof of Theorem \ref{theo-Margulis} we find that $\overline{S}_r(x)^2 \subseteq \overline{S}_{2r}(x) \subseteq \bigcup_{i=1}^k g_i \overline{S}_r(x)$, for some $k\leq \textup{Cov}(2r, \frac{r}{2})$. By \eqref{eq-pack-cov} and Proposition \ref{prop-packing}.(i) we get $k \leq P_0(1+P_0)^{\frac{2r}{\min \lbrace \frac{r}{4},r_0 \rbrace} - 1} =: M_0$. By induction on $p$ we have
	$$\overline{S}_r(x)^p = \overline{S}_r(x)^{p-1} \cdot \overline{S}_r(x) \subseteq \bigcup_{i_1,\ldots,i_{p-1} = 1}^k g_{i_1}\cdots g_{i_{p-1}} \overline{S}_r(x).$$
	Therefore $\mu(\overline{S}_r(x)^p) \leq M_0^{p-1} \cdot \mu(\overline{S}_r(x))$.
\end{proof}

\subsection{Almost abelian groups acting on CAT$(0)$-spaces}	
\label{subsec-almost-abelian-CAT}
Corollary \ref{prop-Margulis-CAT} suggests that almost abelian groups acting on CAT$(0)$-spaces play a special role. In this section we will describe the geometric properties of such groups, generalizing \cite[Section 2.4]{CS23}. 
\begin{prop}
	\label{prop-trace-infinity-almost}
	Let $X$ be a proper \textup{CAT}$(0)$-space and let $A<\textup{Isom}(X)$ be closed, semisimple and almost abelian of rank $k$. Then $k$ is the unique integer for which the following holds.
	\begin{itemize}
		\item[(i)] There exists a closed, convex, $A$-invariant subset $C(A)$ of $X$ that splits as $W\times \mathbb{R}^k$ such that
		\begin{itemize}
			\item[(i.a)] each $g\in A$ preserves the product decomposition and acts as the identity on the first component;
			\item[(i.b)] the image $A_{\mathbb{R}^k}$ of $A$ under the projection $A \to \textup{Isom}(\mathbb{R}^k)$ is a crystallographic
			group;
			\item[(i.c)] if $S$ is a compact, symmetric, generating set for $A$ containing the identity, then there exists a finite subset $\Sigma \subseteq S^{4J(k)+2}$ whose projection $\Sigma_{\mathbb{R}^k}$ on \textup{Isom}$(\mathbb{R}^k)$ generates
			$\mathcal{L}(A_{\mathbb{R}^k})$, where $J(k)$ is the constant of Proposition \ref{prop-Bieberbach}.
		\end{itemize}
		\item[(ii)] There exists $\partial A \subseteq \partial X$, called the trace at infinity of $A$, such that
		\begin{itemize}
			\item[(ii.a)] $\partial A = \partial \textup{Conv}(Ax)$ for every $x\in X$;
			\item[(ii.b)] $\partial A$ is closed, convex, $A$-invariant and isometric to $\mathbb{S}^{k-1}$;
			\item[(ii.c)] if $g$ is an isometry of $X$ then $\partial (gAg^{-1}) = g \partial A$;
			\item[(ii.d)] if $B<A$ is closed then $\partial B \subseteq \partial A$ with equality if and only if $\textup{rk}(B) = \textup{rk}(A)$.
		\end{itemize}
	\end{itemize}
\end{prop}
\begin{proof}
	Let $N$ be an open, compact, normal subgroup of $A$. The set \textup{Min}$(N)$ is closed, convex, and $A$-invariant because $N$ is normal in $A$. It is also non-empty because $N$ is compact (cp. \cite[Corollary II.2.8]{BH09}). Let $N_0$ be the kernel of the action of $A$ on \textup{Min}$(N)$. It is again a compact, open, normal subgroup of $A$, because it contains $N$. Moreover $\text{Min}(N_0)\subseteq \text{Min}(N)$ because $N \subseteq N_0$, and by definition we have $\text{Min}(N_0)= \text{Min}(N)$. It follows that the induced action of $A/N_0$ on the CAT$(0)$-space $\text{Min}(N_0)$ is faithful. Moreover $A/N_0$ is a discrete, finitely generated, virtually abelian group of rank $k$ by Lemma \ref{lemma-rank-almost-abelian}, whose action on $\text{Min}(N_0)$ is by semisimple isometries (\cite[Proposition II.6.2.(4)]{BH09}). Then (i) follows by the classical statement for discrete groups (cp. \cite[Corollary II.7.2]{BH09}, \cite[Proposition 2.9 and Lemma 2.10]{CS23}). \\
	Suppose $h$ is another integer for which (i) holds. Let $N$ be the kernel of the action of $A$ on $C(A) = W \times \mathbb{R}^h$. It is clearly compact and normal. It is also open because the action of $A$ on $C(A)$ is discrete by (i.a) and (i.b). Therefore the group $A/N$ must be discrete, finitely generated and virtually abelian of rank $k$ by Lemma \ref{lemma-rank-almost-abelian}. However the group $A/N$ acts faithfully and cocompactly on $\mathbb{R}^h$, so it must have rank $h$, as a discrete, virtually abelian group. Therefore $k=h$.
	
	\noindent Let $C(A) = W \times \mathbb{R}^k$ be as in (i). Fix $w\in W$ and call $Z = \lbrace w \rbrace \times \mathbb{R}^k$: it is again an $A$-invariant, convex subset because of (i.a). We claim that $\partial Z = \partial \text{Conv}(Ax)$ for every $x\in X$. Fix $x\in X$ and set $R = d(x,Z)$. By $A$-invariance of $Z$ we have that $d(ax, Z) = R$ for every $a\in A$. So $Ax \subseteq \overline{B}(Z,R)$. The latter is a convex set, so $\text{Conv}(Ax) \subseteq \overline{B}(Z,R)$ and $\partial\text{Conv}(Ax) \subseteq \partial\overline{B}(Z,R)$. By convexity of the distance function every geodesic ray of $\overline{B}(Z,R)$ must be at constant distance from $Z$, i.e. parallel to it. This shows that $\partial \text{Conv}(Ax) \subseteq \partial\overline{B}(Z,R) = \partial Z$. For the other inclusion we call $z\in Z$ the projection of $x$ on $Z$. We fix any $\xi \in \partial Z$ and we consider the geodesic ray $[z,\xi]$. The action of $A$ on $Z$ is cocompact by (i.b) so we can find some $D>0$ and elements $a_j \in A$ such that $d(a_jz,[z,\xi]) \leq D$ and $d(z,a_jz)$ tends to $+\infty$.  They belong to $\text{Conv}(Ax)$. Since $d(x,a_jx) \geq d(z,a_jz) - 2R$ tends to $+\infty$, then by properness of $X$ the segments $[x,a_jx]$ subconverge to a geodesic ray $[x,\zeta]$. This ray is contained in $\text{Conv}(Ax)$ because each segment $[x,a_jx]$ is. Moreover $d(a_jx, [z,\xi]) \leq R + D$, so $\xi = \zeta$. This shows that $\partial\text{Conv}(Ax) = \partial Z$. We set $\partial A := \partial Z = \partial\text{Conv}(Ax)$ and we claim it satisfies (ii). The fact that $\partial A$ is closed, convex, $A$-invariant and isometric to $\mathbb{S}^{k-1}$ is clear since $\partial Z$ has these properties.
	In order to get (ii.c) we fix $x\in X$ and we use twice the characterization in (ii.a) to get
	$$\partial(gAg^{-1}) = \partial \text{Conv}((gAg^{-1}) gx) = g\partial\text{Conv}(Ax) = g\partial A.$$
	Moreover (ii.a) gives directly the inclusion in (ii.d). The equality case follows again from (ii.a) applied to some $z\in Z \cong \mathbb{R}^k$ on which both $A$ and $B$ act as discrete groups.
\end{proof}

The following is an immediate consequence of Proposition \ref{prop-trace-infinity-almost}.
\begin{cor}
	\label{cor-rank-dimension}
	Under the same assumptions of Proposition \ref{prop-trace-infinity-almost} we have:
	\begin{itemize}
		\item[(i)] $\textup{rk}(A) \leq \textup{dim}(X)$, for any notion of dimension of $X$ (topological, Hausdorff, geometric as in \cite{Kl99});
		\item[(ii)] for every $C(A) = W\times \mathbb{R}^k$ as in (i) and every subset $S \subseteq A$ whose projection on $\textup{Isom}(\mathbb{R}^k)$ generates a cocompact group of $\mathbb{R}^k$ then $\langle S \rangle < A$ has rank $k$.
	\end{itemize}
\end{cor}

\section{The splitting theorem}
\label{sec-splitting}

In this section we will prove the analogue of the splitting Theorem \cite[Theorem D]{CS23} or more precisely of \cite[Theorem 4.1]{CS23} for sufficiently collapsed totally disconnected group actions, i.e. for actions with free-systole sufficiently small. We will prove a slightly improved version that will be useful in Section \ref{sub-almost stabilizers}, the improvement is in items (v) and (vi). After our work on the Margulis Lemma and on the properties of almost abelian groups acting on CAT$(0)$-spaces we can essentially mimic the proof of \cite{CS23}.
\vspace{1mm}

\noindent To set the notation, recall that for any proper, geodesically complete, $(P_0, r_0)$-packed, CAT$(0)$-space $X$ we have the upper bound $\dim (X)\leq n_0 = P_0/2$ provided by Proposition \ref{prop-packing}  and the Margulis' constant $ \varepsilon_0$ (only depending  on $P_0, r_0$)    given by Proposition \ref{prop-Margulis-CAT}.
For a totally disconnected subgroup  $G < \textup{Isom}(X)$ 
recall the definition (\ref{defsigma}) of the subgroup  $\overline{G}_{r} (x) < G$ generated by  $\overline{S}_r (x)$ given in Section \ref{subsection-isometries}.

\begin{theo}
	\label{theo-splitting-weak}
	Given positive constants  $P_0,r_0,D_0$, 
	there exists a   function 
	$\sigma_{P_0,r_0,D_0}: (0,\varepsilon_0] \rightarrow  (0,\varepsilon_0]$   (depending only on the parameters $P_0,r_0,D_0$) such that the following holds. 
	Let $X$ be a proper, geodesically complete, $(P_0,r_0)$-packed, $\textup{CAT}(0)$-space,   and $G < \textup{Isom}(X)$ be closed, totally disconnected and $D_0$-cocompact. 
	For any chosen $\varepsilon \in (0, \varepsilon_0]$, if $\textup{sys}^\diamond(G,X) \leq \sigma_{P_0,r_0,D_0}(\varepsilon)$ then:
	\begin{itemize} 
		\item[(i)] the space $X$ splits isometrically  as $Y \times \mathbb{R}^k$, with $k\geq 1$, and this splitting is $G$-invariant;
		\item[(ii)] there exists ${\varepsilon^\ast} \in (\sigma_{P_0,r_0,D_0} (\varepsilon), \varepsilon)$ 
		such that the rank of the almost abelian subgroups $\overline{G}_{{\varepsilon^\ast}}(x)$ is  exactly $k$, for all $x\in X$; 
		\item[(iii)]   the traces at infinity $\partial \overline{G}_{{\varepsilon^\ast}}(x)$ equal  the boundary ${\mathbb S}^{k-1}$ of  the convex subsets  $\lbrace y \rbrace \times \mathbb{R}^k$, for all $x \in X$ and all $y\in Y$;
		\item[(iv)] for every $x\in X$ there exists $y\in Y$ such that $\overline{G}_{\varepsilon^\ast}(x)$ preserves $\lbrace y \rbrace \times \mathbb{R}^k$. The closure of the projection of $\overline{G}_{\varepsilon^\ast}(x)$ on $\textup{Isom}(Y)$ is compact;
		\item[(v)]  the projection 
		of $\overline{G}_{\varepsilon^\ast}(x)$ on $\textup{Isom}(\mathbb{R}^k)$ is a crystallographic group, whose maximal lattice $\mathcal{L}_{\varepsilon^*}(x)$ satisfies $\lambda(\mathcal{L}_{\varepsilon^*}(x)) \leq \varepsilon^*/2\sqrt{n_0}$ and is $(\varepsilon^*/2)$-cocompact;
		\item[(vi)] the groups $\overline{G}_{\varepsilon^*}(x)$ are almost commensurated for all $x\in X$.
	\end{itemize}		
\end{theo}

\noindent Here, by {\em $G$-invariant splitting} we mean that every isometry of $G$ preserves the product decomposition. By \cite[Proposition I.5.3.(4)]{BH09} we can see $G$ as a subgroup of $\textup{Isom}(Y) \times \textup{Isom}(\mathbb{R}^k)$. In particular it is meaningful to talk about the projection of $\overline{G}_{\varepsilon^*}(x)$ on $\text{Isom}(Y)$ and $\text{Isom}(\mathbb{R}^k)$.
We call the integer  $1 \leq k\leq n_0$  the {\emph{${\varepsilon^\ast}$-splitting rank}} of $X$. The fact that $k$ is at most $n_0$ follows by Corollary \ref{cor-rank-dimension}.  Observe that we do not need any unimodularity assumption on $X$. However we will say even more on the projection of $\overline{G}_{\varepsilon^*}(x)$ on $\text{Isom}(Y)$ in Section \ref{sub-almost stabilizers} in case $G$ is unimodular. The almost commensurability condition of item (vi) is introduced below: it plays a fundamental role in the proof of Theorem \ref{theo-splitting-weak}.

\subsection{Almost commensurability}
Let $G$ be a topological group. We say that two subgroups $H,K <G$ are \emph{almost commensurable} if $H\cap K$ is cocompact in both $H$ and $K$. A subgroup $H<G$ is \emph{almost commensurated} in $G$ if $H$ and $gHg^{-1}$ are almost commensurable for every $g\in G$. Observe that if $G$ is discrete then this notion coincides with the classical notion of commensurated subgroup. 
\begin{lemma}
	\label{lemma-commensurability-transitive}
	Suppose $S$ is a generating set for $G$. If $H$ and $gHg^{-1}$ are almost commensurable for all $g\in S$ then $H$ is almost commensurated.
\end{lemma}
\begin{proof}
	Let $g=g_1\cdots g_k$ be an element of $G$, with $g_i \in S$. We prove that $H$ and $gHg^{-1}$ are almost commensurable by induction on $k$. The case $k=1$ is the assumption. In the general case, by inductive assumption we have
	\begin{itemize}
		\item[(a)] $g_2\cdots g_k H g_k^{-1}\cdots g_2^{-1} \cap H$ is cocompact in both $g_2\cdots g_k H g_k^{-1}\cdots g_2^{-1}$ and $H$;
		\item[(b)] $g_1Hg_1^{-1} \cap H$ is cocompact in both $g_1Hg_1^{-1}$ and $H$.
	\end{itemize}
	Conjugating (a) by $g_1$ we also get
	\begin{itemize}
		\item[(c)] $g H g^{-1} \cap g_1Hg_1^{-1}$ is cocompact in both $g H g^{-1}$ and $g_1Hg_1^{-1}$.
	\end{itemize}
	Combining (b) and (c) we conclude that $g H g^{-1} \cap g_1Hg_1^{-1} \cap H$ is cocompact in $g_1Hg_1^{-1} \cap H$ and so in $H$, by Lemma \ref{lemma-cocompactness-group}. A fortiori $g H g^{-1} \cap H$ is cocompact in $H$. In a similar way, using again Lemma \ref{lemma-cocompactness-group}, we also have that $g H g^{-1} \cap H$ is cocompact in $gHg^{-1}$.
\end{proof}

Lemma \ref{lemma-cocompactness-group} implies the following property.
\begin{lemma}
	\label{lemma-commensurability-hereditary}
	Let $G$ be a locally compact group and let $K<H < G$ be closed subgroups, with $K$ cocompact in $H$. Then $H$ is almost commensurated if and only if $K$ is.
\end{lemma}

We will show how an almost abelian subgroup which is almost commensurated provides an associated metric splitting of the space, generalizing the discrete case (cp. \cite[Proposition 4.5]{CS23}, \cite{CM19}). We just need to recall an additional definition. Given $Z\subseteq \partial X$ we say that a subset $Y\subseteq X$ is $Z$-boundary-minimal if it is closed, convex, $\partial Y = Z$ and $Y$ is minimal with these properties. The union of all the $Z$-boundary-minimal sets is denoted by \textup{Bd-Min}$(Z)$.

\begin{lemma}[\textup{\cite[Lemma 4.4]{CS23}, \cite[Proposition 3.6]{CM09b}}]
	\label{lemma-split-sphere}
	Let $X$ be a proper $\textup{CAT}(0)$-space and let $Z$ be a closed, convex subset of  $\partial X$ which is isometric to $\mathbb{S}^{k-1}$. Then each $Z$-boundary-minimal subset of $X$ is isometric to $\mathbb{R}^k$ and \textup{Bd-Min}$(Z)$ is a closed, convex subset of $X$ which splits isometrically as $Y\times \mathbb{R}^k$. Moreover $Z$ coincides with the boundary at infinity of all the slices $\lbrace y \rbrace \times \mathbb{R}^k$, for $y \in Y$.
\end{lemma}

We can now generalize \cite[Proposition 4.5]{CS23}.

\begin{prop}
	\label{prop-commensurated-splitting}
	Let $X$ be a proper, geodesically complete, $\textup{CAT}(0)$-space,   and $G < \textup{Isom}(X)$ be closed, totally disconnected and cocompact.
	If $A < G$ is a semisimple, almost abelian, almost commensurated subgroup of $G$ of rank $k$ then we have $X=\textup{Bd-Min}(\partial A)$ and $X$ splits isometrically and $G$-invariantly as $Y \times \mathbb{R}^k$. Moreover,  the projection of $A$ on $\textup{Isom}(\mathbb{R}^k)$ is a crystallographic group and the closure of the projection of $A$ on $\textup{Isom}(Y)$ is compact.	The splitting $X=Y \times \mathbb{R}^k$ satisfies the following properties:
	\begin{itemize}
		\item[(i)] the trace at infinity $\partial A$ is $G$-invariant and coincides with the boundary of each slice $\lbrace y \rbrace \times \mathbb{R}^k$, for all $y \in Y$;
		\item[(ii)] if  $A' < A$ is another almost abelian, almost commensurated subgroup of $G$ of rank $k'$ then the splittings $X = Y \times \mathbb{R}^k$ and $X= Y' \times \mathbb{R}^{k'}$ associated respectively to $A$ and $A'$ are compatible, i.e. $Y'$ is isometric to $Y \times \mathbb{R}^{k-k'}$.
	\end{itemize}
\end{prop}

\begin{proof}
	The trace at infinity $\partial A$ of $A$ is closed, convex and isometric to $\mathbb{S}^{k-1}$, by Proposition \ref{prop-trace-infinity-almost}.(ii.b). We claim it is $G$-invariant. If $g\in G$ then $A\cap gAg^{-1}$ is closed and cocompact in both $A$ and $gAg^{-1}$. Moreover $A\cap gAg^{-1}$ is almost abelian by Lemma \ref{lemma-subgroup-almost-abelian}, and $\text{rk}(A\cap gAg^{-1})= \text{rk}(A) = \text{rk}(gAg^{-1}) = k$ by Lemma \ref{lemma-cocompact-rank}. 
	By (ii.c) and (ii.d) of Proposition \ref{prop-trace-infinity-almost} we have
	$$\partial A = \partial (A\cap gAg^{-1}) = \partial (gAg^{-1}) = g\partial A,$$
	for all $g\in G$.
	By Lemma \ref{lemma-split-sphere} applied to $Z=\partial A$ we deduce that $\text{Bd-Min}(\partial A)$ is a closed, convex subset of $X$ which splits isometrically as $Y\times \mathbb{R}^k$, and that  $\partial A$ coincides with the boundary at infinity of all sets $\lbrace y \rbrace \times \mathbb{R}^k$. Each element of $G$ sends a $\partial A$-boundary-minimal subset into a $\partial A$-boundary-minimal subset because $\partial A$ is $G$-invariant, therefore $\text{Bd-Min}(\partial A)$ itself is $G$-invariant. Since $G$ is  cocompact,  the action of $G$ on  $X$ is minimal (\cite[Lemma 3.13]{CM09b}). We deduce that $X=\text{Bd-Min}(\partial A)$, and so $X$ splits isometrically and $G$-invariantly as $Y\times \mathbb{R}^k$, which proves the first assertion and (i).\\
	The fact that the projection of $A$ on  $\textup{Isom}(\mathbb{R}^k)$ is a crystallographic group follows from Proposition \ref{prop-trace-infinity-almost}. Indeed any set $C(A) = W\times \mathbb{R}^k$ as in \ref{prop-trace-infinity-almost}.(i) is compatible with the splitting we just proved in the sense that $W\subset Y$, because $W \times \mathbb{R}^k$ is union of $\partial A$-boundary minimal sets as follows by (i.a).
	Moreover $A$ acts as the identity on $W$. Therefore the projection of $A$ on $Y$ is precompact since it fixes a point.\\
	Suppose now to have another almost abelian subgroup $A'<A$ of rank $k'$ which is commensurated in $G$. 
	Let $X = Y' \times \mathbb{R}^{k'}$ be the splitting associated to $A'$. Let $x\in X$ be a point and write it as $(y,{\bf v}) \in Y \times \mathbb{R}^k$ and $(y',{\bf v'}) \in Y'\times \mathbb{R}^{k'}$. Then, by the first part of the proof and by Proposition \ref{prop-trace-infinity-almost} we have
	$$\partial (\lbrace y' \rbrace \times \mathbb{R}^{k'}) = \partial A' \subseteq \partial A = \partial (\lbrace y \rbrace \times \mathbb{R}^{k}).$$
	It follows that $\lbrace y' \rbrace \times \mathbb{R}^{k'} \subseteq \lbrace y \rbrace \times \mathbb{R}^{k}$, so the parallel slices associated to $A'$ are contained in the parallel slices associated to $A$. Decomposing $\mathbb{R}^k$ as the orhogonal sum of $\mathbb{R}^{k'}$ and $\mathbb{R}^{k - k'}$, we also deduce that  the sets $\lbrace y \rbrace \times \mathbb{R}^{k - k'}$ are parallel for all $y\in Y$ (since the slices of $A'$ are all parallel). Therefore, $X$ is also isometric to $(Y \times \mathbb{R}^{k - k'}) \times \mathbb{R}^{k'}$, which implies that $Y'$ is isometric to $Y \times \mathbb{R}^{k - k'}$ and proves (ii).
\end{proof}

\begin{obs}
	The same proof shows that the statement holds for proper  \textup{CAT}$(0)$-spaces, not necessarily geodesically complete, and for closed groups of isometries acting minimally (see \cite{CM09b} for more details on this property).
\end{obs}

\subsection{Proof of Theorem \ref{theo-splitting-weak}}
The main step consists in showing an almost abelian subgroup of rank $k \geq 1$ which is almost commensurated in $G$. Observe that it will be automatically semisimple by Theorem \ref{theo-characterization-td}.(ii).\\
Recall the constant $J(k)$ given by Proposition \ref{prop-Bieberbach}, and  define 
$$J_0 := \max_{k\in \lbrace 0,\ldots,n_0\rbrace}J(k) +1$$
which also clearly  depends only on $n_0$, so ultimately only on $P_0$. Recall also that for every $x\in X$ the set $\overline{S}_{2D_0}(x)$ generates the whole $G$, by the discussion in Section \ref{subsection-isometries} and Theorem \ref{theo-characterization-td}.(iii).

\noindent We fix $0<\varepsilon \leq \varepsilon_0$ as in the assumptions of Theorem \ref{theo-splitting-weak}, and we define  inductively the sequence of subgroups  $\overline{G}_{\varepsilon_i} (x)$ associated to positive numbers 
$$\varepsilon_1 := \varepsilon  > \varepsilon_2 > \ldots > \varepsilon_{3n_0 + 1} > 0$$ as follows:\\
-- first, we  apply Proposition \ref{lemma-Sylvain} to $\varepsilon=\varepsilon_1$ and $R=2D_0$ to obtain a smaller $\delta_2 := \delta(P_0,r_0,2D_0,\varepsilon_1)$, and  we set $\varepsilon_2 := \delta_2/8\sqrt{n_0}J_0$;\\
-- then, we define inductively $\delta_{i+1} := \delta(P_0,r_0,2D_0,\varepsilon_i)$ by repeatedly applying Proposition \ref{lemma-Sylvain} to $\varepsilon_i$ and $R=2D_0$, and we set $\varepsilon_{i+1} = \delta_{i+1}/8\sqrt{n_0}J_0$. \\
Notice that, by construction, each $\varepsilon_i$ depends only on $P_0, r_0, D_0$ and $\varepsilon$. By Proposition \ref{prop-Margulis-CAT}, the subgroups  $\overline{G}_{\varepsilon_i} (x)$ form a decreasing sequence of almost abelian, semisimple subgroups for every $x\in X$.\\
In \cite{CS23} we defined $\varepsilon_{i+1} = \delta_{i+1}/4J_0$, here we are using the formula above in order to get the improved version of (v) in Theorem \ref{theo-splitting-weak}. Another difference, for the same scope, is that we define the sequence up to $3n_0 + 1$ instead of $2n_0 + 1$.

\vspace{1mm}
\noindent We set $\sigma_{P_0,r_0,D_0} (\varepsilon):= \varepsilon_{3n_0 + 1}$ and we will  show that this is the function of $\varepsilon$
for which Theorem \ref{theo-splitting-weak} holds;  it clearly depends only on $P_0,r_0,D_0$ and $\varepsilon$. In what follows, we will write  for short  $\sigma  := \sigma_{P_0,r_0,D_0} (\varepsilon)$.

\begin{lemma}
\label{lemma-constant-rank}
If 
$\textup{rk} (\overline{G}_{\sigma}( x)) \geq 1 $
then there exists
$i\in \lbrace 1,\ldots, 3n_0 - 2 \rbrace$ such that
$\textup{rk}(\overline{G}_{\varepsilon_{i}}( x)) = \textup{rk}(\overline{G}_{\varepsilon_{i+1}}(x )) = \textup{rk}(\overline{G}_{\varepsilon_{i+2}}(x )) = \textup{rk}(\overline{G}_{\varepsilon_{i+3}}(x )) \geq 1$.
\end{lemma}

\begin{proof}
The groups $\overline{G}_{\varepsilon_i}(x)$ are semisimple and almost abelian. Moreover by Lemma \ref{lemma-cocompact-rank} they satisfy
$\textup{rk}(\overline{G}_{\varepsilon_i}(x))\geq 1$,   for all $1\leq i \leq 3n_0 +1 $, since they  contain $\overline{G}_{\sigma}(x)$.
By Corollary \ref{cor-rank-dimension},    the rank of each $\overline{G}_{\varepsilon_i}(x)$ cannot exceed the dimension of $X$, which is at most $n_0$. Since, again by Lemma \ref{lemma-cocompact-rank}, the rank decreases as $i$ increases we conclude that  for some  $1\leq i\leq 3n_0 - 2$ we have  $\textup{rk}(\overline{G}_{\varepsilon_{i}}( x)) = \textup{rk}(\overline{G}_{\varepsilon_{i+1}}(x )) = \textup{rk}(\overline{G}_{\varepsilon_{i+2}}(x )) = \textup{rk}(\overline{G}_{\varepsilon_{i+3}}(x )) \geq 1$.
\end{proof}

\begin{prop}
\label{prop-commensurated}
If $\textup{rk} (\overline{G}_{\sigma}(x)) \geq 1 $ then there exists $i \in \lbrace 1,\ldots, 3n_0 - 2 \rbrace$ such that $\overline{G}_{\varepsilon_{i+j}}(x)$ is almost commensurated for all $j\in \lbrace 0,1,2,3\rbrace$, and all these subgroups have the same rank $k\geq 1$.
\end{prop}

\begin{proof}
Consider the almost abelian groups $\overline{G}_{\varepsilon_{i+3}}(x) < \overline{G}_{\varepsilon_{i+2}}(x) < \overline{G}_{\varepsilon_{i+1}}(x) < \overline{G}_{\varepsilon_{i}}(x)$ which have the same rank $k\geq 1$, 
given by Lemma \ref{lemma-constant-rank}, for some 
$1 \leq i \leq 3n_0 - 2$.
Let $C_{i+3}=W_{i+3} \times {\mathbb R}^{k}$ be a $\overline{G}_{\varepsilon_{i+3}}(x)$-invariant, convex subset of $X$ as in Proposition  \ref{prop-trace-infinity-almost}.(i), applied to $\overline{G}_{\varepsilon_{i+3}}(x)$, 
and  call $\hat{G}_{i+3}$  the image of $\overline{G}_{\varepsilon_{i+3}}(x)$ under the projection  $\pi_{i+3}:\overline{G}_{\varepsilon_{i+3}}(x) \rightarrow \text{Isom}(\mathbb{R}^{k})$ on the second factor of $C_{i+3}$.
Finally, denote by  ${\mathcal L}_{i+3}$ the  maximal Euclidean lattice of the crystallograhic group $\hat{G}_{i+3}$.
\noindent By Proposition \ref{prop-trace-infinity-almost}(i.c) we can find a finite subset $\Sigma$ of 
$\overline{S}_{\varepsilon_{i+3}}(x)^{4J_0} \subseteq \overline{S}_{4J_0\cdot \varepsilon_{i+3}}(x)$	
whose projection $\Sigma_{\mathbb{R}^k}=\pi_{i+3}(\Sigma)$ on  Isom$(\mathbb{R}^k)$ generates the lattice  ${\mathcal L}_{i+3}$.  
In particular every non-trivial element of $\Sigma$ is hyperbolic.\\
Moreover,  by the definition of $\varepsilon_{i+3}  = \delta_{i+3} /8\sqrt{n_0} J_0 $,
the following holds:
\vspace{-2mm}

\begin{center}
{\em  $ \forall g\in  \Sigma$ 
and $\forall h\in \overline{S}_{2D_0}(x)$  there exists $m >0$  
such that    $h^{-1}g^m h    \in \overline{S}_{\varepsilon_{i+2}}(x)$
}
\end{center}

\vspace{-2mm}
\noindent (in fact,  $d(x,gx) < \delta_{i+3}=\delta(P_0,r_0,2D_0, {\varepsilon_{i+2}})$ for every $g \in \Sigma$,  so by Proposition \ref{lemma-Sylvain} there exists $m> 0$ such that \nolinebreak
$ d(x, h^{-1}g^m hx) = d(hx,g^mhx) \leq \varepsilon_{i+2}$, that is  $h^{-1}g^m h \in \overline{S}_{\varepsilon_{i+2}}(x)$). \\
Theorem \ref{theo-characterization-td}.(c) implies that the cardinality of the orbit $\overline{S}_{2D_0}(x) x$ is finite, so we can fix $h_1,\ldots,h_\ell \in \overline{S}_{2D_0}(x)$ such that for every $h\in \overline{S}_{2D_0}(x)$ there exists $j\in \lbrace 1,\ldots, \ell\rbrace$ with $hx=h_jx$. The discussion above says that for every $j \in \lbrace 1,\ldots,\ell\rbrace$ and every $g\in \Sigma$ there exists an integer $m_{j,g}>0$ such that $h_j^{-1}g^{m_{j,g}}h_j \in \overline{S}_{\varepsilon_{i+2}}(x)$. Moreover if $hx=h_jx$ then also $h^{-1}g^{m_{j,g}}h \in \overline{S}_{\varepsilon_{i+2}}(x)$, because
\begin{equation*}
	\begin{aligned}
		d(h^{-1}g^{m_{j,g}}hx, x) = d(g^{m_{j,g}}hx, hx) &= d(g^{m_{j,g}}h_jx, h_jx) \\
		&= d(h^{-1}_jg^{m_{j,g}}h_jx, x) \leq \varepsilon_{i+2}.
	\end{aligned}
\end{equation*}
Therefore, denoting by $M = \prod_{j,g} m_{j,g} > 0$, we have that $hg^Mh^{-1} \in {\overline{G}_{\varepsilon_{i+2}}}(x)$ for all $g\in \Sigma$ and all $h\in \overline{S}_{2D_0}(x)$. Observe that $M$ is finite because $\Sigma$ is finite and $j\in \lbrace 1,\ldots,\ell\rbrace$.\\
Let $A < G_{\varepsilon_{i+3}} (x)$ 
be the subgroup  generated by the subset $\lbrace g^M  \,\;|\;\, g \in \Sigma \rbrace$. 
We claim that $A$ is almost commensurated in $G$.
Actually, notice first that $A$ is an almost abelian group of rank $k$ because of Corollary \ref{cor-rank-dimension}, since its projection $\pi_{i+3}(A)$ is a subgroup of finite index of the lattice ${\mathcal L}_{i+3}$ of  ${\mathbb{R}^k}$.
Therefore, for all $h\in \overline{S}_{2D_0}(x)$, the almost abelian group $hAh^{-1}$ has also rank $k$, and is contained in the almost abelian group $\overline{G}_{\varepsilon_{i+2}}(x)$ of same rank. This implies, again by Lemma \ref{lemma-cocompact-rank}, that $A$ and $hAh^{-1}$ are both cocompact in $\overline{G}_{\varepsilon_{i+2}}(x)$. In particular $A\cap  hAh^{-1} $ is cocompact in both $A$ and  $hAh^{-1}$ by Lemma \ref{lemma-cocompactness-group}. Hence $A$ and  $hAh^{-1}$ are almost commensurable  for every  $h\in \overline{S}_{2D_0}(x)$.
Lemma \ref{lemma-commensurability-transitive} shows that therefore $A$ is almost commensurated in $G$. Since $A<\overline{G}_{\varepsilon_{i+j}}(x)$ for  $j=0,\ldots,3$ and all these groups have same rank, we deduce that $A$ is cocompact also in $\overline{G}_{\varepsilon_{i +j}}(x)$, again by Lemma \ref{lemma-cocompact-rank}. The conclusion follows by Lemma \ref{lemma-commensurability-hereditary}.
\end{proof}

Putting the ingredients all together we can give the

\begin{proof}[Proof of Theorem \ref{theo-splitting-weak}]  
Since $\textup{sys}^\diamond(G,X)\leq \sigma $, then there exists $x_0 \in X $ with $\textup{sys}^\diamond(G,x_0)\leq \sigma $; in particular, $\overline{G}_{\sigma}(x_0)$   contains a hyperbolic isometry, hence  $\text{rk}(\overline{G}_{\sigma}(x_0)) \geq 1$. This follows by the characterization of the rank given in Proposition \ref{prop-trace-infinity-almost}. Then we can apply Proposition \ref{prop-commensurated} to find
$1 \leq i \leq 3n_0 - 2$ such that the groups $\overline{G}_{\varepsilon_{i+j}}(x_0)$, $j=0,\ldots,3$, have all rank $k\geq 1$ and are all almost commensurated.
Proposition \ref{prop-commensurated-splitting} now implies that $X$ splits isometrically and $G$-invariantly as $Y\times \mathbb{R}^k$, proving (i). 
Moreover, we know that $\partial \overline{G}_{\varepsilon_{i+j}}(x)$ coincides with the boundary at infinity of each slice $\lbrace y \rbrace \times \mathbb{R}^k$, for $j=0,\ldots,3$. Let us call $Z \subseteq \partial X$ this common trace at infinity. \\
We show that the statements (ii)-(vi) hold for 
${\varepsilon^\ast}=\varepsilon_{i+1} \in (\sigma, \varepsilon)$. 
We start proving them for any $x\in X$ such that   $d(x, x_0) \leq D_0$. Denote by $k_x$ the rank of $\overline{G}_{\varepsilon_{i+1}}(x)$.
 Let $\Sigma \subseteq \overline{S}_{\varepsilon_{i+1}}^{4J(k) + 2}(x) \subseteq \overline{S}_{4J_0\cdot \varepsilon_{i+1}}(x)$ be the finite subset provided by Proposition \ref{prop-trace-infinity-almost}.(i.c). Observe that by Corollary \ref{cor-rank-dimension} the group generated by $\Sigma$ has rank $k_x$. By Proposition \ref{lemma-Sylvain} and by definition of $\varepsilon_{i+1}$, for every $g\in \Sigma$ there  exists $m_g >0$ such that $d(x_0,g^{m_g}x_0)\leq \varepsilon_{i}$. Let $A$ be the group generated by $\lbrace g^{m_g} \text{ s.t. } g\in \Sigma\rbrace$. The rank of $A$ is still $k_x$ by Corollary \ref{cor-rank-dimension}. Moreover $A< \overline{G}_{\varepsilon_{i}}(x_0)$, therefore by Lemma \ref{lemma-cocompact-rank}
 $$k_x = \text{rk}(\overline{G}_{\varepsilon_{i+1}}(x)) = \text{rk}(B) \leq \text{rk}(\overline{G}_{\varepsilon_{i}}(x_0)) = k.$$
By Proposition \ref{prop-trace-infinity-almost}  we also have 
$\partial \overline{G}_{\varepsilon_{i+1}}(x) = \partial A \subseteq \partial \overline{G}_{\varepsilon_{i}}(x_0) = Z$.
Reversing the roles of $x$ and $x_0$ and starting from $\overline{G}_{\varepsilon_{i+3}}(x_0)$ we obtain the estimate
$$k=\text{rk}(\overline{G}_{\varepsilon_{i+3}}(x_0)) \leq \text{rk}(\overline{G}_{\varepsilon_{i+2}}(x)) \leq \text{rk}(\overline{G}_{\varepsilon_{i+1}}(x)) = k_x$$
and $Z= \partial \overline{G}_{\varepsilon_{i+3}}(x_0)\subseteq \partial \overline{G}_{\varepsilon_{i+2}}(x) \subseteq \partial \overline{G}_{\varepsilon_{i+1}}(x)$.
This proves proves (ii) and (iii) in this case. Observe that it also shows that for all $x$ with $d(x,x_0) \leq D_0$ we have $\text{rk}(\overline{G}_{\varepsilon_{i+2}}(x)) = \text{rk}(\overline{G}_{\varepsilon_{i+1}}(x)) = k$ and $\partial \overline{G}_{\varepsilon_{i+2}}(x) = \partial \overline{G}_{\varepsilon_{i+1}}(x) = Z$.\\
The same group $A$ above is cocompact in both $\overline{G}_{\varepsilon_{i +1}}(x)$ and $\overline{G}_{\varepsilon_{i}}(x_0)$ by Lemma \ref{lemma-cocompact-rank}. A double application of Lemma \ref{lemma-commensurability-hereditary} gives that $\overline{G}_{\varepsilon_{i +1}}(x)$ is almost commensurated, that is (vi). Moreover Proposition \ref{prop-commensurated-splitting} implies that the closure of the projection of $\overline{G}_{\varepsilon_{i}}(x_0)$ on $\text{Isom}(Y)$ is compact, so it is the closure of the projection of $A$ and automatically also the closure of the projection of $\overline{G}_{\varepsilon_{i + 1}}(x)$.
Since the the closure of the projection of $\overline{G}_{\varepsilon_{i+1}}(x)$ on $\text{Isom}(Y)$ is compact then it fix some point $y\in Y$ (cp. \cite[Proposition II.2.8]{BH09}). Hence $\overline{G}_{\varepsilon_{i+1}}(x)$ preserves $\lbrace y \rbrace \times \mathbb{R}^k$. This shows (iv).\\
As in the proof of Proposition \ref{prop-commensurated-splitting} we have that any splitting $W\times \mathbb{R}^k$ as in Proposition \ref{prop-trace-infinity-almost} associated to $\overline{G}_{\varepsilon_{i+2}}(x)$ is compatible with the splitting $X=Y\times \mathbb{R}^k$, in the sense that $W\subseteq Y$. This is just because $\partial \overline{G}_{\varepsilon_{i+2}}(x) = Z$.
Therefore the projection of $\overline{G}_{\varepsilon_{i+2}}(x)$ on $\text{Isom}(\mathbb{R}^k)$ is a crystallographic group by Proposition \ref{prop-trace-infinity-almost}.(i.b). By (i.c) of the same proposition we have that the maximal lattice $\mathcal{L}_{i+2}$ of this projection is generated by elements of $\overline{S}_{\varepsilon_{i+2}}(x)^{4J_0} \subseteq \overline{S}_{4J_0\cdot\varepsilon_{i+2}}(x)$. Observe that $4J_0\cdot\varepsilon_{i+2} \leq \varepsilon_{i+1}/2\sqrt{n_0}$, by definition.
The same argument shows that the projection of $\overline{G}_{\varepsilon_{i + 1}}(x)$ on $\text{Isom}(\mathbb{R}^k)$ is crystallographic too. Since the maximal lattice $\mathcal{L}_{i+1}$ of the projection of $\overline{G}_{\varepsilon_{i+1}}(x)$ on $\text{Isom}(\mathbb{R}^k)$ contains $\mathcal{L}_{i+2}$, then $\lambda(\mathcal{L}_{i+1}) \leq \lambda(\mathcal{L}_{i+2}) \leq \varepsilon_{i+1}/2\sqrt{n_0}$. Therefore $\mathcal{L}_{i+1}$ is $\varepsilon_{i+1}/2$-cocompact. This concludes (v).\\
Finally, assume that $x'$ is any point of $X$, say  $x'=gx$ with $d(x,x_0) \leq D_0$.
Observe that 
$\overline{G}_{\varepsilon_{i+1}}(x') = g \overline{G}_{\varepsilon_{i+1}}(x)g^{-1}$, so the rank does not change and 
the conditions (ii) and (vi) continue to hold  since 
the splitting is $G$-invariant.
Moreover
$\partial \overline{G}_{\varepsilon_{i+1}}(x')
=g \cdot \partial \overline{G}_{\varepsilon_{i+1}}(x)
=g \cdot Z = Z$, because $Z$ is $G$-invariant, so (iii) still holds. 
Finally, if the projection of $\overline{G}_{\varepsilon_{i+1}}(x)$ on $\text{Isom}(Y)$ preserves  $\lbrace y \rbrace \times \mathbb{R}^k$, then the projection of $\overline{G}_{\varepsilon_{i+1}}(x')$ on $\text{Isom}(Y)$ preserves $\lbrace gy \rbrace \times \mathbb{R}^k$, which proves (iv).
\end{proof}

\vspace{2mm}
\section{The Renormalization Theorem}
\label{sec-renormalization}
In this section we prove the analogue of the Renormalization Theorem \cite[Theorem E]{CS23} for totally disconnected group actions. This is a useful tool that that can be used to rescale a collapsing sequence in order to obtain a non-collapsing one. We will not use it here but it is essential for other applications in future works. The main point is to use the Splitting Theorem \ref{theo-splitting-weak} to show that, up to increasing in a controlled way the codiameter, we can always suppose that the {\em free-systole} is bounded away from zero by a universal positive constant. 
\begin{theo}
	\label{theo-bound-systole}
	Given $P_0,r_0,D_0>0$, there exist $s_0 = s_0(P_0,r_0,D_0)>0$ and $\Delta_0 = \Delta_0(P_0,D_0)$ such that the following holds.
	Let $X$ be a proper, geodesically complete, $(P_0,r_0)$-packed, $\textup{CAT}(0)$-space and let $G < \textup{Isom}(X)$ be closed, totally disconnected and $D_0$-cocompact.
	Then $G$ admits another faithful, continuous, $\Delta_0$-cocompact action by isometries on a $\textup{CAT}(0)$-space $X'$ isometric to $X$, i.e. $G$ can be seen as a closed, totally disconnected, $\Delta_0$-cocompact subgroup of $\textup{Isom}(X')$, with the additional property that 
	$\textup{sys}^\diamond(G, X') \geq s_0$.
\end{theo}
With the tools we have developed so far the proof is exactly the same of \cite[Theorem E]{CS23}, we will report it for completeness. \\
Recall the constants $n_0=P_0/2$, which bounds the dimension of any proper, geodesically complete,  CAT$(0)$-space  $X$ which is $(P_0, r_0)$-packed, and  $J_0$, introduced at the beginning of Section \ref{sec-splitting}. Also recall  the Margulis constant $\varepsilon_0=\varepsilon_0(P_0, r_0)$ given by Proposition \ref{prop-Margulis-CAT}, which we will always assume smaller than $1$ in the sequel. 

\begin{proof}[Proof of Theorem \ref{theo-bound-systole}] 
Recall the function $ \sigma_{P_0,r_0,D_0}(\varepsilon)$ of  Theorem \ref{theo-splitting-weak}.
Then we define inductively $D_1 = 2D_0 + \sqrt{n_0}$, $\sigma_1 = \sigma_{P_0,r_0,D_1}( \varepsilon_0)$ and
$$D_{j} = 2D_{j-1} + \sqrt{n_0} \;\,,\hspace{1cm}
\sigma_{j} = \sigma_{P_0,r_0,D_j} (\sigma_{j-1}) > 0.$$
\noindent We claim that $\Delta_0 := D_{n_0 -1}$ and $s_0 := \sigma_{n_0}$ satisfy the thesis; notice that both depend only on $P_0,r_0$ and $D_0$.\\
We define an algorithm which takes the CAT$(0)$-space $X_0:=X$ and produces a new proper, geodesically complete, $(P_0,r_0)$-packed CAT$(0)$-space $X_1$ on which $G$ still acts  faithfully by isometries, still satisfying all the assumptions of the theorem, except that  $\text{diam}(G \backslash X_1) \leq D_1$; and we will show that,  repeating again and again this process, we end up  with a CAT$(0)$-space $X_j$ with $\textup{sys}^\diamond(G,X_j) > \sigma_{n_0}$, for some $j\leq n_0$.\\
If $\textup{sys}^\diamond(G,X_0) > \sigma_{n_0}$, there is nothing to do, and we just set $X'=X_0$.\\
Otherwise,	 $\textup{sys}^\diamond(G,X) \leq \sigma_{n_0} < \sigma_1 = \sigma_{P_0,r_0,D_0} (\varepsilon_0) $, and we apply 
Theorem \ref{theo-splitting-weak} 
with $\varepsilon=\varepsilon_0$.
Then, there exists $\varepsilon^\ast_0:=\varepsilon^\ast \in (\sigma_1 ,\varepsilon_0)$ such that the groups 
$\overline{G}_{\varepsilon^\ast_0}(x,X_0)$ have all rank $k_0 \geq 1$ for every  $x\in X_0$. We then fix   $x_0 \in X_0$.   
By Theorem \ref{theo-splitting-weak} the group $\overline{G}_{\varepsilon^\ast_0}(x_0,X)$ of rank $k_0 \geq 1$ is almost commensurated in $G$ and $X_0$ splits isometrically and $G$-invariantly as $Y_0 \times \mathbb{R}^{k_0}$.
Moreover,  always by Theorem \ref{theo-splitting-weak}, there exists $y_0\in Y_0$ such that $\overline{G}_{\varepsilon^\ast_0}(x_0, X_0)$ preserves $\lbrace y_0 \rbrace \times \mathbb{R}^{k_0}$,  and  the maximal lattice $\mathcal{L}_{\varepsilon^\ast_0}(x_0,X_0)$ of the crystallographic group $\pi_{\mathbb{R}^{k_0}}(\overline{G}_{\varepsilon^\ast_0}(x_0, X_0))$ satisfies $\lambda(\mathcal{L}_{\varepsilon^*_0}(x)) \leq \varepsilon^*_0/2\sqrt{n_0}$. Here $\pi_{\mathbb{R}^{k_0}} \colon G \to \text{Isom}(\mathbb{R}^{k_0})$ is the projection map.
So, we can find a shortest basis $\mathcal{B}_0 = \lbrace \bf b^0_1, \ldots, \bf b^0_{k_0} \rbrace$  of the lattice $\mathcal{L}_{\varepsilon^\ast_0}(x_0,X_0)$ whose vectors have all length at most $\varepsilon^\ast_0/2\sqrt{n_0} < 1$; without loss of generality, we may suppose that $\Vert {\bf b^0_1} \Vert_0 \leq \ldots \leq  \Vert {\bf b^0_{k_0}} \Vert_0 = \lambda \left(\mathcal{L}_{\varepsilon^\ast_0}(x_0,X_0)\right) =:\ell_0 < 1$, where $\Vert \hspace{2mm}\Vert_0$ denotes the Euclidean norm of $\mathbb{R}^{k_0}$.
By \eqref{eq-lattice-relation}, we also know that $\rho(\mathcal{L}_{\varepsilon^\ast_0}(x_0,X_0)) \leq \frac{\sqrt{k_0}}{2}\cdot \ell_0 \leq \frac{\sqrt{n_0}}{2}\cdot\ell_0 $.\\	
Now, we define the metric space 
$X_1 := Y_0 \times\left( \frac{1}{\ell_0}\cdot \mathbb{R}^{k_0} \right)$. This is again a proper, geodesically complete, CAT$(0)$-space, still $(P_0,r_0)$-packed,   on which $G$ still acts faithfully by isometries (because the splitting of $X_0$ is $G$-invariant). 
We claim that the action of $G$ on $X_1$ is $D_1$-cocompact. In fact, let $x = (y,{\bf v})$ be any point of $X_1$. Since the action of $G$ on $X_0$ is $D_0$-cocompact, we know that there exists $g\in G$ such that $d_{X_0}\left(x , g\cdot(y_0,{\bf O}) \right) \leq D_0$; moreover, 
as $\overline{G}_{\varepsilon^\ast_0}(x_0, X)$ preserves $\lbrace  y_0 \rbrace \times \mathbb{R}^{k_0}$, up to compose $g$ with elements of this group we can suppose that $g = (g',g'') \in G$ satisfies $d_{X_0}(x, g \cdot  (y_0,{\bf O})) \leq D_0$ and 
$\Vert {\bf v} -  g'' \cdot {\bf O} \Vert_0 \leq  \frac{\sqrt{n_0}}{2}\cdot\ell_0 $. Therefore
\begin{equation*}
\begin{aligned}
d_{X_1}((x, g \cdot (y_0,{\bf O})) &
\leq \sqrt{d_{Y_0}(y, g' \cdot y_0)^2 + \left(\frac{1}{\ell_{0}}\right)^2\cdot 
\Vert {\bf v} -  g'' \cdot {\bf O} \Vert_0^2} \\ 
&\leq \sqrt{D_0^2 + \frac{n_0}{4}} \leq D_0 + \frac{\sqrt{n_0}}{2} = \frac{D_1}{2}.
\end{aligned}
\end{equation*}
As   $(y,{\bf v}) \in X_1$ was arbitrary, we then deduce that 
$ X_1 = G\cdot \overline{B}_{X_1}\left(( y_0,{\bf O}), \frac{D_1}{2}\right),$ 
so $G < \text{Isom}(X_1)$ is $D_1$-cocompact. 
If now $\textup{sys}^\diamond(G,X_1) > \sigma_{n_0}$, we stop the process and  set $X' = X_1$: this space has all the desired properties. \\
Otherwise, we have
$\textup{sys}^\diamond(G,X_1) \leq   \sigma_{n_0} 
< \sigma_2 = \sigma_{P_0,r_0,D_1}  (\sigma_1)$
and we can  apply again Theorem \ref{theo-splitting-weak} to $X_1$, with  $\varepsilon=\sigma_1$. 
Then, there exists  $\varepsilon^\ast_1 \in (\sigma_2 ,\sigma_1)$ 
such that the groups 
$\overline{G}_{\varepsilon^\ast_1}(x,X_1)$ have rank $k_1 \geq 1$ for every  $x\in X_1$, in particular  $\text{rk}(\overline{G}_{\varepsilon^\ast_1}(x_0,X_1))=k_1$.  Moreover $\overline{G}_{\varepsilon^\ast_1}(x_0,X_1)$ is almost commensurated in $G$, the space $X_1$ splits isometrically and $G$-invariantly as $Y_1 \times \mathbb{R}^{k_1}$, and the trace at infinity $\partial \overline{G}_{\varepsilon^\ast_1}(x_0, X_1)$ coincides with the boundary at infinity of all the sets $\lbrace y \rbrace \times \mathbb{R}^{k_1}$. Moreover the maximal lattice $\mathcal{L}_{\varepsilon^\ast_1}(x_0,X_1)$ of   $\pi_{\mathbb{R}^{k_1}}(\overline{G}_{\varepsilon^\ast_1}(x_0, X_1))$ satisfies $\lambda(\mathcal{L}_{\varepsilon^\ast_1}(x_0,X_1)) \leq \varepsilon_1^*/2\sqrt{n_0}$, where the notation is as before. Therefore, we can find
a shortest basis $\mathcal{B}^1 = \lbrace \bf b_1^1, \ldots, \bf b_{k_1}^1 \rbrace$ of $\mathcal{L}_{\varepsilon^\ast_1}(x_0,X_1)$ with lengths (with respect to the Euclidean norm  $\Vert \hspace{2mm} \Vert_1$ of $\mathbb{R}^{k_1}$)
$$\Vert {\bf b_{1}^1}\Vert_1 \leq \ldots \leq  \Vert {\bf b_{k_1}^1} \Vert_1 = \lambda \left(\mathcal{L}_{\varepsilon^\ast_1}(x_0,X_1)\right) =:\ell_1 \leq \varepsilon^\ast_1/2\sqrt{n_0} < 1.$$
Observe that, as
by construction the factor $\left(\frac{1}{\ell_0 }\right)^2$ is bigger than $1$, we have	
\begin{equation}
\label{eq-group-inclusion}
\overline{G}_{\varepsilon^\ast_1}( x_0, X_1) < \overline{G}_{\sigma_1}( x_0, X_1) < \overline{G}_{\sigma_1}( x_0, X_0) < \overline{G}_{\varepsilon^\ast_0}( x_0, X_0),
\end{equation}
so
$k_1= \text{rk}(\overline{G}_{\varepsilon^\ast_1}( x_0, X_1)) \leq  \text{rk}(\overline{G}_{\varepsilon^\ast_0}( x_0, X_0)) = k_0$. Moreover the metric splittings of $X_1$ as $Y_0 \times \left( \frac{1}{\ell_0}\cdot{\mathbb R}^{k_0}\right)$  and $Y_1 \times {\mathbb R}^{k_1}$  determined, respectively, by the groups $\overline{G}_{\varepsilon^\ast_0}( x_0, X_0)$ and $\overline{G}_{\varepsilon^\ast_1}( x_0, X_1)$ satisfy ${\mathbb R}^{k_1} \subset \frac{1}{\ell_0} \cdot {\mathbb R}^{k_0}$ and $Y_0 \subset Y_1$, because of  the second part of Proposition \ref{prop-commensurated-splitting}. \\
We will now show that $k_1 \lneq k_0$.
Actually, suppose that  $k_1 = k_0=:k$. 
Then $Y_0= Y_1$ and the two splittings coincide.
The lengths of the basis $\mathcal{B}^1 $ with respect to the Euclidean norm $\Vert \hspace{2mm} \Vert_0$ of  the Euclidean factor ${\mathbb R}^{k_0}$ of $X_0$ are $\Vert {\bf b_{i}^1} \Vert_0 = \ell_0  \cdot \Vert {\bf b_{i}^1} \Vert_1 < \ell_0$ 
for every $i=1,\ldots,k$.
But then, since $\mathcal{L}_{\varepsilon^\ast_1}(x_0,X_1) < \mathcal{L}_{\varepsilon^\ast_0}(x_0,X_0)$ by \eqref{eq-group-inclusion}, we would be able to find $k$ independent vectors of $\mathcal{L}_{\varepsilon^\ast_0}(x_0,X_0)$ of length less than its shortest generating radius, which is impossible.
This shows that $k_1 < k_0$.\\
We can now define   $X_2:= Y_1 \! \times \left(\frac{1}{\ell_1} \!\cdot \mathbb{R}^{k_1} \right)$, on which $G$ acts faithfully by isometries and  $D_2$-cocompactly, for $D_2 = 2D_1 + \sqrt{n_0}$ computed as before. 
We can repeat this process to get a sequence of proper, geodesically complete, $(P_0,r_0)$-packed, CAT$(0)$-spaces $X_j$ on which $G$ always acts faithfully and $D_j$-cocompactly by isometries. Moreover at each step 
either $\textup{sys}^\diamond(G,X_j) \geq \sigma_{n_0}$ or the $\varepsilon^\ast_j$-splitting rank $k_j$ of $X_j$ provided by Theorem \ref{theo-splitting-weak} is strictly smaller than the $\varepsilon^\ast_{j-1}$-splitting rank $k_{j-1}$ of $X_{j-1}$. Since the splitting rank of $X_0$ is at most $n_0$ there must exist $j \in \lbrace 1,\ldots, n_0 \rbrace$ such that $\textup{sys}^\diamond(G,X_j) \geq \sigma_{n_0}$. The proof then ends by setting $X' = X_j$. It is clear from the construction of $X'$ that it is isometric to $X$.
\end{proof}

\begin{obs} 
\label{rmk-splitting}
The proof actually produces a sequence of almost abelian commensurated subgroups of $G$
$$\lbrace \textup{id} \rbrace \lneqq A_0 \lneqq A_1\ldots \lneqq A_m$$
for some $m \leq n_0-1$, such that:
\begin{itemize}
\item[(a)] denoting by $k_j$ the rank of $A_j$, then $1 \leq  k_0 <k_1 \ldots < k_m$;
\item[(b)] setting $h_j = k_j - k_{j-1}$ then there is a corresponding  isometric and $G$-invariant splitting of $X$ as $Y\times \mathbb{R}^{h_0} \times \mathbb{R}^{h_1} \times \cdots \times \mathbb{R}^{h_m}$. Moreover, there exist $0<L_j < 1$ such that the natural action of $G$ on the space
\begin{equation}
\label{eq-finalsplitting}
X' := Y\times \left(\frac{1}{L_0}\cdot\mathbb{R}^{h_0} \right) \times \left(\frac{1}{L_1}\cdot\mathbb{R}^{h_1} \right)\times \cdots \times \left( \frac{1}{L_m}\cdot\mathbb{R}^{h_m} \right)
\end{equation}
is $\Delta_0$-cocompact with free-systole at least $s_0$. 
\end{itemize}
\end{obs}

\section{Convergence and collapsing}
\label{sec-convergence}
In this last section we will finally attack Theorem \ref{theo-intro-closure} using the tools developed until now.
First of all, for the reader's convenience, we briefly recall the notion of equivariant Gromov-Hausdorff convergence and ultralimits.

\subsection{Equivariant Gromov-Hausdorff convergence and ultralimits}
\label{sub-GH} ${}$\\  
An {\em isometric action} on a pointed space is a triple  $(X,x,G)$ where $X$ is a proper metric space, $x \in X$ is a basepoint and $G < \text{Isom}(X)$ is a closed subgroup. An {\em equivariant isometry} between isometric  actions of pointed spaces  $(X,x,G)$ and $(Y,y,H)$ is an isometry $F\colon X \to Y$ such that
\begin{itemize}
\item[--] $F(x)=y$;
\item[--] $F_*\colon \textup{Isom}(X) \to \textup{Isom}(Y)$ defined by $F_*(g) = F\circ g \circ F^{-1}$ is an isomorphism between $G$ and $H$.
\end{itemize}
\noindent The best known notion of convergence for isometric actions of pointed spaces is the  \emph{equivariant pointed Gromov-Hausdorff} convergence,  as defined by K. Fukaya in \cite{Fuk86}:
we will write $(X_j,x_j,G_j) \underset{\textup{eq-pGH}}{\longrightarrow} (X_\infty,x_\infty,G_\infty)$  
for a sequence $(X_j,x_j,G_j)$  of isometric actions converging in the  equivariant pointed Gromov-Hausdorff sense to an isometric action $(X_\infty,x_\infty,G_\infty)$. 

\noindent Forgetting about the group actions and  considering just  pointed metric spaces $(X_j,x_j)$,  this convergence reduces to the \emph{pointed Gromov-Hausdorff} convergence: we will write $(X_j,x_j) \underset{\textup{pGH}}{\longrightarrow} (X_\infty,x_\infty)$ for a sequence of pointed metric spaces $(X_j,x_j)$ converging to the pointed metric space $(X_\infty, x_\infty)$  in this sense. If moreover all the spaces  $X_j$ under consideration are compact we can even drop  the basepoints,  and we will simply write $X_j \underset{\textup{GH}}{\longrightarrow} X_\infty$ when  $X_j$ converges  to the compact metric space $X_\infty$. \\
We do not recall here the definition  of equivariant, pointed Gromov-Hausdorff convergence.
An equivalent approach uses ultralimits. We present it in this section, and then we recall the relations with  the usual Gromov-Hausdorff convergence, referring to \cite{Jan17} and \cite{Cav21ter}.\\
A {\em non-principal ultrafilter} $\omega$ is a finitely additive measure on $\mathbb{N}$ such that $\omega(A) \in \lbrace 0,1 \rbrace$ for every $A\subseteq \mathbb{N}$ and $\omega(A)=0$ for every finite subset of $\mathbb{N}$. \linebreak Accordingly, we will write  {\em $\omega$-a.s.}  or {\em for $\omega$-a.e.$(j)$} in the usual measure theoretic sense.  
Given a bounded sequence $(a_j)$ of real numbers and a non-principal ultrafilter $\omega$ there exists a unique $a\in \mathbb{R}$ such that for every $\varepsilon > 0$ the set $\lbrace j \in \mathbb{N} \text{ s.t. } \vert a_j - a \vert < \varepsilon\rbrace$ has $\omega$-measure $1$ (cp. \cite[Lemma 10.25]{DK18}). \linebreak The real number $a$ is then called {\em the ultralimit of the sequence $a_j$} and it is denoted by $\omega$-$\lim a_j$.

\noindent A     non-principal ultrafilter  $\omega$ being given,   one can define the {\em ultralimit pointed metric space} 
$(X_\omega, x_\omega)= \omega$-$\lim (X_j, x_j)$ of any  sequence of pointed metric spaces $(X_j, x_j)$: \\
-- first, one  says that a sequence    $(y_j)$, where $y_j\in X_j$ for every $j$, is {\em admissible} if there exists $M$ such that $d(x_j,y_j)\leq M$ for $\omega$-a.e.$(j)$; \\
-- then, one 
defines $(X_\omega, x_\omega)$ as  set of admissible sequences $(y_j)$ modulo the relation $(y_j)\sim (y_j')$ if and only if $\omega$-$\lim d(y_j,y_j') = 0$. \\
The point of $X_\omega$ defined by the class of the sequence $(y_j)$ is denoted by  $y_\omega = \omega$-$\lim y_j$. 
Finally, the formula $d(\omega$-$\lim y_j, \omega$-$\lim y_j') = \omega$-$\lim d(y_j,y_j')$ defines a metric on $X_\omega$ which is called the ultralimit distance on $X_\omega$.

\noindent Using a non-principal ultrafilter  $\omega$, one can also talk of limits of  isometries  and  of  isometry groups of pointed metric spaces. A sequence of isometries  $g_j $ of pointed metric spaces $(X_j, x_j)$ is {\em admissible} if there exists $M\geq 0$ such that $d(x_j, g_jx_j) \leq M$ $\omega$-a.s. Any such sequence defines a limit isometry $g_\omega = \omega$-$\lim g_j$ of $X_\omega=\omega$-$\lim (X_j, x_j)$  by the formula: $g_\omega y_\omega = \omega$-$\lim g_jy_j$  (\cite[Lemma 10.48]{DK18}).
Given a sequence of groups $G_j < \text{Isom}(X_j)$ we set
$$G_\omega = \lbrace \omega\text{-}\lim g_j \text{ s.t. } g_j \in G_j \text{ for } \omega\text{-a.e.}(j)\rbrace.$$
In particular the elements of $G_\omega$ are ultralimits of admissible sequences. 
One has a well-defined composition law on $G_\omega$ (\cite[Lemma 3.7]{Cav21ter}): if $ g_\omega   = \omega$-$\lim g_j$ and  $ h_\omega = \omega$-$\lim h_j$ we set $g_\omega \, \circ\, h_\omega := \omega\text{-}\lim(g_j \circ h_j).$
With this operation $G_\omega$ becomes a group of isometries of $X_\omega$, which  we call  {\em the ultralimit group} of the sequence of groups $G_j$.  
Notice that if $X_\omega$ is proper then $G_\omega$ is always a closed subgroup of isometries of $X_\omega$ \cite[Proposition 3.8]{Cav21ter}.

\vspace{1mm}
\noindent In conclusion,  a non-principal ultrafilter  $\omega$ being given, for any sequence of isometric actions on pointed  spaces $(X_j,x_j, G_j)$ there exists an {\em ultralimit isometric action} on a pointed space
\vspace{-3mm}

$$(X_\omega, x_\omega, G_\omega) =  \omega \!-\! \lim (X_j,x_j, G_j).$$ 

\vspace{1mm}
\noindent The ultralimit approach and the Gromov-Hausdorff convergence are essentially equivalent.

\begin{prop}[\cite{Cav21ter}, Proposition 3.13 \& Corollary 3.14] 
\label{prop-GH-ultralimit} ${}$\\
Let   $(X_j, x_j, G_j)$ be a sequence of isometric actions of pointed spaces:
\begin{itemize} 
\item[(i)]   if $(X_j,x_j,G_j) \underset{\textup{eq-pGH}}{\longrightarrow} (X_\infty,x_\infty,G_\infty)$, then $(X_\omega, x_\omega, G_\omega) \cong (X_\infty,x_\infty,G_\infty)$  for every non-principal ultrafilter $\omega$;
\item[(ii)] reciprocally, if $\omega$  is a non-principal ultrafilter and   $(X_\omega, x_\omega)$ is proper, then    $(X_{j_k},x_{j_k},G_{j_k}) \underset{\textup{eq-pGH}}{\longrightarrow} (X_\omega,x_\omega,G_\omega)$ for some subsequence $\lbrace{j_k}\rbrace$.
\end{itemize}
Moreover,  if for every non-principal ultrafilter $\omega$ the ultralimit  $(X_\omega,x_\omega,G_\omega)$ is equivariantly isometric to the same isometric  action   $(X,x,G)$, with $X$ proper,  then $(X_j,x_j,G_j) \underset{\textup{eq-pGH}}{\longrightarrow} (X,x,G)$.
\end{prop}

The sequence $(X_j,x_j, G_j)$   is called {\em $D$-cocompact} if  each $G_j$ is $D$-cocompact. 
The ultralimit of a sequence of isometric  actions on pointed spaces does not depend on the choice of the basepoints, 
provided that the actions have uniformly bounded codiameter.
\begin{lemma}
\label{lemma-ultralimit-cocompact}
Let $(X_j,x_j,G_j)$ be a sequence of isometric actions on pointed spaces, which are all $D$-cocompact and let  $x_j'\in X_j$ be different basepoints.  Then, for any   non-principal ultrafilter $\omega$, the ultralimit of   $(X_j,x_j,G_j)$ is equivariantly isometric to the ultralimit of  $(X_j,x_j',G_j)$.
\end{lemma}

\noindent Therefore, when considering the convergence of uniformily cocompact isometric actions, 
we will often omit the basepoint  $x$, if unnecessary for our arguments.\\
Finally, we remark that the equivariant pointed Gromov-Hausdorff convergence of a sequence of  isometric  actions   with uniformly bounded codiameter implies the pointed Gromov-Hausdorff convergence of the quotients.
\begin{lemma}[\cite{Fuk86}, Theorem 2.1] 
\label{lemma-ultralimit-quotient}${}$\\
Let $(X_j,G_j)$ be a given sequence of $D$-cocompact isometric  actions. \\
If  $(X_j, G_j) \underset{\textup{eq-pGH}}{\longrightarrow} (X_\infty,G_\infty)$ then $G_j\backslash X_j =: M_j \underset{\textup{GH}}{\longrightarrow} M_\infty := G_\infty \backslash X_\infty$.

\end{lemma}
\noindent In general the group $G_\infty$ is unknown and the structure of the quotient $G_\infty \backslash X_\infty$ is not clear at all. We want to clarify what happens in our specific setting. With this purpose, we now precisely define the following classes of isometric  actions of CAT$(0)$-groups we are interested in:
\begin{itemize}
\item[-] $\textup{CAT}_0^\textup{td,u}(D_0)$: this is the set  
of isometric actions $(X,  G)$
where $X$ is a  proper, geodesically complete, CAT$(0)$-space,  and $G$ is a $D_0$-cocompact, {\em totally disconnected and unimodular} subgroup of  $\text{Isom}(X)$; 
\item[-] $\textup{CAT}_0^{\textup{td,u}}(P_0,r_0,D_0)$:  the subset  of   $\textup{CAT}_0^{\textup{td,u}}(D_0)$ made of the 
actions $(X,  G)$ 
such that  $X$ is, moreover,   $(P_0,r_0)$-packed.
\end{itemize}
\noindent We will denote by ${\mathcal{O}} \text{-CAT}_0^{\textup{td-u}} (D_0)$,  ${\mathcal O} \text{-CAT}_0^{\textup{td-u}} (P_0, r_0, D_0)$ the  respective classes of quotients $M=G \backslash X$: these are  compact {\em generalized} CAT{$(0)$\em-orbispaces}, with 
$(X, G)$ respectively in $\text{CAT}_0^{\textup{td-u}} (D_0)$ and  $\text{CAT}_0^{\textup{td-u}} (P_0, r_0, D_0)$. \\
%
%
\noindent We say that a sequence of spaces  $X_j$
(or a sequence of actions on $X_j$) is {\em uniformly packed} if there exists $(P_0, r_0)$ such that every  $X_j$ is  $(P_0, r_0)$-packed. \linebreak
By  \cite[Theorem 6.1]{CavS20}, the class of proper, geodesically complete, pointed  $\textup{CAT}(0)$-spaces $(X_j,x_j)$ which are  $(P_0, r_0)$-packed  is closed under  ultralimits and compact with respect to the pointed Gromov-Hausdorff convergence. Moreover,   the proof of \cite[Lemma 5.4]{Cav21ter} implies that
$$\textup{CAT}_0^{\textup{td-u}}(D_0) = \bigcup_{P_0, r_0} \textup{CAT}_0^{\textup{td-u}}(P_0,r_0,D_0).$$
More precisely, we have: 

\begin{prop}[\textup{\cite[Proposition 7.5]{CS23}}]
\label{lemma-GH-compactness-packing}
A subset $K \subseteq \textup{CAT}_0^{\textup{td-u}}(D_0)$ is precompact (with respect to the equivariant pointed Gromov-Hausdorff convergence) if and only if there exist $P_0,r_0 > 0$ such that $K\subseteq \textup{CAT}_0^{\textup{td-u}}(P_0,r_0,D_0)$.
\end{prop}

Let us consider a sequence of isometric actions 
$(X_j, G_j)$ in $\textup{CAT}_0^{\textup{td-u}}(D_0)$,
$(X_j, G_j) \underset{\textup{eq-pGH}}{\longrightarrow} (X_\infty, G_\infty)$.
Our goal is to describe the limit group $G_\infty$ acting on $X_\infty$  and the quotient space $G_\infty \backslash X_\infty$.  

\begin{defin}[Standard setting of convergence]\label{defsetting}${}$\\
We say that we are in the \emph{standard setting of convergence} when we have a sequence 
$(X_j,G_j)$
in $ \textup{CAT}_0^{\textup{td-u}}(P_0,r_0,D_0)$ such that 
$(X_j,G_j) \underset{\textup{eq-pGH}}{\longrightarrow} (X_\infty,G_\infty)$.\\
We will denote by $M_j = G_j\backslash X_j$ and $M_\infty = G_\infty \backslash X_\infty$ the quotient spaces.
The standard setting of convergence  (or, equivalently,  the convergence of the isometric actions 
$(X_j,G_j)$
and of  the generalized  orbispaces $M_j$) will be  called:
\begin{itemize}
\item[-]  \emph{without collapsing} if $\limsup_{j\to+\infty} \textup{sys}^\diamond(G_j,X_j) > 0$;  
\item[-] \emph{with collapsing} if $\liminf_{j\to+\infty} \textup{sys}^\diamond(G_j,X_j) = 0$.
\end{itemize}
\end{defin}
\noindent These cases are mutually exclusive.
Actually,  the two cases  are respectively equivalent to the the conditions that the dimension  of the limit  $M_\infty = G_\infty \backslash X_\infty$  equals  the dimension  of the quotients $M_j$ or decreases,  as we  will   prove   in Theorem \ref{theo-collapsing-characterization}.  

\subsection{Convergence of Haar measures}
One of the main issues in the proof of Theorem \ref{theo-intro-closure} is to find some kind of stability of unimodularity under equivariant pointed Gromov-Hausdorff convergence. In this section we study this problem finding a sufficient condition for such stability.
We start by improving \cite[Proposition 7.6]{Cav21ter} to general groups, using the distances defined in \eqref{eq-defin-metric-group}.
\begin{prop}
	\label{prop-convergence-group-metric}
	Suppose that $(X_j,x_j,G_j) \underset{\textup{eq-pGH}}{\longrightarrow} (X_\infty,x_\infty,G_\infty)$, with $X_j$ geodesic spaces. Endow $G_j$ with the metric $d_{x_j}^\ell$ (resp. $d_{x_j}^\textup{r}$) and $G_\infty$ with the metric $d_{x_\infty}^\ell$ (resp. $d_{x_\infty}^\textup{r}$). Then $(G_j,\textup{id}, d_{x_j}^\ell) \underset{\textup{pGH}}{\longrightarrow} (G_\infty,\textup{id}, d_{x_\infty}^\ell)$ and $(G_j,\textup{id}, d_{x_j}^\textup{r}) \underset{\textup{pGH}}{\longrightarrow} (G_\infty,\textup{id}, d_{x_\infty}^\textup{r})$.
\end{prop}
\begin{proof}
	We do the proof for the left invariant metric, the other case is similar.
	By definition $X_\infty$ is proper, so $\textup{Isom}(X_\infty)$ is locally compact. Moreover $G_\infty < \textup{Isom}(X_\infty)$ is closed, so locally compact as well. As recalled in Section \ref{subsection-isometries} the metric $d_{x_\infty}^\ell$ is therefore proper. This means that we are in position to apply Proposition \ref{prop-GH-ultralimit}. In particular the thesis is equivalent to prove that $(G_\omega, \text{id}, d_{x_\omega}^\ell)$ is isometric as pointed metric space to $\omega$-$\lim (G_j, \text{id}, d_{x_j}^\ell)$ for every non-principal ultrafilter $\omega$. Let us fix a non-principal ultrafilter $\omega$. Let $(g_j)$ be an admissible sequence defining a point of $\omega$-$\lim (G_j, \text{id}, d_{x_j}^\ell)$. This means that $\omega$-$\lim d_{x_j}^\ell(g_j,\id) < +\infty$. This implies $\omega$-$\lim d(g_jx_j,x_j) < +\infty$ by the very definition \eqref{eq-defin-metric-group}, i.e. the sequence $g_j\in G_j$ defines a limit isometry $g_\omega \in G_\omega$. We set $\Phi \colon \omega$-$\lim (G_j, \text{id}, d_{x_j}^\ell) \to G_\omega$ defined by $(g_j) \mapsto g_\omega$. The distance function on $\omega$-$\lim (G_j, \text{id}, d_{x_j}^\ell)$ will be denoted by $\omega$-$\lim d_{x_j}^\ell$.\\
	\textbf{Good definition.} Let us take two admissible sequences $(g_j)$, $(h_j)$ defining the same point of $\omega$-$\lim (G_j, \text{id}, d_{x_j}^\ell)$, i.e. $\omega$-$\lim d_{x_j}^\ell(g_j,h_j)=0$. By definition, for every $\varepsilon > 0$ we have $d_{x_j}^\ell(g_j,h_j) < \varepsilon$ for $\omega$-a.e.$(j)$. By \eqref{eq-defin-metric-group} there exists $R> \frac{1}{\varepsilon}$ such that $d_{x_j,R}^\ell(g_j,h_j) < \varepsilon$ for $\omega$-a.e.$(j)$. If we now fix some $R > 0$ and we let $\varepsilon$ go to zero we get that $g_\omega$ coincides with $h_\omega$ on $\overline{B}(x_\omega, R)$. Since this is true for every $R>0$ we conclude that $g_\omega = h_\omega$, so $\Phi$ is well defined.\\
	\textbf{Isometric embedding.} We now prove that $\Phi$ is an isometry. By definition
	\begin{equation*}
		d_{x_\omega}^\ell(\Phi((g_j)), \Phi((h_j))) = d_{x_\omega}^\ell(g_\omega, h_\omega) = \inf_{R > 0}\left\lbrace \frac{1}{R} + d_{x_\omega,R}^\ell(g_\omega, h_\omega) \right\rbrace.
	\end{equation*}
	Recalling \eqref{eq-defin-L-infty-pseuodidistance} we have
	\begin{equation*}
		d_{x_\omega,R}^\ell(g_\omega, h_\omega) = \max_{y_\omega\in \overline{B}(x_\omega,R)} d(g_\omega y_\omega, h_\omega y_\omega) = \max_{y_\omega\in \overline{B}(x_\omega,R)} \omega\text{-}\lim d(g_j y_j, h_j y_j).
	\end{equation*}
	Using the fact that $\overline{B}(x_\omega, R) = \omega$-$\lim \overline{B}(x_j,R)$ (see \cite[Lemma A.8]{CavS20}) it is straightforward to conclude that 
	$$d_{x_\omega,R}^\ell(g_\omega, h_\omega) = \omega\text{-}\lim \max_{y_j\in \overline{B}(x_j,R)} d(g_j y_j, h_j y_j) = \omega\text{-}\lim d_{x_j,R}^\ell(g_j,h_j).$$
	Therefore $d_{x_\omega}^\ell(g_\omega, h_\omega) = \inf_{R > 0}\left\lbrace \omega\text{-}\lim \left(\frac{1}{R} + d_{x_j,R}^\ell(g_j, h_j) \right) \right\rbrace$.
	On the other side we know that $d_{x_j}^\ell(g_j, h_j) = \inf_{R > 0}\left\lbrace \frac{1}{R} + d_{x_j,R}^\ell(g_j, h_j) \right\rbrace$. We obtain
	$$\omega\text{-}\lim d_{x_j}^\ell((g_j), (h_j)) := \omega\text{-}\lim d_{x_j}^\ell(g_j,h_j)\leq d_{x_\omega}^\ell(g_\omega, h_\omega).$$
	To prove the other inequality we fix $\varepsilon > 0$ and for every $j$ we take $R_j >0$ such that $\frac{1}{R_j} + d_{x_j,R_j}^\ell(g_j, h_j) \leq d_{x_j}^\ell(g_j,h_j) + \varepsilon$. Since we know that $\omega\text{-}\lim d_{x_j}^\ell(g_j,h_j)$ is finite we deduce that $\omega$-$\lim R_j > 0$. If $\omega$-$\lim R_j =: R_\omega$ is finite then we conclude, as before, that $\omega$-$\lim \frac{1}{R_j} + d_{x_j,R_j}^\ell(g_j, h_j) = \frac{1}{R_\omega} + d_{x_\omega,R_\omega}^\ell(g_\omega, h_\omega)$. So 
	$$d_{x_\omega}^\ell(g_\omega, h_\omega) \leq \frac{1}{R_\omega} + d_{x_\omega,R_\omega}^\ell(g_\omega, h_\omega) \leq \omega\text{-}\lim d_{x_j}^\ell(g_j,h_j) + \varepsilon.$$
	If $R_\omega = +\infty$ we proceed as follows. We fix $R > \frac{1}{\varepsilon}$. For $\omega$-a.e.$(j)$ we have $R_j\geq R$. Therefore
	\begin{equation*}
		\begin{aligned}
			\frac{1}{R} + d_{x_\omega, R}^\ell(g_\omega, h_\omega) \leq d_{x_\omega,R}^\ell(g_\omega, h_\omega) + \varepsilon &= \omega\text{-}\lim d_{x_j,R}^\ell(g_j,h_j) + \varepsilon 	\\		&\leq \omega\text{-}\lim d_{x_j,R_j}^\ell(g_j,h_j) + \varepsilon \\
			&\leq \omega\text{-}\lim d_{x_j}^\ell(g_j,h_j) + 2\varepsilon.
		\end{aligned}
	\end{equation*}
	This implies $d_{x_\omega}^\ell(g_\omega, h_\omega) \leq \omega\text{-}\lim d_{x_j}^\ell(g_j,h_j) + 2\varepsilon$.
	In any case, by arbitrariness of $\varepsilon$, we conclude that $d_{x_\omega}^\ell(g_\omega, h_\omega) \leq \omega\text{-}\lim d_{x_j}^\ell(g_j,h_j)$, i.e. $\Phi$ is an isometric embedding.\\
	\textbf{Surjectivity.} In the last step we just need to show that $\Phi$ is surjective. Indeed if it is the case then it is a surjective isometric embedding between $\omega$-$\lim (G_j, \text{id}, d_{x_j}^\ell)$ and $(G_\omega, \id, d_{x_\omega}^\ell)$ with $\Phi(\id) = \id$, i.e. an isometry of pointed metric spaces. Let $g_\omega \in G_\omega$. By definition it is the ultralimit of a sequence of isometries $g_j \in G_j$ with $\omega$-$\lim d(g_j x_j, x_j) =: M < \infty$. Fix $R = 1$ and $y_j \in \overline{B}(x_j,1)$. Then $d(g_jy_j, y_j) \leq 2 + 2M$ for $\omega$-a.e.$(j)$, by triangular inequality. Then $d_{x_j}^\ell(g_j, \id) \leq 3 + 2M$ for $\omega$-a.e.$(j)$. In other words the sequence $(g_j)$ defines a point in $\omega$-$\lim (G_j,\id,d_{x_j}^\ell)$. It is clear that the image through $\Phi$ of this point is $g_\omega$, so $\Phi$ is surjective.
\end{proof}

In the situation above the groups $G_j$ are also equipped with Haar measures $\mu_j$, so we can ask when the sequence of pointed metric measure spaces $(G_j,\id,d_{x_j}^\ell,\mu_j)$ converges to a pointed metric measure spaces $(G_\infty, \id,d_{x_j}^\ell, \mu_\infty)$, and if it is the case when $\mu_\infty$ is a Haar measure of $G_\infty$.
We recall that we denote by $\overline{B}_{x_j}^\ell(\id,r)$ (resp. $\overline{B}_{x_j}^\textup{r}(\id,r))$ the closed ball of center $\id$ and radius $r$ in $G_j$ with respect to the metric $d_{x_j}^\ell$ (resp. $d_{x_j}^\textup{r}$), for $j\in \mathbb{N} \cup \lbrace \infty \rbrace$.\\
For reader's convenience we state the definition of pointed measured Gromov-Hausdorff convergence in this special situation, only for the left-invariant metrics. For the right invariant ones the definition is similar.
We say that that the sequence $(G_j,\id,d_{x_j}^\ell,\mu_j)$ converges in the pointed measured Gromov-Hausdorff sense to $(G_\infty, \id, d_{x_\infty}^\ell, \mu_\infty)$, and if it  is the case we write $(G_j, \id, d_{x_j}^\ell, \mu_j) \underset{\textup{pmGH}}{\longrightarrow} (G_\infty, \id, d_{x_\infty}^\ell, \mu_\infty)$, if there exist sequences $R_j \to +\infty$, $\varepsilon_j \to 0$ and $\psi_j\colon G_j \to G_\infty$ satisfying:
\begin{itemize}
	\item[(a)] $\psi_j(\id)=\id$;
	\item[(b)] $\sup_{g_j,h_j \in \overline{B}_{x_j}^\ell(\id,R_j)} \vert d_{x_j}^\ell(g_j,h_j) - d_{x_\infty}^\ell(\psi_j(g_j), \psi_j(h_j)) \vert < \varepsilon_j$;
	\item[(c)] for every $g_\infty \in \overline{B}_{x_\infty}^\ell(\id,R_j - \varepsilon_j)$ there exists $g_j \in \overline{B}_{x_j}^\ell(\id,R_j)$ such that $d_{x_\infty}^\ell(g_\infty, \psi_j(g_j)) < \varepsilon_j$;
	\item[(d)] for all $f\colon G_\infty \to \mathbb{R}$ continuous with compact support it holds
	\begin{equation}
		\label{eq-defin-pmGH}
		\int_{G_\infty} f d(\psi_j)_* \mu_j = \int_{G_\infty} f \circ \psi_j d\mu_j \underset{j\to +\infty}{\longrightarrow} \int_{G_\infty} f d\mu_\infty.
	\end{equation} 
\end{itemize}
We recall that if we remove condition (d) we have exactly the definition of pointed Gromov-Hausdorff convergence. An equivalent notion of convergence can be stated with ultrafilters (cp. \cite{PS21}). 
Before stating the main result of the section we need to improve Proposition \ref{prop-convergence-group-metric}.
\begin{lemma}
	Same assumptions of Proposition \ref{prop-convergence-group-metric}. Let $\psi_j \colon (G_j, \id, d_{x_j}^\ell) \to (G_\infty, \id, d_{x_\infty}^\ell)$ be maps realizing the pointed Gromov-Hausdorff convergence as above.	
	Let $g_j \in G_j$ such that $\psi_j(g_j)$ converges to $g_\infty \in G_\infty$. Consider the multiplications on the left $L_{g_j} \colon G_j \to G_j$ for $j\in \mathbb{N}\cup\lbrace \infty \rbrace$. Then the sequence of maps $L_{g_j}$ converges to $L_{g_\infty}$ in the following sense. For all $R\geq 0$ it holds
	\begin{equation}
		\label{eq-limit-left}
		\lim_{j\to +\infty}\sup_{h_j \in \overline{B}_{x_j}^\ell(\id,R)} d_{x_\infty}^\ell(\psi_j (L_{g_j}(h_j)), L_{g_\infty}(\psi_j(h_j))) = 0.
	\end{equation}
	A similar statement holds for the right invariant metrics and multiplications on the right.
\end{lemma}
\begin{proof}
	We prove the statement for the left case, the right one is similar. Since the metric $d_{x_j}^\ell$ is left-invariant each $L_{g_j}$ is an isometry and $L_{g_j}(\id) = \id$. We fix a non-principal ultrafilter $\omega$ and we consider $\omega$-$\lim (G_j,\id,d_{x_j}^\ell)$ which is isometric to $(G_\omega, \id, d_{x_\omega}^\ell)$ by Proposition \ref{prop-convergence-group-metric} and so to $(G_\infty, \id, d_{x_\infty}^\ell)$ by Proposition \ref{prop-GH-ultralimit}. The sequence of isometries $L_{g_j}$ is admissible, so it defines a limit isometry $\omega$-$\lim L_{g_j}$ in the usual way (cp. \cite[Lemma 10.48]{DK18}). It is clear that this limit isometry is sent via the map $\Phi$ of the proof of Proposition \ref{prop-convergence-group-metric} to $L_{g_\omega}$, where $L_{g_\omega}$ is the multiplication on the left by $g_\omega$ of $G_\omega$. Using again Proposition \ref{prop-GH-ultralimit} we get that the sequence of isometries $L_{g_j}$ defines a limit isometry on $G_\infty$ and it coincides with $L_{g_\infty}$. Now \eqref{eq-limit-left} is exactly the definition of convergence of maps.
\end{proof}

\begin{prop}
	\label{prop-convergence-Haar}
	Suppose that $(X_j,x_j,G_j) \underset{\textup{eq-pGH}}{\longrightarrow} (X_\infty,x_\infty,G_\infty)$, $X_j$ geodesic spaces, and let $\mu_j$ be a left invariant Haar measure on $G_j$. Then:
	\begin{itemize}
		\item[(i)] $(G_j,\id, d_{x_j}^\ell, \mu_j) \underset{\textup{pmGH}}{\longrightarrow} (G_\infty,\id,d_{x_\infty}^\ell,\mu_\infty)$ up to a subsequence if and only if $\sup_j \mu_j(\overline{B}_{x_j}^\ell(\id,R)) < +\infty$ for every $R \geq 0$ . If it is the case then $\mu_\infty$ is left invariant.
		\item[(ii)] The measure $\mu_\infty$ is a left invariant Haar measure of $G_\infty$ if and only if $\inf_{j} \mu_j(\overline{B}_{x_j}^\ell(\id,r)) > 0$ for every $r>0$.
	\end{itemize}
A similar statement holds for right invariant measures and metrics.
\end{prop}

	\begin{proof}
		We provide the proof only for the left invariant case, the other one being similar.
		The first part of (i) follows by \cite[Theorem 3.28]{Khe20} together with \cite[Theorem 2.6]{ADH13} and Proposition \ref{prop-convergence-group-metric}. We need to check that $\mu_\infty$ is left invariant. Fix maps $\psi_j \colon G_j \to G_\infty$ realizing the pointed measured Gromov-Hausdorff convergence. Fix an element $g_\infty \in G_\infty$: it is limit of elements $g_j \in G_j$. Call $L_{g_j}$ and $L_{g_\infty}$ the corresponding left multiplication maps. Fix a continuous function $f\colon G_\infty \to \mathbb{R}$ with compact support, say contained in $\overline{B}_{x_\infty}^\ell(\id,R)$ for some $R> 0$. 
		The continuous function $f$ is uniformly continuous.
		Using the compact support of $f$, its uniform continuity and \eqref{eq-limit-left} we obtain the following  estimate. For every $\varepsilon > 0$ we have
		$$\left\vert\int_{G_j} f \circ (\psi_j \circ L_{g_j})d\mu_j - \int_{G_j} f\circ L_{g_{\infty}} \circ \psi_j d\mu_j \right\vert < \varepsilon \mu_j(\overline{B}_{x_j}^\ell(\id,2R))$$
		for  $j$ big enough. Since $\mu_j(\overline{B}_{x_j}^\ell(\id,2R))$ is uniformly bounded in $j$ and $\varepsilon$ is arbitrary, we get
		$$\lim_{j \to + \infty} \int_{G_j} f \circ (\psi_j \circ L_{g_j})d\mu_j = \lim_{j\to + \infty} \int_{G_j} f\circ L_{g_{\infty}} \circ \psi_j d\mu_j.$$
		The left hand side equals $\int_{G_\infty} f d\mu_\infty$ because $(L_{g_j})_*\mu_j = \mu_j$ and \eqref{eq-defin-pmGH}. The function $f\circ L_{g_\infty}$ is still continuous and with bounded support, so again by \eqref{eq-defin-pmGH} the right hand side coincides with $\int_{G_\infty} f\circ L_{g_{\infty}} d\mu_\infty$. By the arbitrariness of $f$ we get $(L_{g_\infty})_*\mu_\infty = \mu_\infty$, i.e. $\mu_{\infty}$ is $L_{g_\infty}$-invariant. This shows that $\mu_{\infty}$ is left invariant.\\
		Clearly $\mu_\infty$ is finite on compact sets. Moreover it is positive on open sets if and only $\liminf_{j\to +\infty} \mu_j(\overline{B}_{x_j}^\ell(\id,r)) > 0$ for every $r>0$. This follows from the left  invariance and from, for instance, \cite[Remark 11.5]{PS21}. In this case $\mu_\infty$ is a left invariant Haar measure of $G_\infty$ by Lemma \ref{lemma-Haar-measure}.
	\end{proof}

\begin{cor}
	\label{cor-unimodularity-limit}
	Suppose that $(X_j,x_j,G_j) \underset{\textup{eq-pGH}}{\longrightarrow} (X_\infty,x_\infty,G_\infty)$, $X_j$ geodesic spaces, and $G_j$ unimodular. If there exist Haar measures $\mu_j$ of $G_j$ such that:
	\begin{itemize}
		\item[(i)] $\sup_{j} \mu_j(\overline{S}_R(x_j)) < + \infty$ for every $R\geq 0$;
		\item[(ii)] $\inf_{j} \mu_j(\overline{S}_r(x_j,1/r)) < + \infty$ for every $r>0$;
	\end{itemize}
	then $G_\infty$ is unimodular.
\end{cor}
\begin{proof}
	The conditions (i) and (ii) imply, and are essentially equivalent to, conditions in (i) and (ii) of Proposition \ref{prop-convergence-Haar} by Lemma \ref{lemma-comparison-metric-group}.
\end{proof}

\subsection{Almost stabilizers} 
\label{sub-almost stabilizers} In this part we generalize \cite[§7.2]{CS23} for totally disconnected, \emph{unimodular} groups. This is the part where unimodularity plays a role. Recall the function $\sigma_{P_0,r_0,D_0}: (0,\varepsilon_0] \rightarrow  (0,\varepsilon_0]$ given by Theorem \ref{theo-splitting-weak}, where $\varepsilon_0$ is the Margulis constant.
Remark that Theorem \ref{theo-splitting-weak}.(iv) yields, for any $x\in X$, a slice $\lbrace y \rbrace \times \mathbb{R}^k$ which is preserved by $\overline{G}_{\varepsilon^\ast}(x)$ (with $\sigma_{P_0,r_0,D_0}(\varepsilon)<  \varepsilon^\ast < \varepsilon$),   provided that the free-systole of $G$ is smaller than $\sigma_{P_0,r_0,D_0}(\varepsilon)$, so that $X$ splits as $Y \times \mathbb{R}^k$. It is useful to know that, for at least a specific point $x$,  the slice   can be chosen to be the one passing through $x$. This is proved, in the unimodular case, by the following:

\begin{prop}
\label{prop-minimal-close}
Let $X$ be a  proper, geodesically complete, $(P_0,r_0)$-packed   $\textup{CAT}(0)$-space,   and let  $G < \textup{Isom}(X)$  be closed, totally disconnected, unimodular and $D_0$-cocompact.
Assume that $\textup{sys}^\diamond(G,X)< \sigma_{P_0,r_0,D_0}(\varepsilon)$, so that $X$ splits as $Y\times \mathbb{R}^k$ by Theorem \ref{theo-splitting-weak}. 
Then there exists $x_0 = (y_0, {\bf v})\in X$ such that $\overline{G}_{\varepsilon^\ast}(x_0)$ preserves $\lbrace y_0 \rbrace \times \mathbb{R}^k$.
\end{prop}

\noindent The proof is based on a maximality argument, similar to the proof of \cite[Theorem 3.1]{CS22}.
We just need the following additional fact.
\begin{lemma}[compare with \textup{\cite[Lemma 3.3]{CS22}}]
\label{lemma-basic-containment}
Let $X$ be a proper, geodesically complete, \textup{CAT}$(0)$-space and let $G < \textup{Isom}(X)$ be closed, totally disconnected and cocompact.
For every $x\in X$ and $r > 0$ there exists an open set $U \ni x$ such that $\overline{S}_r(y)\subseteq \overline{S}_r(x)$ for all $y\in U$.
\end{lemma}	
\begin{proof}
	Suppose it is not true. We can find $y_j$ converging to $x$ and $g_j \in G$ such that $d(y_j,g_jy_j) \leq r$ but $d(x,g_jx) > r$. We can suppose that $g_j$ converges to $g \in G$. Clearly $d(g^{-1}g_jx,x)$ tends to $0$, so by Theorem \ref{theo-characterization-td}.(a) we have $g^{-1}g_j x = x$ for $j$ big enough. Therefore $d(x, g x) = d(x,g_jx) > r$ but on the other hand it holds $d(x, g x) = \lim_{j\to + \infty}d(y_j,g_jy_j) \leq r$, a contradiction.
\end{proof}

\begin{proof}[Proof of Proposition \ref{prop-minimal-close}]
Let $\mu$ be a bi-invariant Haar measure of $G$. Observe that $0<\mu(\overline{S}_{\varepsilon^\ast}(x)) <+\infty$ for every $x\in X$, because the set $\overline{S}_{\varepsilon^\ast}(x)$ is open and compact by Theorem \ref{theo-characterization-td}.(b). The map $x\mapsto \mu(\overline{S}_{\varepsilon^\ast}(x))$ is upper semicontinuous by Lemma \ref{lemma-basic-containment} and $G$-invariant by unimodularity. Therefore it admits a maximum by cocompactness of $G$. Let $x$ be a point where the maximum is realized. Apply Theorem \ref{theo-splitting-weak}.(iv) to the point $x$: there exists $y_0\in Y$ such that $\overline{G}_{\varepsilon^\ast}(x)$ preserves $\lbrace y_0 \rbrace \!\times\mathbb{R}^k$. If $x \in  \lbrace y_0 \rbrace \!\times \mathbb{R}^k$, then we set $x_0=x$ and there is nothing  more to prove. 
Otherwise call $x_0$ the projection of $x$ on the closed, convex, $\overline{G}_{\varepsilon^\ast}(x)$-invariant subset $\lbrace y_0 \rbrace \!\times \mathbb{R}^k$. 
Let $c\colon [0,d(x,x_0)] \to X$ be the geodesic $[x,x_0]$ and set 
$$T=\sup\lbrace t \in [0,d(x,x_0)] \text{ s.t. } \overline{S}_{\varepsilon^\ast}(c(t)) = \overline{S}_{\varepsilon^\ast}(x)\rbrace.$$
By definition $T\geq 0$ and we claim that $T=d(x,x_0)$ and it is  a maximum.\\
In order to do so we need the following claim.
\vspace{2mm}

\noindent\textbf{Claim:} if $\overline{S}_{\varepsilon^\ast}(y) \subseteq \overline{S}_{\varepsilon^\ast}(y')$ then the equality holds if and only if $\mu(\overline{S}_{\varepsilon^\ast}(y)) = \mu(\overline{S}_{\varepsilon^\ast}(y'))$. One direction is clear, so suppose $\mu(\overline{S}_{\varepsilon^\ast}(y)) = \mu(\overline{S}_{\varepsilon^\ast}(y'))$. The two sets are both open and closed by Theorem \ref{theo-characterization-td}.(b). Therefore if $\overline{S}_{\varepsilon^\ast}(y) \subsetneq \overline{S}_{\varepsilon^\ast}(y')$ then the set $\overline{S}_{\varepsilon^\ast}(y') \setminus \overline{S}_{\varepsilon^\ast}(y)$ would be open and non-empty, so of positive $\mu$-measure, contradicting the assumption.
\vspace{2mm}

\noindent We can now continue the proof. Assume $\overline{S}_{\varepsilon^\ast}(c(t)) = \overline{S}_{\varepsilon^\ast}(x)$. 
Then, $\overline{S}_{\varepsilon^\ast}(c(t + t')) \subseteq \overline{S}_{\varepsilon^\ast}(c(t)) = \overline{S}_{\varepsilon^\ast}(x)$ for all $t'>0$ small enough, by Lemma \ref{lemma-basic-containment}.
On the other hand, for every $g \in \overline{S}_{\varepsilon^\ast}(x)$ we have $d(x_0, g x_0) \leq d(x,gx)$ (since $g$ acts on $Y$ fixing $y_0$).
Therefore, by the convexity of the displacement functions we deduce  that $\overline{S}_{\varepsilon^\ast}(x) \subseteq \overline{S}_{\varepsilon^\ast}(c(t+t'))$ too. 	
This shows that the supremum defining $T$ is not realized, unless $T = d(x,x_0)$. Let now $0\leq t_j < T$ such that $t_j$ tends to $T$, so the points $c(t_j)$ converge to $c(T)$. By Lemma \ref{lemma-basic-containment} we get $\overline{S}_{\varepsilon^\ast}(c(t_j)) \subseteq \overline{S}_{\varepsilon^\ast}(c(T))$ for all $n$ big enough. But $\overline{S}_{\varepsilon^\ast}(c(t_j)) = \overline{S}_{\varepsilon^\ast}(x)$ has maximal measure among the sets of this form, therefore $\mu(\overline{S}_{\varepsilon^\ast}(x)) = \mu(\overline{S}_{\varepsilon^\ast}(c(t_j))) = \mu(\overline{S}_{\varepsilon^\ast}(c(T)))$ for all $j$. The claim implies that $\overline{S}_{\varepsilon^\ast}(c(T)) = \overline{S}_{\varepsilon^\ast}(x)$, i.e. $T$ is actually  a maximum and so $T=d(x,x_0)$.
Hence, $\overline{S}_{\varepsilon^\ast}(x_0) = \overline{S}_{\varepsilon^\ast}(x)$,
and $x_0$ is the point we are looking for.
\end{proof}

The following corollary is new even for discrete groups. Observe that in general the closure of the projection on the Euclidean factor is not totally disconnected (i.e. not discrete), also for discrete groups (see examples in \cite{LM21}, \cite{CM19} and \cite[Example 6.1]{CS23}). 
\begin{cor}
	\label{cor-projection-td}
	Let $X$ be a  proper, geodesically complete, $(P_0,r_0)$-packed   $\textup{CAT}(0)$-space,   and let  $G$  be a closed, totally disconnected, unimodular, $D_0$-cocompact subgroup of $\textup{Isom}(X)$.
	Assume that $\textup{sys}^\diamond(G,X)< \sigma_{P_0,r_0,D_0}(\varepsilon)$, so that $X$ splits as $Y\times \mathbb{R}^k$ by Theorem \ref{theo-splitting-weak}. 
	Let $x_0 = (y_0, {\bf v})\in X$ be such that $\overline{G}_{\varepsilon^\ast}(x_0)$ preserves $\lbrace y_0 \rbrace \times \mathbb{R}^k$, as provided by Proposition \ref{prop-minimal-close}. Then the closure $\overline{G_Y}$ of the projection $G_Y$ of $G$ on $\textup{Isom}(Y)$ is totally disconnected and the orbit $\overline{G_Y}y_0$ is made of $\frac{\varepsilon^\ast}{2}$-separated points.
\end{cor}
\begin{proof}
	We will prove that the orbit $G_Yy_0$ is made of $\frac{\varepsilon^\ast}{2}$-separated points. This implies that also $\overline{G_Y}y_0$ is made of $\frac{\varepsilon^\ast}{2}$-separated points and that $\overline{G_Y}$ is totally disconnected by Theorem \ref{theo-characterization-td}.(iv). Let us suppose there exists $g\in G$, $g=(g',g'') \in \text{Isom}(Y)\times \text{Isom}(\mathbb{R}^k)$, such that $0<d(g'y_0,y_0) \leq \frac{\varepsilon^\ast}{2}$. By Theorem \ref{theo-splitting-weak}.(v) we know that the projection of the group $\overline{G}_{\varepsilon^\ast}(x_0)$ on $\text{Isom}(\mathbb{R}^k)$ is a crystallographic group which is $\varepsilon^*/2$-cocompact.
	We can therefore compose $g$ with an element of $\overline{G}_{\varepsilon^\ast}(x_0)$ in order to obtain an isometry $h = (h',h'')$ with $d(h''{\bf v}, {\bf v}) \leq \frac{\varepsilon^\ast}{2}$. Each element of $\overline{G}_{\varepsilon^\ast}(x_0)$ acts on $Y$ by fixing $y_0$, hence $d(h'y_0, y_0) = d(gy_0, y_0)$. Hence $d(hx_0,x_0) \leq \varepsilon^\ast$, so $h \in \overline{G}_{\varepsilon^\ast}(x_0)$, implying $h'y_0 = y_0$, which is a contradiction.
\end{proof}

Another consequence of the work that we developed in Section  \ref{sec-splitting} is the following control of the almost stabilizers, which will be useful in   studying  converging sequences.

\begin{cor}
\label{cor-0-rank}
Let $X$ be a proper, geodesically complete, $(P_0,r_0)$-packed $\textup{CAT}(0)$-space,   and let  $G < \textup{Isom}(X)$  be closed, totally disconnected, unimodular and $D_0$-cocompact.
Let $\sigma (G, X) \! :\!= \! \sigma_{P_0,r_0,D_0} (\varepsilon^\diamond)$ 
be the constant obtained for $\varepsilon^\diamond \!= \!\min\left\{ \varepsilon_0, \textup{sys}^\diamond(G,X) \right\}$.
Then
\begin{itemize} 
\item[(i)]   the almost stabilizers   $\overline{G}_{\sigma (G, X)}(x)$ are compact, for all $x\in X$; 
\item[(ii)]   there exists $x_0\in X$ such that $\overline{G}_{\sigma (G, X)}(x_0)$ fixes $x_0$.
\end{itemize}
\end{cor}

\noindent Observe that, since by (i) the group $\overline{G}_{\sigma(G,X)}(x)$ is compact, then it has a fixed point (cp. \cite[Corollary II.2.8]{BH09}). Here we are saying  that for at least one specific point $x_0$ we have that $\overline{G}_{\sigma(G,X)}(x_0)$ fixes exactly $x_0$.

\begin{proof}[Proof of Corollary \ref{cor-0-rank}]
We first prove part (i), and set, for short, $\sigma= \sigma (G, X)$.
Suppose that there exists $x\in X$ such that $\overline{G}_{\sigma}(x)$ is not compact. If $\text{rk}(\overline{G}_{\sigma}(x)) = 0$ then $\overline{G}_{\sigma}(x)$ should fix all points of every set $C(A)$ as in Proposition \ref{prop-trace-infinity-almost}.(i), implying that $\overline{G}_{\sigma}(x)$ is compact. Therefore $\text{rk}(\overline{G}_{\sigma}(x)) \geq 1$.
Then, we can repeat all the arguments in the proof of  Theorem \ref{theo-splitting-weak} (namely, the construction of the sequence $\varepsilon_i$ and Proposition \ref{prop-commensurated}) for   $\varepsilon=\varepsilon^\diamond$ to show that $X$ splits isometrically and $G$-invariantly as $Y\times \mathbb{R}^k$, with $k\geq 1$, and that there exists  $\varepsilon^\ast \in \left(\sigma,  \varepsilon^\diamond \right)$ such  that $\overline{G}_{\varepsilon^\ast} (x)$ projects on $\text{Isom}(\mathbb{R}^k)$ as a crystallographic group. 
Moreover, we can find elements $g=(g',g'')$ of $\overline{G}_{\varepsilon^*}(x)$ such that $g''$ is a translation of length at most $\varepsilon^*/2\sqrt{n_0}$ and $g'$ fixes a point.
But these elements are hyperbolic isometries of $G$ with translation length at most $\varepsilon^\ast/2\sqrt{n_0} < \textup{sys}^\diamond(G,X)$, a contradiction.\\
Assertion (ii)  follows from an argument similar to that of Proposition \ref{prop-minimal-close}:  
we consider the same function $x\mapsto \overline{S}_\sigma(x)$, which admits a maximum by unimodularity and cocompactness of $G$. We fix a point $x$ where the maximum is realized. The group $\overline{G}_{\sigma}(x)$ is  compact, so  the closed, convex set $\text{Fix}(\overline{G}_{\sigma}(x))$ is not empty. If $x\in \text{Fix}(\overline{G}_{\sigma}(x))$, then $x_0=x$ and there is nothing more to prove. 
Otherwise we call $x_0$ the projection of $x$ on $ \text{Fix}(\overline{G}_{\sigma}(x))$, we consider again the geodesic $c\colon [0,d(x,x_0)] \to [x,x_0] \subset X$  and  we show as in Proposition \ref{prop-minimal-close} that 
$$ T :=\sup\lbrace t \in [0,d(x,x_0)] \text{ s.t. } \overline{S}_{\sigma}(c(t)) = \overline{S}_{\sigma}(x)\rbrace = d(x,x_0)$$
by Lemma \ref{lemma-basic-containment} and the convexity of the displacement function. Moreover, again by Lemma \ref{lemma-basic-containment} and by the maximality of $x$ it follows that  $T=d(x,x_0)$ is a maximum,
hence $\overline{G}_{\sigma}(x_0) = \overline{G}_{\sigma}(x)$, and   $x_0$ is the announced fixed point.
\end{proof}

\subsection{Controlled convergence, without unimodularity}
\label{subsec-controlled}
In this part we develop the main technical tools we need to deal with both the collapsed and the non-collapsed case. We introduce now the setup. Let $P_0,r_0,D_0$ be fixed constants. An isometric action $(X,x,G)$ is said to be $\sigma$-controlled if
\begin{itemize}
	\item[(a)] $X$ is a proper, geodesically complete, $(P_0,r_0)$-packed, \textup{CAT}$(0)$-space;
	\item[(b)] $G$ is a closed, totally disconnected and $D_0$-cocompact group of isometries, \emph{not necessarily unimodular};
	\item[(c)] there exists $\sigma > 0$ such that for every $g \in G$ either $d(x,gx)=0$ or $d(x,gx) \geq \sigma$.
\end{itemize}

We present the two main cases of $\sigma$-controlled isometric actions.  
\begin{itemize}
	\item[-] Let $(X,G) \in \textup{CAT}_0^{\textup{td-u}}(P_0,r_0,D_0)$. Corollary \ref{cor-0-rank}.(ii) shows that $(X,x,G)$ is $\sigma$-controlled with $\sigma = \sigma(G,X)$, for a suitable choice of the basepoint $x \in X$.
	\item[-] Let $(X,G) \in \textup{CAT}_0^{\textup{td-u}}(P_0,r_0,D_0)$ and suppose $\textup{sys}^\diamond(G,X) \leq  \varepsilon \leq \varepsilon_0$. Let $X = Y\times \mathbb{R}^{k}$ be the splitting and $x = (y,{\bf v})$ be the point provided by Proposition \ref{prop-minimal-close}. Then the isometric action $(Y, y, \overline{G_{Y}})$ satisfies the assumptions (a), (b) and (c) with $\sigma = \sigma_{P_0,r_0,D_0}(\varepsilon)$, as shown by Corollary \ref{cor-projection-td}, where $\overline{G_{Y}}$ is the closure of the projection of $G$ on $\textup{Isom}(Y)$. Indeed $\sigma_{P_0,r_0,D_0}(\varepsilon) \leq \varepsilon^*/2$.
\end{itemize}

\noindent We write $(X_j, x_j, G_j) \overset{\sigma}{\underset{\textup{eq-pGH}}{\longrightarrow}} (X_\infty, x_\infty, G_\infty)$ if each $(X_j,x_j,G_j)$ is $\sigma$-controlled and if $(X_j, x_j, G_j) {\underset{\textup{eq-pGH}}{\longrightarrow}} (X_\infty, x_\infty, G_\infty)$ .

\begin{prop}
	\label{prop-sigma-controlled-limit}
	If 
	$(X_j, x_j, G_j) \overset{\sigma}{\underset{\textup{eq-pGH}}{\longrightarrow}} (X_\infty, x_\infty, G_\infty)$ then $(X_\infty, x_\infty, G_\infty)$ is $\sigma$-controlled. So $G_\infty$ is closed, totally disconnected and $D_0$-cocompact.
\end{prop}
\begin{proof}
	We fix a non-principal ultrafilter $\omega$.
	By Proposition \ref{prop-GH-ultralimit}, it is enough to show the thesis for the ultralimit group $G_\omega$.
	By property (c) of the definition of $\sigma$-controlled isometric action the following holds: every isometry $g_j \in G_j$ either fixes $x_j$ or it moves it by at least $\sigma$. This implies that every isometry $g_\omega \in G_\omega$ either fixes $x_\omega$ or moves it by at least $\sigma$. In particular the orbit $G_\omega x_\omega$ is discrete. It is clear that $G_\omega$ is $D_0$-cocompact and closed (\cite[Proposition 3.8]{Cav21ter}), so $G_\omega$ is totally disconnected by Theorem \ref{theo-characterization-td}.(iv).
\end{proof}

This fact has many important consequences. We begin with a quantified version of \cite[Remark 6.2]{CM09b}.
\begin{prop}
	\label{prop-bounded-order-sigma-controlled}
	Let $P_0,r_0,D_0, \sigma,R >0$. Then there exists $N_0(\sigma, R) = N_0(P_0,r_0,D_0, \sigma, R) > 0$ such that the following holds true.
	Let $(X,x,G)$ be $\sigma$-controlled. Then for all $x\in X$ the cardinality of the image of the map $\textup{Stab}_G(x) \to \textup{Isom}(\overline{B}(x,R))$ is at most $N_0(\sigma, R)$.
\end{prop}

\begin{proof}
	Let us suppose the thesis is not true, so we can find $\sigma$-controlled isometric actions $(X_j, x_j, G_j)$ and points $y_j \in X_j$  such that the image of the map $\textup{Stab}_{G_j}(y_j) \to \textup{Isom}(\overline{B}(y_j,R))$ has at least $j$ elements. 
	Fix a non-principal ultrafilter $\omega$. By Proposition \ref{prop-sigma-controlled-limit} and Proposition \ref{prop-GH-ultralimit} we know that $G_\omega$ is closed, totally disconnected and $D_0$-cocompact. 
	By \cite[Remark 6.2]{CM09b} the image of the map $\Phi\colon \text{Stab}_{G_\omega}(y_\omega) \to \text{Isom}(\overline{B}(y_\omega, R + 1)$ has a finite number of elements, say $\Phi(g_{\omega,1}),\ldots,\Phi(g_{\omega, N})$, with $g_{\omega,1} = \text{id}$. So there exists $0<\varepsilon <1$ such that $d_{y_\omega,R+1}^\ell(g_{\omega, i}, g_{\omega,m}) > 4\varepsilon$ for all $1\leq i < m \leq N$. 
	By definition of ultralimit group we can write these isometries as $g_{\omega,i} = \omega$-$\lim g_{j,i}$, with $g_{j,1} = \text{id}$. We claim that the following statement holds for $\omega$-a.e.$(j)$:	
	\begin{itemize}
		\item[(i)] for all $h_j \in \text{Stab}_{G_j}(y_j)$ there exists $i\in \lbrace 1,\ldots, N\rbrace$ such that \linebreak $d_{y_j,R+1}^\ell(h_j, g_{j,i}) \leq \varepsilon$.
	\end{itemize}	
	Indeed if for $\omega$-a.e.$(j)$ we can find isometries $h_j$ contradicting this statement we get that the isometry $h_\omega$, which is well defined since each $h_j$ fixes $y_j$, satisfies $h_\omega \in \text{Stab}_{G_\omega}(y_\omega)$ and $d_{y_\omega,R+1}^\ell(h_\omega, g_{\omega,i}) \geq \varepsilon$ for all $i=1,\ldots,N$ which is absurd. 
	Moreover we have 
	\begin{itemize}
		\item[(ii)] $d_{y_j,R+1}^\ell(g_{j,i}, g_{j,m}) \geq 4\varepsilon$ for all $1\leq i < m \leq N$ and for $\omega$-a.e.$(j)$.
	\end{itemize}
	This implies that the ball $\overline{B}_{d_{y_j,R+1}^\ell}(\text{id}, \varepsilon)$ is a subgroup of $\text{Stab}_{G_j}(y_j)$ for $\omega$-a.e.$(j)$. Indeed if $h_j,h_j' \in \overline{B}_{d_{y_j,R+1}^\ell}(\text{id}, \varepsilon)$ then $h_jh_j' \in \overline{B}_{d_{y_j,R+1}^\ell}(\text{id}, 2\varepsilon)$. By (i) the element $h_jh_j'$ is $\varepsilon$-close to some $g_{j,i}$, so $d_{y_j,R+1}^\ell(\text{id}, g_{j,i}) = d_{y_j,R+1}^\ell(g_{j,1}, g_{j,i}) \leq 3\varepsilon$. Now (ii) implies that $g_{j,i} = g_{j,1} = \text{id}$.
	The last thing we need to recall is 
	\begin{itemize}
		\item[(iii)] the multiplication on the left by $g_{j,i}$ defines a $d_{y_j,R+1}^\ell$-isometry between $\overline{B}_{d_{y_j,R+1}^\ell}(\text{id}, \varepsilon)$ and $\overline{B}_{d_{y_j,R+1}^\ell}(g_{j,i}, \varepsilon)$.
	\end{itemize}
	Recall that we are supposing that the image of $\text{Stab}_{G_j}(y_j) \to \text{Isom}(\overline{B}(y_j, R))$ has cardinality bigger than $N+1$ for $\omega$-a.e.$(j)$. This means, together with (iii), that the following holds true $\omega$-a.s.: there exists an isometry $h_j \in \text{Stab}_{G_j}(y_j)$ such that $d_{y_j,R+1}^\ell(h_j, \text{id}) \leq \varepsilon$ and $d_{y_j,R}^\ell(h_j, \text{id}) > 0$. \\
	We claim that $d_{y_j,R+1}^\ell(h_j^k,\text{id}) > 1$ for some $k\in \mathbb{Z}$, which is a contradiction to the fact that $\overline{B}_{d_{y_j,R+1}^\ell}(\text{id}, \varepsilon)$ is a subgroup. We proceed as follows. Since $d_{y_j,R}^\ell(h_j, \text{id}) > 0$ we can find a point $p_j \in \overline{B}(y_j,R)$ which is not fixed by $h_j$. Call $z_j$ the projection of $p_j$ to $\text{Fix}(h_j)$ and extend the geodesic $[z_j,p_j]$ beyond $p_j$ up to find a point $w_j$ at distance $R+1$ from $x_j$. 
	Three things hold: $z_j$ is the projection of $w_j$ to $\text{Fix}(h_j)$, $d(w_j,z_j) > 1$ and $w_j\in \overline{B}(y_j,R+1)$. If the whole orbit $\langle h_j \rangle w_j$ would be contained in $\overline{B}(w_j,1)$ then the center of this orbit (cp. \cite[Proposition II.2.7]{BH09}) would be a fixed point of $h_j$ at distance at most $1$ from $w_j$, which is a contradiction since $d(w_j,\text{Fix}(h_j)) > 1$. 
	Therefore there must be some $k\in \mathbb{Z}$ such that $d(h_j^kw_j, w_j) > 1$, i.e. $d_{y_j,R+1}^\ell(h_j^k,\text{id}) > 1$ for some $k\in\mathbb{Z}$.
\end{proof}

This gives a similar uniform estimate for all unimodular groups.
\begin{cor}
	\label{cor-bounded-order-unimodular}
	Let $P_0,r_0,D_0,R >0$. There exists a constant $N_0^\textup{u}(R) = N_0^\textup{u}(P_0,r_0,D_0, R) > 0$ such that the following holds true.
	Let $X$ be a proper, geodesically complete, $(P_0,r_0)$-packed, $\textup{CAT}(0)$-space and let $G < \textup{Isom}(X)$ be closed, totally disconnected, unimodular and $D_0$-cocompact. Then for all $x\in X$ the cardinality of the image of the map $\textup{Stab}_G(x) \to \textup{Isom}(\overline{B}(x,R))$ is at most $N_0^\textup{u}(R)$.
\end{cor}
\begin{proof}
	If $\textup{sys}^\diamond(G,X) > \sigma_0:=\sigma_{P_0,r_0,D_0}(\varepsilon_0)$ then it is enough to take $N_0^\textup{u}(R) = N_0(\sigma_0, R)$. Otherwise we apply Theorem \ref{theo-splitting-weak} with $\varepsilon = \varepsilon_0$. We have a splitting $X= Y \times \mathbb{R}^k$, $k\geq 1$ and we write $x = (y,{\bf v})$. Every element of $\textup{Stab}_G(x)$ is of the form $(g',g'')$, where $g'$ fixes $y$ and $g''$ is a finite order isometry of a crystallographic group fixing ${\bf v}$. The order of the stabilizer of ${\bf v}$ inside this crystallographic group is at most $J_0$ by Proposition \ref{prop-Bieberbach}. The isometric action $(Y,y,\overline{G_Y})$ is $\sigma_0$-controlled by Corollary \ref{cor-projection-td}, see also the beginning of this section. Therefore the cardinality of the image of the map $\text{Stab}_{\overline{G_Y}}(y) \to \text{Isom}(\overline{B}(y,R))$ is at most $N_0(\sigma_0, R)$. The thesis follows taking $N_0^\textup{u}(R) = N_0(\sigma_0, R) \cdot J_0$.
\end{proof}

As a consequence of Corollary \ref{cor-bounded-order-unimodular} we have a quantified version of \cite[Proposition 6.8]{CM09b} for unimodular groups. A similar statement for $\sigma$-controlled groups with $N_0(\sigma,R)$ in place of $N_0^\text{u}(R)$ holds, but we do not write it. 
\begin{cor}
	\label{cor-bounded-angle-unimodular}
	Let $X$ be a proper, geodesically complete, $(P_0,r_0)$-packed, \textup{CAT}$(0)$-space and let $G < \textup{Isom}(X)$ be closed, totally disconnected, unimodular and $D_0$-cocompact.	Let $g$ be an elliptic isometry of $G$, $x\notin \textup{Fix}(g)$ and $y$ be the projection of $x$ on $\textup{Fix}(g)$. Then 
	\begin{itemize}
		\item[(i)] $\angle_y(x,gx) \geq 2\arcsin(\frac{1}{2N_0^{\textup{u}}(1)})$;
		\item[(ii)] $d(x,gx) \geq \frac{1}{N_0^\textup{u}(1)}\cdot d(x,\textup{Fix}(g))$.
	\end{itemize}
\end{cor}
\begin{proof}
	We need to recall the formula \cite[Proposition II.3.1]{BH09}:
	\begin{equation}
		\label{eq-angle}
		\alpha:=\angle_y(x,gx) = \lim_{t\to 0} 2\arcsin\left(\frac{d(c(t), gc(t))}{2t}\right),
	\end{equation}
	where $c(t)$ is the geodesic $[y,x]$, and where the right hand side is monotone decreasing as $t$ goes to $0$. By Corollary \ref{cor-bounded-order-unimodular} we know that the order of $g$, when seen as an isometry of $\overline{B}(x,1)$, is at most $N_0^\text{u}(1)$. Suppose $\alpha < 2\arcsin(\frac{1}{2N_0^{\textup{u}}(1)})$ and choose $\varepsilon > 0$ so that $\alpha+\varepsilon < 2\arcsin(\frac{1}{2N_0^\textup{u}(1)})$. We can find $t> 0$ such that
	\begin{equation}
		\label{eq-CM-error}
		\alpha +\varepsilon \geq 2\arcsin\left(\frac{d(c(t), gc(t))}{2t}\right),
	\end{equation}
	i.e. $2t\sin(\frac{\alpha+ \varepsilon}{2}) \geq d(c(t),gc(t))$. If $k\in \lbrace 1,\ldots,N_0^\textup{u}(1)\rbrace$ we get $$d(c(t),g^kc(t)) \leq 2kt\sin\left(\frac{\alpha+ \varepsilon}{2}\right) < \frac{k}{N_0^\textup{u}(1)}t \leq t.$$
	Therefore the orbit $\langle g \rangle c(t)$ is contained in a ball of center $c(t)$ and radius smaller than $t$. The center of this orbit (cp. \cite[Proposition II.2.7]{BH09}) is a point which is fixed by $g$ and which is at distance $<t$ from $c(t)$, contradicting the fact that $y$ is the projection of $c(t)$ on $\text{Fix}(g)$.
	
	\noindent A direct application of \eqref{eq-angle} gives (ii). Indeed, recalling that the quantity inside the limit in \eqref{eq-angle} is decreasing as $t$ goes to $0$, we get
	$$2\arcsin\left( \frac{d(x,gx)}{2d(x,y)}\right) \geq \alpha.$$
\end{proof}

\begin{obs}
	The proof of (i) is conceptually equivalent to the proof of \cite[Proposition 6.8]{CM09b}, just notice that they wrote the wrong inequality, essentially opposite to \eqref{eq-CM-error}. However their proof still works using the limit procedure as we did.
\end{obs}

We can also prove Theorem \ref{theo-intro-order-td}, actually a bit stronger result.

\begin{cor}
	Let $X$ be a proper, geodesically complete, $(P_0,r_0)$-packed, \textup{CAT}$(0)$-space. Let $G<\textup{Isom}(X)$ be closed, totally disconnected, unimodular and $D_0$-cocompact. Let $x\in X$, let $g \in \textup{Stab}_G(x)$ and let $R\geq 0$. Let $m = p_1^{\alpha_1} \cdots p_k^{\alpha_k}$ be the prime decomposition of the order of $g$ restricted to $\textup{Isom}(\overline{B}(x,R))$. Then $\max\lbrace p_1,\ldots, p_k\rbrace \leq N_0^\textup{u}(1)$.\\
	In particular the maximal prime number appearing in the prime decomposition of the order of any finite order isometry is at most $N_0^\textup{u}(1)$. 
\end{cor}
Clearly the $\alpha_k$'s are not bounded in general. A similar statement holds in the $\sigma$-controlled case with $N_0(\sigma,1)$ in place of $N_0^\text{u}(1)$, but we do not report it here.
\begin{proof}
	The thesis is equivalent to bound the order of $g\vert_{\overline{B}(x,R)}$ when it has prime order $p > 1$. For all $1\leq j < p$ we can find $k\in \mathbb{Z}$ such that $kj \equiv 1$ modulo $p$, so
	\begin{equation}
		\label{eq-fix-prime-order}
		\text{Fix}(g) \subseteq \text{Fix}(g^j) \subseteq \text{Fix}(g^{kj}) = \text{Fix}(g),
	\end{equation}
	forcing the containments to be equalities. Let $y$ be a point of $\overline{B}(x,R)$ which is not fixed by $g$ and let $x'$ be the closest point to $y$ along $[x,y]$ which is fixed by $g$. Finally let $y'$ be a point of $[x',y]$ at distance at most $1$ from $x'$ and different from $x'$. By \eqref{eq-fix-prime-order} we conclude that the points $g^jy'$ are all distinct for $1\leq j < p$ and contained in $\overline{B}(x',1)$. Therefore $p \leq N_0^\text{u}(1)$.
\end{proof}

The last important consequence is the following uniform bound, that we explicit in both the unimodular and the $\sigma$-controlled situation. 

\begin{cor}
	\label{cor-uniform-index}
	Let $X$ be a proper, geodesically complete, $(P_0,r_0)$-packed, \textup{CAT}$(0)$-space and $x\in X$. Let $G < \textup{Isom}(X)$ be closed, totally disconnected, unimodular (resp. $\sigma$-controlled) and $D_0$-cocompact. Then for every $R\geq 0$ the group $\textup{Stab}_G(x,R)$ has index at most $N_0^\textup{u}(R)$ (resp. $N_0(\sigma,R)$) in $\textup{Stab}_G(x)$.
\end{cor}
\begin{proof} 
	Corollary \ref{cor-bounded-order-unimodular} (resp. Proposition \ref{prop-bounded-order-sigma-controlled}) implies that the cardinality of the group $\text{Stab}_{G}(x)/\text{Stab}_G(x,R)$ is at most $N_0^\textup{u}(R)$ (resp. $N_0(\sigma,R)$), which is the thesis.
\end{proof}

\subsection{Convergence without collapsing} 
\label{sub-convergencenocollapsing}
We are ready to describe the limit group in the non-collapsed standard setting of convergence.
\begin{theo}
\label{theo-noncollapsed}
Assume that we are in the standard setting of convergence of 
$(X_j, G_j) \underset{\textup{eq-pGH}}{\longrightarrow} (X_\infty, G_\infty)$, 
without collapsing. Then the limit group $G_\infty$ is closed, totally disconnected, unimodular and $D_0$-cocompact.
\end{theo}

\begin{proof} 
	Up to pass to a subsequence, which converges to the same limit, we can suppose that there exists $0<\varepsilon \leq \varepsilon_0$ such that
	$\text{sys}^\diamond(G_{j},X_{j}) \geq \varepsilon > 0$ for  every $j$. Therefore each $(X_j,G_j)$ is $\sigma := \sigma_{P_0,r_0,D_0}(\varepsilon)$-controlled, as explained at the beginning of Section \ref{subsec-controlled}. Proposition \ref{prop-sigma-controlled-limit} implies that $G_\infty$ is closed, $D_0$-cocompact and totally disconnected. \\
	Let $x_j$ be points of $X_j$ such that $\overline{G}_\sigma(x_j) = \text{Stab}_{G_j}(x_j)$.
	Let $\mu_j$ be bi-invariant Haar measures on $G_j$ normalized in such a way that $\mu_j(\text{Stab}_{G_j}(x_j)) = 1$.
	We claim that conditions (i) and (ii) of Corollary \ref{cor-unimodularity-limit} are satisfied, and so that $G_\infty$ is unimodular. We first show that $\sup_j\mu_j(\overline{S}_R(x_j)) < + \infty$ for every $R\geq 0$. The condition $\overline{G}_\sigma(x_j) = \text{Stab}_{G_j}(x_j)$ implies that the set $G_jx_j$ is made of $\sigma/4$-separated points. In particular the distinct points of the set $\overline{S}_R(x_j) x_j$ are at most $\text{Pack}(R,\frac{\sigma}{8}) =: P$, a number that does not depend on $j$ by Proposition \ref{prop-packing}. Let us take elements $g_1,\ldots,g_m \in \overline{S}_R(x_j)$ such that $\lbrace g_i x_j \rbrace_{i=1}^m = \overline{S}_R(x_j) x_j$. By the discussion above we know that $m \leq P$. Moreover if $g \in \overline{S}_R(x_j)$ is arbitrary then there exists $i \in \lbrace 1,\ldots,m\rbrace$ such that $gx_j = g_ix_j$, so $g_i^{-1}g \in \text{Stab}_{G_j}(x_j)$. This implies 
	$$\overline{S}_R(x_j) = \bigcup_{i=1}^m g_i \text{Stab}_{G_j}(x_j).$$
	Therefore the claim is proved since
	$$\mu_j(\overline{S}_R(x_j)) \leq \sum_{i=1}^m \mu_j (g_i \text{Stab}_{G_j}(x_j)) = \sum_{i=1}^m \mu_j (\text{Stab}_{G_j}(x_j)) = m \leq P.$$
	On the other hand Corollary \ref{cor-uniform-index} implies that for every $r>0$ we have 
	$$\inf_{j} \mu_j\left(\overline{S}_r\left(x_j,\frac{1}{r}\right)\right) \geq \inf_{j} \mu_j\left(\text{Stab}_{G_j}\left(x_j, \frac{1}{r}\right)\right) \geq \frac{1}{N_0^\text{u}(\frac{1}{r})} > 0.$$
\end{proof}

\subsection{Convergence with  collapsing}
\label{sub-convergencecollapsing} 

We now deal with the collapsing case.
\begin{theo}
\label{theo-collapsed} 
Assume that we are in the  standard setting of convergence of 
$(X_j, G_j) \underset{\textup{eq-pGH}}{\longrightarrow} (X_\infty, G_\infty)$
with collapsing. Then: 
\begin{itemize} 
\item[(i)]  $X_\infty$ splits isometrically and $G_\infty$-invariantly as $ X_\infty' \times \mathbb{R}^\ell$, for some $ \ell \geq 1$. In particular $X_\infty'$ is a proper, geodesically complete, $(P_0,r_0)$-packed, $\textup{CAT}(0)$-space.
\item[(ii)]  $G_\infty^\circ  =  \lbrace \textup{id}\rbrace \times \textup{Transl}(\mathbb{R}^\ell)$, and 
$ X_\infty' = G_\infty^\circ \backslash X_\infty$.	
\item[(iii)] $\ell$ is characterized as follows: there exists $\rho^* > 0$ such that for every $0<\rho \leq \rho^*$ there exists $j_\rho$ such that $\textup{rk}(\overline{G}_{j,\rho}(z_j)) = \ell$ for all $j\geq j_\rho$ and all $z_j\in X_j$. Here $\overline{G}_{j,\rho}(z_j) := \overline{(G_j)}_\rho(z_j)$ as defined in \eqref{defsigma}.
\item[(iv)] the group $G_\infty$ is unimodular.
\item[(v)] the projection $G_\infty'$ of $G_\infty$ on $\textup{Isom}(X_\infty')$ is closed, totally disconnected, unimodular and $D_0$-cocompact.
\item[(vi)] the  quotient spaces $M_j=G_j \backslash X_j$  converge to $M_\infty = G_\infty \backslash X_\infty$, which  is isometric to the quotient of $X_\infty'$ by $G_\infty'$. In particular $M_\infty \in \mathcal{O}\textup{-CAT}_0^{\textup{td,u}}(P_0,r_0,D_0)$.
\end{itemize} 
\end{theo}

This theorem is stronger than and refines \cite[Theorem 7.11]{CS23}, which is stated for discrete groups. The main, and technically more important, improvement is the characterization of $\ell$ given in (iii).
We recall a couple of facts we need.

\begin{lemma}[\textup{\cite[Lemma 7.12]{CS23}}]
\label{lemma-ultralimit-product}
Let $(X_j, x_j, G_j), (X_j', x_j', G_j')$ be two sequences of isometric actions and $\omega$ be a non-principal ultrafilter. Then the ultralimit $\omega\text{-}\lim(X_j\times X_j', (x_j,x_j'), G_j \times G_j')$ is equivariantly isometric to $(X_\omega \times X_\omega', (x_\omega, x_\omega'), G_\omega \times G_\omega')$.
\end{lemma}

\begin{lemma}[\textup{\cite[Lemma 7.13]{CS23}}]
\label{lemma-translations}
Let $A$ be a group of translations of  $\mathbb{R}^k$.
Then $A \cong  \mathbb{R}^{\ell} \times \mathbb{Z}^{d}$ with $\ell + d \leq k$. Moreover, there is a corresponding $A$-invariant metric factorization of $\mathbb{R}^k$ as $\mathbb{R}^{\ell} \times \mathbb{R}^{k-\ell} $, such that the connected component $A^\circ \cong \mathbb{R}^{\ell}$ can be identified with $ \textup{Transl}(\mathbb{R}^\ell)$.
\end{lemma}

Recall again the function $\sigma_{P_0,r_0,D_0}(\cdot)$ provided by Theorem \ref{theo-splitting-weak}, and the value $\sigma_0 = \sigma_{P_0,r_0,D_0}(\varepsilon_0)$.

\begin{proof}[Proof of Theorem \ref{theo-collapsed}]
The sets defined in \eqref{defsigma} will be denoted by $S_{j,r}(z_j, R)$, $\overline{S}_{j,r}(z_j, R)$, $G_{j,r}(z_j, R)$, $\overline{G}_{j,r}(z_j, R)$ if they refer to the group $G_j$. 
\vspace{2mm}

\noindent\textbf{Proof of (i) and (ii).} We fix a non-principal ultrafilter $\omega$ such that $\omega$-$\lim\textup{sys}^\diamond(G_j,X_j) = 0$. By Proposition \ref{prop-GH-ultralimit} it is enough to show the thesis for the ultralimit   $( X_\omega, G_\omega)$. By our choice of the ultrafilter $\omega$, we have  $\textup{sys}^\diamond(G_j,X_j) \leq \sigma_0$  for $\omega$-a.e.$(j)$. Thus $X_j$ splits as $Y_j\times \mathbb{R}^{k_j}$ with $k_j \geq 1$ for $\omega$-a.e.$(j)$, by Theorem \ref{theo-splitting-weak}. If $k = \omega$-$\lim k_j$, we have  $k\leq n_0$ and $k_j = k$ for $\omega$-a.e.$(j)$.
By Lemma \ref{lemma-ultralimit-cocompact} the limit does not depend on the choice of the basepoints, so we can assume that the basepoints are the points $x_j=(y_j, {\bf v_j})$ provided by Proposition \ref{prop-minimal-close} applied to $\varepsilon = \varepsilon_0$. In particular  $\overline{G}_{j,\sigma_0}(x_j)$ preserves the slice $\lbrace y_j \rbrace \times \mathbb{R}^{k}$.\\
Choose a positive 
$\eta 
\leq \frac12 \min \left\{ \sigma_0 ,  
\sqrt{2 \sin \left( \frac{\pi}{J(k)}\right) }\right\}.$  By Lemma \ref{lemma-bieber}, every element of a crystallographic group of $\mathbb{R}^{k}$ moving every point of $B_{\mathbb{R}^{k}} ({\bf v_j}, \frac{1}{2\eta})$ less than $2\eta$ is a translation. The group $G_\eta(x_\omega, 1/\eta)$ is open, so $G_\eta(x_\omega, 1/\eta) \cap G_\omega^\circ = G_\omega^\circ$. 
Remark that if 	$g_\omega = \omega$-$\lim g_j $ belongs to $S_\eta(x_\omega, 1/\eta)$,  then 
for $\omega$-a.e.$(j)$ the isometry $g_j$ belongs to $S_{j,2\eta}(x_j, 1/2\eta)$.
Every element $g_j \in S_{j,2\eta}(x_j, 1/2\eta)$ belongs to  $\overline{G}_{j,\sigma_0}(x_j)$ since $2\eta \leq \sigma_0$, therefore it acts on $X_j= Y_j \times \mathbb{R}^{k}$ as  $g_j = (g_j',g_j'') $, where   $g'_j$ fixes $y_j$ (because of our choice of basepoints); moreover,   $g_j''$ is  a global translation of $\mathbb{R}^{k}$, since it moves the points of   $B_{\mathbb{R}^{k}} ({\bf v_j}, \frac{1}{2\eta})$ less than $2\eta$. 
An application of Lemma \ref{lemma-ultralimit-product}	says that also $X_\omega$ splits isometrically as $Y_\omega \times \mathbb{R}^{k}$, where $Y_\omega$ is the ultralimit of the  $(Y_j,y_j)$'s. Moreover $G_\omega$ preserves the product decomposition.
Every element of $g_\omega \in G_\omega^\circ$ can be written as a  product $u_\omega(1) \cdots u_\omega(n)$, with each $u_\omega(i) $ in $ S_{\eta}(x_\omega, 1/\eta)$.
As $u_\omega(i) = \omega$-$\lim  u_j(i)$ with $u_j(i)= (u_j(i)', u_j(i)'') \in S_{j,2\eta}(x_j, 1/2\eta)$, where  $u_j(i)'$ fixes $y_j$ and $u_j(i)''$ is a translation, it follows that also
$g_\omega$ can be written as $(g_\omega',g_\omega'') \in \text{Isom}(Y_\omega) \times \text{Isom}(\mathbb{R}^{k})$ where  $g_\omega' y_\omega = y_\omega$  and $g_\omega''$ is a global translation   (being the ultralimit of Euclidean translations). Let us call $\pi\colon G_\omega \to \text{Isom}(Y_\omega)$ the projection map. The group $\pi(G_\omega^\circ)$ is normal in $\pi(G_\omega)$. The set of fixed points $\text{Fix}(\pi(G_\omega^\circ))$ is closed, convex, non-empty and $\pi(G_\omega)$-invariant (since $\pi(G_\omega^\circ)$ is normal in $\pi(G_\omega ))$. Since $G_\omega$ is clearly $D_0$-cocompact, so it is $\pi(G_\omega)$. Then the  action of $\pi(G_\omega)$ on $Y_\omega$ is minimal (cp. \cite[Lemma 3.13]{CM09b}) which implies that $\text{Fix}(\pi(G_\omega^\circ)) = Y_\omega$, that is $\pi(G_\omega^\circ) = \lbrace \text{id}\rbrace$. 
We conclude that $G_\omega^\circ$ is a connected subgroup of 
$\lbrace \text{id}\rbrace \times \text{Transl}(\mathbb{R}^{k})$. By Lemma \ref{lemma-translations},  $\mathbb{R}^{k}$ splits isometrically as $\mathbb{R}^\ell \times \mathbb{R}^{k-\ell}$ and $ G_\omega^\circ$ can be identified with the subgroup of translations of the factor  $\mathbb{R}^\ell$,
for some $\ell \leq k$. 
Setting $X_\omega':=(Y_\omega \times \mathbb{R}^{k_\omega-\ell})$, this is still a  proper, geodesically complete,  $(P_0,r_0)$-packed, $\text{CAT}(0)$-space, and clearly $X_\omega' =G_\omega^\circ \backslash X_\omega$. 
Notice that, since $G_\omega^\circ$ is normal in $G_\omega$,  then  the splitting $X_\omega =X_\omega' \times \mathbb{R}^\ell$ is $G_\omega$-invariant. Let us now  show that $ G_\omega^\circ$ is non-trivial, hence $ \ell \geq 1$.
Actually, 	if $ G_\omega^\circ$ was trivial then $G_\omega$ would be  totally disconnected, implying $\text{sys}^\diamond(G_\omega, X_\omega) > 0$ by Theorem \ref{theo-characterization-td}.(ii). However, we are able to exhibit hyperbolic isometries of $G_\omega$ with arbitrarily small translation length, which will prove  that $G_\omega^\circ$ is non-trivial.  Indeed, fix any $\lambda > 0$ and an error $\xi > 0$. 
By the collapsing assumption we can find hyperbolic isometries $g_j \in G_j$ with $\ell(g_j) \leq \xi$, for $\omega$-a.e.$(j)$. By $D_0$-cocompactness, up to conjugating $g_j$   we can suppose  $g_j$ has an axis  at distance at most $D_0$ from $x_j$. Take a power $m_j$ of $g_j$ such that $\lambda < \ell(g_j^{m_j}) \leq \lambda + \xi$. Then, the sequence $(g_j^{m_j})$ is admissible and defines a hyperbolic element of $G_\omega$ whose translation length is between $\lambda$ and $\lambda + \xi$. By the arbitrariness of $\lambda$ and  $\xi$ we conclude. 
\vspace{2mm}

At this point we notice the following. Suppose to be in the standard setting of convergence $(X_j,G_j)\underset{\text{eq-pGH}}{\longrightarrow} (X_\infty, G_\infty)$. Theorem \ref{theo-noncollapsed} shows that if along a subsequence $\lbrace j_h\rbrace$ we have $\lim_{h \to +\infty} \textup{sys}^\diamond(G_{j_h}, X_{j_h}) > 0$ then $G_\infty$ is totally disconnected. The proof above shows that if along a subsequence $\lbrace j_h\rbrace$ we have $\lim_{h \to +\infty} \textup{sys}^\diamond(G_{j_h}, X_{j_h}) = 0$ then $G_\infty$ has a non trivial connected component of the identity. Therefore there cannot be a mixed behaviour: either along any subsequence the above limit is positive or along any subsequence it is zero. In particular the fact that $\omega$-$\lim \textup{sys}^\diamond(G_{j}, X_{j})$ is zero or not does not depend on the non-principal ultrafilter $\omega$. As a consequence all what we proved and all we are going to prove is true \emph{for every possible non-principal ultrafilter $\omega$} and not only for the one we fixed at the beginning of the proof. So from now on we fix an arbitrary non-principal ultrafilter $\omega$.
\vspace{1mm}

\noindent \textbf{Proof of (iii).} We proceed by steps. We first show 
\begin{equation}
	\label{eq-delta-ultra}
	\lim_{\rho \to 0} \omega \text{-}\lim \text{rk}(\overline{G}_{j,\rho}(x_j)) = \ell
\end{equation}
for a generic non-principal ultrafilter $\omega$. 
For $0 < \rho \leq \varepsilon_0$ the group $\overline{G}_{j,\rho}(x_j)$ is almost abelian, so it is meaningful to speak about is rank. For simplicity we set $\ell_\rho := \omega$-$\lim \text{rk}(\overline{G}_{j,\rho}(x_j))$. Observe that $\ell_\rho$ is a non-decreasing sequence of integers, depending on $\rho$. 
We start proving $\ell \leq \lim_{\rho \to 0} \ell_\rho$. For every $0 < \rho < 2\eta$, where $\eta$ is as in the proof of (i) and (ii), we choose $\frac{\pi}{2}$-linearly independent translations $g_{\omega,1},\ldots,g_{\omega,\ell} \in G_\omega^o = \lbrace \id \rbrace \times \text{Transl}(\mathbb{R}^\ell)$ of length $\frac{\rho}{2}$. Each isometry can be written as $g_{\omega,i} = \omega$-$\lim g_{j,i}$. For $\omega$-a.e.$(j)$ the following conditions are true: $d(g_{j,i}x_j,x_j)\leq \rho$ for every $i$ and the projections of the $g_{j,i}$'s on $\mathbb{R}^{k}$ are $\frac{\pi}{4}$-linearly independent translations, in particular $\text{rk}(\overline{G}_{j,\rho}(x_j)) \geq \ell$. The second condition holds because the projection of the  $g_{j,i}$'s on $\mathbb{R}^{k}$ are isometries moving every point of $B_{\mathbb{R}^{k}}({\bf v_j}, \frac{1}{2\eta})$ less than $2\eta$, for $\omega$-a.e.$(j)$, and they belong to a crystallographic group. Therefore by Lemma \ref{lemma-bieber} they must be translations. The condition on the uniform linear independence is clear. Since this happens for $\omega$-a.e.$(j)$ we deduce that $\ell_\rho \geq \ell$ for every $\rho$ small enough, then $\lim_{\rho \to 0} \ell_\rho \geq \ell$.\\
We now move to the other inequality. Let $\ell' = \lim_{\rho \to 0}\ell_\rho$. Then there exists $\rho' > 0$ such that $\ell_{\rho} = \ell'$ for all $0<\rho\leq \rho'$ because the sequence $\ell_\rho$ takes integer values.
We claim that for every $R, \lambda > 0$ we can construct isometries $g_{\omega,1}^{R,\lambda}, \ldots g_{\omega,\ell'}^{R,\lambda} \in G_\omega$ such that their projections on $Y_\omega$ fix pointwise $\overline{B}(y_\omega,R)$ and whose projections on $\mathbb{R}^{k}$ are translations of length $\lambda$ that are $\vartheta_0$-linearly independent, where $\vartheta_0 = \min_{k\leq n_0} \vartheta_k$, and $\vartheta_k$ is provided by Lemma \ref{lemma-LLL-basis}. Let us first observe why this concludes the proof of \eqref{eq-delta-ultra}. Once we have isometries $g_{\omega,1}^{R,\lambda}, \ldots g_{\omega,\ell'}^{R,\lambda} \in G_\omega$ as above we let $R$ go to $+\infty$. We get limit isometries $g_{\omega,1}^{\lambda}, \ldots g_{\omega,\ell'}^{\lambda} \in G_\omega$ with the property that their projections on $Y_\omega$ is the identity and their projections on $\mathbb{R}^{k}$ are still translations of length $\lambda$ that are $\vartheta_0$-linearly independent. We now consider the group $H_\omega = \lbrace g_\omega \in G_\omega \text{ whose projection on } Y_\omega \text{ is trivial}\rbrace$. $H_\omega$ is a closed subgroup of $G_\omega$ that is topologically isomorphic to a closed subgroup of $\text{Isom}(\mathbb{R}^{k})$. Therefore it is a Lie group. In particular its connected component of the identity is open. Observe that each sequence $g_{\omega,i}^\lambda$ belongs to $H_\omega$ and by construction it converges to $\id$ as $\lambda$ goes to $0$, so there must be $\lambda$ small enough such that $g_{\omega,i}^\lambda \in H_\omega^o$ for every $i=1,\ldots,\ell'$. But $H_\omega^o \subseteq G_\omega^o$, because $H_\omega^o$ is a connected set containing $\id$. Then $g_{\omega,i}^\lambda \in G_\omega^o$ for every $i=1,\ldots,\ell'$, and these translations of $\mathbb{R}^{k}$ are $\vartheta_0$-linearly independent. It follows that $\ell' \leq \ell$. \\
In order to prove \eqref{eq-delta-ultra} it remains to show the construction of the isometries $g_{\omega,1}^{R,\lambda}, \ldots g_{\omega,\ell'}^{R,\lambda}.$ We fix $R,\lambda > 0$ and we choose $\rho < \min \left\lbrace \rho', \frac{\lambda}{2^{n_0 - 1}\cdot J_0 \cdot N_0(\sigma_0,R)} \right\rbrace$, where $N_0(\sigma_0,R)$ is the constant provided by Proposition \ref{prop-bounded-order-sigma-controlled}. By Corollary \ref{cor-projection-td} the closure of the projection $\overline{G_{Y_j}}$ of $G_j$ on $Y_j$ is $\sigma_0$-controlled (see also the beginning of Section \ref{subsec-controlled}), while the projection of $\overline{G}_{j,\rho}(x_j)$ on $\mathbb{R}^{k}$ is contained in a crystallographic group of rank $k\leq n_0$. Moreover another application of Theorem \ref{theo-splitting-weak} with $\varepsilon = \rho$ provides another splitting $X_j = Y_j \times \mathbb{R}^{k - \ell'} \times \mathbb{R}^{\ell'}$, where the splitted factor has dimensions $\ell'$ because this is the rank of $G_{j,\tau}(x_j)$ for every $\tau < \rho$. This splitting is also compatible with the original splitting because of Proposition \ref{prop-commensurated-splitting}. In particular by Theorem \ref{theo-splitting-weak}.(v) we can find isometries $g_{j,1}, \ldots, g_{j,\ell'} \in \overline{G}_{j,\rho}(x_j)$ whose projections on $\mathbb{R}^{\ell'}$ are translations of length $<\rho$ and generate a lattice $\mathcal{L}_j$. Hence $\lambda(\mathcal{L}_j) < \rho$. By Lemma \ref{lemma-LLL-basis} we can replace the $g_{j,i}$'s with other isometries, that we still denote $g_{j,i}$, whose projections on $\mathbb{R}^{\ell'}$ still generate $\mathcal{L}_j$, have length at most $2^{n_0-1}\rho$ and are $\vartheta_0$-linearly independent. The projection of each $g_{j,i}$ on the factor $\mathbb{R}^{k - \ell'}$ is elliptic by Theorem \ref{theo-splitting-weak}.(iv). Remember that each $g_{j,i}$ lives in a crystallographic group of $\mathbb{R}^k$. Proposition \ref{prop-Bieberbach} implies that a power not greater than $J_0$ of each $g_{j,i}$ acts as a translation on $\mathbb{R}^k$. We replace each $g_{j,i}$ with this power, without changing the notation. From what we said above we deduce that the projection of $g_{j,i}$ on $\mathbb{R}^{k - \ell'}$ is the identity. On the other hand they project as translations of length at most $2^{n_0 - 1}\cdot J_0 \cdot \rho$ that are $\vartheta_0$-linearly independent on $\mathbb{R}^{\ell'}$.
Since the projection $\overline{G_{Y_j}}$ is $\sigma_0$-controlled, then the projection of $g_{j,i}^{N_0(\sigma_0,R)}$ on $Y_j$ is the identity on $\overline{B}_{Y_j}(y_j,R)$, by Proposition \ref{prop-bounded-order-sigma-controlled}. We set  $g_{j,i}^{N_0(\sigma_0,R)} =: g_{j,i}^{R,\lambda}$ and we claim they do the job. By construction their projections on $Y_j$ are the identity on $\overline{B}_{Y_j}(y_j,R)$. Moreover their projections on $\mathbb{R}^{k - \ell'}$ are the identity. Finally their projections on $\mathbb{R}^{\ell'}$ are translations of length at most $2^{n_0 - 1}\cdot J_0 \cdot N_0(\sigma_0,R) \cdot \rho = \lambda$ that are $\vartheta_0$-linearly independent. This concludes the proof of \eqref{eq-delta-ultra}.\\
By the arbitrariness of the non-principal ultrafilter $\omega$ and using \cite[Lemma 6.3]{Cav21ter} we can promote \eqref{eq-delta-ultra} to
\begin{equation}
	\label{eq-delta.limit}
	\lim_{\rho \to 0} \liminf_{j\to +\infty} \textup{rk}(\overline{G}_{j,\rho}(x_j)) = \lim_{\rho \to 0} \limsup_{j\to +\infty} \textup{rk}(\overline{G}_{j,\rho}(x_j)) = \ell.
\end{equation}
Since all the quantities involved are integer-valued we conclude that there exists $\rho' > 0$ such that for all $0<\rho \leq \rho'$ it holds
$$\liminf_{j\to +\infty} \textup{rk}(\overline{G}_{j,\rho}(x_j)) = \limsup_{j\to +\infty} \textup{rk}(\overline{G}_{j,\rho}(x_j)) = \ell.$$
In particular for all $0<\rho \leq \rho'$ we have: $\exists \lim_{j\to + \infty} \textup{rk}(\overline{G}_{j,\rho}(x_j)) = \ell.$
As a consequence for all $0<\rho \leq \rho'$ there exists $j_\rho$ such that if $j\geq j_\rho$ then $\textup{rk}(\overline{G}_{j,\rho}(x_j)) = \ell$. In order to conclude the proof of (iii) we need to extend this result to arbitrary $z_j \in X_j$. We choose $\rho^* = \sigma_{P_0,r_0,D_0}(\rho')$ and we claim it satisfies the thesis of (iii). 
Fix an arbitrary $0<\rho \leq \rho^*$. We apply Theorem \ref{theo-splitting-weak} with $\varepsilon = \rho'$. Theorem \ref{theo-splitting-weak}.(ii) applied to $\varepsilon = \rho'$ says that there exists $\varepsilon^* \in (\rho^*, \rho')$ such that
$$\text{rk}(\overline{G}_{j,\rho}(z_j)) \leq \text{rk}(\overline{G}_{j,\varepsilon^*}(z_j)) = \text{rk}(\overline{G}_{j,\varepsilon^*}(x_j)) \leq \text{rk}(\overline{G}_{j,\rho'}(x_j)) = \ell$$
for every $z_j \in X_j$, where the last equality is true for all $j\geq j_{\rho'}$. On the other hand, if we apply Theorem \ref{theo-splitting-weak}.(ii) to $\varepsilon = \rho$ we find $\varepsilon^* \in (\sigma, \rho)$, where $\sigma = \sigma_{P_0,r_0,D_0}(\rho)$ such that
$$\text{rk}(\overline{G}_{j,\rho}(z_j)) \geq \text{rk}(\overline{G}_{j,\varepsilon^*}(z_j)) = \text{rk}(\overline{G}_{j,\varepsilon^*}(x_j)) \geq \text{rk}(\overline{G}_{j,\sigma}(x_j)) = \ell,$$
where the last equality is true for all $j\geq j_\sigma$. Therefore for all $j\geq \max \lbrace j_{\rho'}, j_{\sigma} \rbrace$ we have that $\text{rk}(\overline{G}_{j,\rho}(z_j)) = \ell$ for all $z_j \in X_j$. This ends the proof of (iii).

\vspace{2mm}

Let $\rho^*$ be the quantity provided by (iii) and set $\sigma^*=\sigma_{P_0,r_0,D_0}(\rho^*)$. Theorem \ref{theo-splitting-weak} applied to $\varepsilon = \rho^*$ says that, for $j$ big enough, each $X_j$ splits as $X_j' \times \mathbb{R}^\ell$, where $X_j' = Y_j \times \mathbb{R}^{k-\ell}$ by Proposition \ref{prop-commensurated-splitting}. This is because the rank of the group $\overline{G}_{j,\varepsilon^*}(x_j)$ with $\varepsilon^* \in (\sigma^*, \rho^*)$ is exactly $\ell$ for all $j\geq \max\lbrace j_{\sigma^*}, j_{\rho^*} \rbrace$. It follows that $X_j'$ converges to $X_\infty'$. Corollary \ref{cor-projection-td} applied to $\varepsilon = \rho^*$ shows that we can suppose, up to change the basepoint $x_j = (x_j', {\bf v_j}) \in X_j' \times \mathbb{R}^\ell$, that $\overline{G}_{j, \rho^*}(x_j)$ preserves the slice $\lbrace x_j' \rbrace \times \mathbb{R}^\ell$ and that the closure $\overline{G_{X_j'}}$ of each $G_j$ on $X_j'$ has $\sigma^*$-separated orbit. 
Moreover, by Theorem \ref{theo-splitting-weak}.(v), the maximal lattice of the projection of $\overline{G}_{j,\rho^*}(x_j)$ on $\mathbb{R}^\ell$ is generated by elements of length at most $\rho^*/2\sqrt{n_0}$.
\vspace{2mm}

\noindent\textbf{Proof of (iv).} We want to show that $G_\infty$ is unimodular. Let $\rho^*$ be the positive number provided by (iii). Take bi-invariant Haar measures $\mu_j$ of $G_j$ normalized in such a way that $\mu_j(\overline{S}_{j,2\rho^*}(x_j)) = 1$. It is enough to verify  that conditions (i) and (ii) of Corollary \ref{cor-unimodularity-limit} are satisfied. 
Every isometry $g \in G_j$ will be written as $(g', g'') \in \text{Isom}(X_j') \times \text{Isom}(\mathbb{R}^\ell)$.
As we noticed above, the orbit $\overline{G_{X_j'}}x_j'$ is $\sigma^*$-separated. In particular there are at most $\text{Pack}(R,\frac{\sigma^*}{2}) =: P$ points of this orbit inside $\overline{B}_{X_j'}(x_j',R)$. Notice that $P$ does not depend on $j$ by Proposition \ref{prop-packing}. Let us take $g_1,\ldots,g_m \in \overline{S}_{j,R}(x_j)$ such that for every $g\in \overline{S}_{j,R}(x_j)$ there exists $i \in \lbrace 1,\ldots,m\rbrace$ such that $g'(x_j')=g_i'(x_j')$. Therefore for every $g\in \overline{S}_{j,R}(x_j)$ there exists $i\in \lbrace 1,\ldots,m\rbrace$ such that $(g_i^{-1}g)'x_j'=x_j'$ and $d_{\mathbb{R}^\ell}((g_i^{-1}g)'' {\bf v_j}, {\bf v_j}) \leq 2R$. Let us denote by $S_j'$ the set of isometries of $G_j$ whose projection on $X_j'$ fixes $x_j'$ and whose projection on ${\mathbb{R}^\ell}$ moves ${\bf v_j}$ by at most $2R$. We just showed that 
$\overline{S}_{j,R}(x_j) \subseteq \bigcup_{i=1}^m g_i S_j',$
so $\mu_j(\overline{S}_{j,R}(x_j)) \leq P \cdot \mu_j(S_j')$. To conclude the thesis we just need to bound $\mu_j(S_j')$ from above by a quantity that does not depend on $j$. The projection of
$\overline{G}_{j,\rho^*}(x_j)$ on $\mathbb{R}^\ell$ is a crystallographic group whose maximal lattices $\mathcal{L}_{j,\rho^*}$ ia generated by elements of length at most $\rho^*/2\sqrt{n_0}$. We first apply Lemma \ref{lemma-sublattice-controlled} to find a sublattice $\mathcal{L}'_{j,\rho^*} < \mathcal{L}_{j,\rho^*}$ with $\lambda(\mathcal{L}_{j,\rho^*}) = \lambda(\mathcal{L}'_{j,\rho^*}) \leq 2\tau(\mathcal{L}'_{j,\rho^*}) \leq \rho^*/\sqrt{n_0}$. Secondly, by replacing each generator of $\mathcal{L}'_{j,\rho^*}$ by a fixed power, depending on $j$, we can suppose that 
$$\frac{\rho^*}{2\sqrt{n_0}} \leq \tau(\mathcal{L}'_{j,\rho^*}) \leq \lambda(\mathcal{L}'_{j,\rho^*}) \leq \frac{\rho^*}{\sqrt{n_0}}.$$
Let $g_j = (g'_j,g''_j) \in S_j'$. First we find an element $h_j = (h_j',h_j'') \in \overline{G}_{j,\rho^*}(x_j)$ such that $h_j'' \in \mathcal{L}'_{j,\rho^*}$ and $d_{\mathbb{R}^\ell}(g_j'' {\bf v_j}, h_j'' {\bf v_j}) \leq 2\rho^*$, because $\mathcal{L}'_{j,\rho^*}$ is $2\rho^*$-cocompact by \eqref{eq-lattice-relation}. In particular $h_j^{-1}g_j \in \overline{S}_{j,2\rho^*}(x_j)$ because $h_j'^{-1}g_j' x_j' = x_j'$.
Moreover $\Vert h_j'' \Vert \leq 2R + 1$, so by Lemma \ref{lemma-lattice-svarc-milnor} applied to $\mathcal{L}'_{j,\rho^*}$ we can find a word $w_j = (w_j',w_j'')$ of length at most $\frac{(2R + 1)2\sqrt{n_0}}{\rho^*} + 1$ in $\overline{S}_{j,2\rho^*}(x_j)$ such that $w_j'^{-1}h_j'$ fixes $x_j'$ and $w_j'' = h_j''$.
In particular $w_j^{-1}h_j^{-1} \in \overline{S}_{j,2\rho^*}(x_j)$. Combining all together we conclude that $g_j$ can be written as a word of length at most $\frac{(2R + 1)2\sqrt{n_0}}{\rho^*} + 3 =: p$ in the alphabet $\overline{S}_{j,2\rho^*}(x_j)$. Observe that $p$ does not depend on $j$. In other words
$$S_j' \subseteq (\overline{S}_{j,2\rho^*}(x_j))^{p}.$$
Corollary \ref{cor-measure-powers} ensures that 
$$\mu_j(S_j') \leq M_0^{p-1} \mu_j(\overline{S}_{j,2\rho^*}(x_j)) = M_0^{p-1},$$
which is independent of $j$. \\
We now move to the proof of condition (ii) of Corollary \ref{cor-unimodularity-limit}.
First we show
\begin{equation}
	\label{eq-liminf-delta}
	\inf_j \mu_j(\overline{S}_{j, \rho}(x_j)) > 0
\end{equation}
for every $0<\rho \leq \rho^*$.
For fixed $\rho$ we restrict the attention to the indices $j\geq j_\rho$, so that $\text{rk}(\overline{S}_{j, \rho}(x_j)) = \text{rk}(\overline{S}_{j, \rho^*}(x_j)) = \ell$. We can apply Theorem \ref{theo-splitting-weak}.(iv) with $\varepsilon = \rho$ to find a lattice $\mathcal{L}_{j,\rho}$ of the projection of $\overline{G}_{j,\rho}(x_j)$ on $\mathbb{R}^\ell$ generated by elements of length at most $\rho/2\sqrt{n_0}$. 
Using Lemma \ref{lemma-sublattice-controlled} as above we can find a sublattice $\mathcal{L}'_{j,\rho}$ of $\mathcal{L}_{j,\rho}$ satisfying
$$\frac{\rho}{4\sqrt{n_0}} \leq \tau(\mathcal{L}'_{j,\rho}) \leq \lambda(\mathcal{L}'_{j,\rho}) \leq \frac{\rho}{2\sqrt{n_0}}.$$
The lattice $\mathcal{L}'_{j,\rho}$ is $\rho$-cocompact because of \eqref{eq-lattice-relation}. Repeating word by word the same argument as before 
we can write each element of $\overline{S}_{j,2\rho^*}(x_j)$ as a word of length at most $8\sqrt{n_0}(\frac{\rho^*}{\rho} + 3) =: p$ in the alphabet $\overline{S}_{j,\rho}(x_j)$.
In other words
$$\overline{S}_{j, 2\rho^*}(x_j) \subseteq \overline{S}_{j,\rho}(x_j)^p.$$
Again by Corollary \ref{cor-measure-powers} we have
$$1 = \mu_j(\overline{S}_{j, 2\rho^*}(x_j)) \leq M_0^{p-1} \mu_j(\overline{S}_{j, \rho}(x_j)),$$
where $p$ is independent of $j$. This shows \eqref{eq-liminf-delta}. \\
We now fix $r >0$. We choose $\rho = \frac{r}{J_0\cdot N_0(\sigma^*,1/r)}$. Let $w$ be any word in the alphabet $\overline{S}_{j, \rho}(x_j)$ of length $J_0 \cdot N_0(\sigma^*,1/r)$, that we write as $w=(w', w'')$. Since the projection of $G_j$ on $X_j'$ is $\sigma^*$-controlled we deduce that $w'$ acts as the identity on $\overline{B}_{X_j'}(x_j,1/r)$ by Proposition \ref{prop-bounded-order-sigma-controlled}. On the other hand $w''$ is a translation of $\mathbb{R}^\ell$ by Proposition \ref{prop-Bieberbach}. So $w$ acts on $X_j'$ fixing the ball $\overline{B}_{X_j'}(x_j, 1/\varepsilon)$ and on $\mathbb{R}^\ell$ as a translation of length at most $\rho \cdot J_0\cdot N_0(\sigma^*, 1/r) \leq r$. In particular $w$ belongs to $\overline{S}_r(x_j, 1/r)$. This shows that $\overline{S}_{j,\rho}(x_j)$ can be covered with at most $J_0\cdot N_0(\sigma^*,1/r)$ translated of $\overline{S}_r(x_j, 1/r)$.
Therefore $\mu_j(\overline{S}_r(x_j, 1/r)) \geq \frac{1}{J_0\cdot N_0(\sigma^*,1/r)}\cdot  \mu_j(\overline{S}_{j,\rho}(x_j))$ for every $j$ big enough. This implies the thesis by \eqref{eq-liminf-delta}. In conclusion $G_\infty$ is unimodular.
\vspace{2mm}

\noindent{\textbf{Proof of (v)}.} We first show that $G_\infty'$ is closed, where $G_\infty'$ is the projection of $G_\infty$ on $X_\infty'$. Let $g_i'$ be a sequence of elements of $G_\infty'$ and suppose they converge to $g'$. By applying elements of $G_\infty^\circ = \lbrace \id \rbrace \times \text{Transl}(\mathbb{R}^\ell)$ we can suppose to have elements $g_i = (g_i', g_i'') \in G_\infty$, with $g_i'' {\bf v_\infty} = {\bf v_\infty}$. In particular $g_i''$ converges, up to subsequence, to some $g''$. Therefore $g_i$ converges to $g = (g',g'') \in G_\infty$, since $G_\infty$ is closed. Therefore $g' \in G_\infty'$. \\
Let us show now that $G_\infty'$ is unimodular. Let $\mu_\infty$ be a bi-invariant Haar measure of $G_\infty$, which exists by (iv). Let us denote by $\overline{S}_1({\bf v_\infty})$ the subset of the projection of $G_\infty$ on $\text{Isom}(\mathbb{R}^\ell)$ made of the isometries that move ${\bf v_\infty}$ by at most $1$. Let $A\subseteq G_\infty'$ be any set and consider the set 
$$A^* := A \times \overline{S}_1({\bf v_\infty}).$$
We claim that the formula $\nu(A):=\mu_\infty(A^*)$ defines a bi-invariant Haar measure of $G_\infty'$. Clearly $\nu(\emptyset)=0$, while if $A_i \subseteq G_\infty'$ are pairwise disjoint then also $A_i^* \subseteq G_\infty$ are pairwise disjoint. Therefore
$$\nu\left(\bigcup_i A_i\right) = \mu\left(\left(\bigcup_i A_i\right)^*\right) = \mu\left(\bigcup_i A_i^*\right) = \sum_i \mu\left(A_i^*\right) = \sum_i \nu\left(A_i\right).$$
This shows that $\nu$ is actually a measure on the Borel subsets of $G_\infty'$, since if $A$ is Borel then $A^*$ is Borel too. The next step is to show that it is finite on compact sets and positive on open sets. Observe that if $A$ is compact then $A^*$ is compact, so $\nu(A) = \mu(A^*) < +\infty$. On the other hand if $A$ is open then $A^*$ has non-empty interior, so $\nu(A) = \mu(A^*) > 0$. It remains to show it is left and right invariant. Every element $g' \in G_\infty'$ comes from some $g=(g',g'') \in G_\infty$. After a composition with an element of $G_\infty^\circ = \lbrace \id \rbrace \times \text{Transl}(\mathbb{R}^\ell)$ we can suppose that $g''$ fixes ${\bf v_\infty}$. In particular  $g'' \overline{S}_1({\bf v_\infty}) = \overline{S}_1({\bf v_\infty}) = \overline{S}_1({\bf v_\infty}) g''$.
Therefore we have
$$(g' A)^* = g'A \times \overline{S}_1({\bf v_\infty}) = g'A \times g'' \overline{S}_1({\bf v_\infty}) = (g',g'') (A\times \overline{S}_1({\bf v_\infty})) = (g',g'') A^*.$$
This implies 
$$\nu(g'A)=\mu((g'A)^*) = \mu ((g',g'') A^*) = \mu (A^*) = \nu (A).$$
In the same way we can prove the right invariance of $\nu$, so $G_\infty'$ is unimodular.\\
The last step is to show that it is totally disconnected. This can be done observing that, with the notation of the previous steps, each $\overline{G_{X_j'}}$ is $\sigma^*$-controlled and using Proposition \ref{prop-sigma-controlled-limit}. There is also another way to get the conclusion. Indeed let us consider the natural action of the totally disconnected group $G_\infty / G_\infty^o$ on $X_\infty'$. By construction the image of such action is $G_\infty'$. The kernel is
$$\lbrace g = (g',g'') \in G_\infty / G_\infty^o \text{ s.t. } g' = \id \rbrace.$$
Any class in the set above has a representative of the type $(\id, g'')$, with $g'' {\bf v_\infty} = {\bf v_\infty}$. Therefore the kernel of the action is a compact, normal subgroup of $G_\infty / G_\infty^o$. Then the thesis follows since the group $G_\infty'$ is the quotient of the locally compact, totally disconnected group $G_\infty / G_\infty^o$ by a closed, normal subgroup. Even if it is not necessary, observe that the fact that the kernel of the action is compact shows also that the group $G_\infty / G_\infty^o$ is unimodular. The fact that $G_\infty'$ is $D_0$-cocompact is obvious.
\vspace{2mm}

\noindent \textbf{Proof of (vi).} The quotients  $M_j=G_j \backslash X_j$  converge to $M_\infty = G_\infty \backslash X_\infty$, by Lemma \ref{lemma-ultralimit-quotient}. Moreover $$G_\infty \backslash X_\infty = (G_\infty /G_\infty^\circ) \backslash (G_\infty^\circ \backslash X_\infty) =  G_\infty'   \backslash X_\infty'.$$ 
\end{proof}

\subsection{Characterization of collapsing}
\label{sub-characterization}${}$

\noindent If  $X \in \textup{CAT}_0^\textup{td,u}(P_0,r_0,D_0)$  then dim$(X) \leq n_0=P_0/2$, by  Proposition \ref{prop-packing}(ii). Moreover, if  $ (X_j,x_j) \underset{\textup{pGH}}{\longrightarrow} (X_\infty,x_\infty)$ then,  by \cite[Theorem 6.5]{CavS20},
$$\text{dim}(X_\infty) \leq \liminf_{j \rightarrow +\infty} \text{dim}(X_j).$$
The following theorem precisely relates the collapsing (as defined in \ref{defsetting}) in the standard setting of convergence  to the dimension of the limit quotients, and is a direct consequence of the convergence Theorems \ref{theo-noncollapsed} \&  \ref{theo-collapsed}. The proof is the same of \cite[Theorem 7.15]{CS23}.

\begin{theo}[Characterization of collapsing]
\label{theo-collapsing-characterization}  ${}$\\
In  the standard setting of convergence  of 
$(X_j, G_j) \underset{\textup{eq-pGH}}{\longrightarrow} (X_\infty, G_\infty)$,
let  \linebreak $M_j = G_j \backslash X_j$ and  $M_\infty = G_\infty \backslash X_\infty$ be the quotient spaces. Then:
\begin{itemize}
\item[(i)] the sequence is non-collapsing iff $\textup{TD}(M_\infty) = \lim_{j\to +\infty} \textup{TD}(M_j)$;
\item[(ii)] the sequence is collapsing iff $\textup{TD}(M_\infty) < \lim_{j\to +\infty} \textup{TD}(M_j)$.
\end{itemize}
Moreover, in the above characterization,  the topological dimension \textup{TD} can be replaced by the Hausdorff dimension \textup{HD}.
\end{theo}

Notice that since we proved that   $\lim_{j\to +\infty}\text{TD}(M_j)$ exists,  Theorem \ref{theo-collapsing-characterization} excludes to have  sequences $(X_j,G_j)$ converging   with mixed behaviour  (that is, such that along some subsequence the convergence is collapsed, and along other subsequences it is non-collapsed). We have already noticed this fact during the proof of Theorem \ref{theo-collapsed}. Finally we conclude with the
\begin{proof}[Proof of Theorem \ref{theo-intro-closure}]
	Let $M_j \in \mathcal{O}\text{-CAT}_0^{\text{td,u}}(P_0,r_0,D_0)$ and suppose it converges to some $M_\infty$ in the Gromov-Hausdorff sense. By definition there are isometric actions $(X_j, G_j) \in \text{CAT}_0^{\text{td,u}}(P_0,r_0,D_0)$ such that $G_j\backslash X_j = M_j$. Up to pass to a subsequence we can suppose that $(X_j,G_j) \underset{\text{eq-pGH}}{\longrightarrow} (X_\infty, G_\infty)$ by Proposition \ref{lemma-GH-compactness-packing}. We are in the standard setting of convergence. Observe that by Lemma \ref{lemma-ultralimit-quotient} the space $M_\infty$ is isometric to the quotient $G_\infty \backslash X_\infty$. If the sequence $(X_j,G_j)$ is non-collapsed then we get $M_\infty \in \mathcal{O}\text{-CAT}_0^{\text{td,u}}(P_0,r_0,D_0)$ by Theorem \ref{theo-noncollapsed}. If it is collapsed then $M_\infty \in \mathcal{O}\text{-CAT}_0^{\text{td,u}}(P_0,r_0,D_0)$ by Theorem \ref{theo-collapsed}. By the remark above these two cases cover all possible cases, so we conclude the proof of the closure of $\mathcal{O}\text{-CAT}_0^{\text{td,u}}(P_0,r_0,D_0)$. The compactness is a consequence of Proposition \ref{lemma-GH-compactness-packing} and Lemma \ref{lemma-ultralimit-quotient}.
\end{proof}

\bibliographystyle{alpha}
\bibliography{collapsing}

\end{document}